%% file: FractalScreens_arXiv_v3.tex
\documentclass[final]{siamltex}

\usepackage{amsmath}
\usepackage{mathtools}
\usepackage{amsfonts}
\usepackage{accents}
\usepackage[mathscr]{eucal}
\usepackage{amssymb}
\usepackage{enumerate}
\usepackage{cases}
\usepackage{multirow}
\usepackage{graphicx}
\usepackage{pgf,tikz}
\usepackage{color}
\usepackage[bookmarks=true, bookmarksopen=true, bookmarksopenlevel=4]{hyperref}
\definecolor{myblue}{rgb}{0,0,0.6}     %
\hypersetup{pdftitle={Well-posed PDE and integral equation formulations for scattering by fractal screens},
            pdfauthor={Chandler-Wilde, Hewett},
     colorlinks=true, linkcolor=myblue,  citecolor=myblue, filecolor=myblue,   urlcolor=myblue,  }
\definecolor{dhcol}{rgb}{0,0.5,0}

\definecolor{sccol}{rgb}{0,0,0.5}

\newcommand{\SDw}{$\sS\sD$-{\rm w}}
\newcommand{\SNw}{$\sS\sN$-{\rm w}}
\newcommand{\SDcl}{$\sS\sD$-{\rm cl}}
\newcommand{\SNcl}{$\sS\sN$-{\rm cl}}
\newcommand{\Dst}{$\sD$-{\rm st}}
\newcommand{\Nst}{$\sN$-{\rm st}}

\newcommand{\SDws}{$\sS\sD$-{\rm w} }
\newcommand{\SNws}{$\sS\sN$-{\rm w} }
\newcommand{\SDcls}{$\sS\sD$-{\rm cl} }
\newcommand{\SNcls}{$\sS\sN$-{\rm cl} }
\newcommand{\Dsts}{$\sD$-{\rm st} }
\newcommand{\Nsts}{$\sN$-{\rm st} }

\graphicspath{{Figs/}}
\begin{document}
\input{macros}
\newcommand{\sO}{\mathsf{O}}
\newcommand{\sC}{\mathsf{C}}
\newcommand{\mS}{\Gamma}
\newcommand{\omS}{{\overline{\mS}}}
\newcommand{\sumpm}[1]{\{\!\!\{#1\}\!\!\}}
\title{Well-posed PDE and integral equation formulations for scattering by fractal screens}
\author{Simon N.\ Chandler-Wilde\footnotemark[1] and David P.\ Hewett\footnotemark[2]}
\renewcommand{\thefootnote}{\fnsymbol{footnote}}
\footnotetext[1]{Department of Mathematics and Statistics, University of Reading, Whiteknights PO Box 220, Reading RG6 6AX, UK. Email: \texttt{s.n.chandler-wilde@reading.ac.uk}. This work was supported by EPSRC grant EP/F067798/1.}
\footnotetext[2]{Department of Mathematics, University College London, Gower Street, London
WC1E 6BT, UK. Email: \texttt{d.hewett@ucl.ac.uk}. This work was supported by EPSRC grant EP/P511262/1.}

\maketitle
\renewcommand{\thefootnote}{\arabic{footnote}}
\begin{abstract}
We consider time-harmonic acoustic scattering by planar sound-soft (Dirichlet) and sound-hard (Neumann) screens embedded in $\R^n$ for $n=2$ or $3$. In contrast to previous studies in which the screen is assumed to be a bounded Lipschitz (or smoother) relatively open subset of the plane, we consider screens occupying arbitrary bounded subsets. Thus our study includes cases where the screen is a relatively open set with a fractal boundary, and cases where the screen is fractal with empty interior. We elucidate for which screen geometries the classical formulations of screen scattering are well-posed, showing that the classical formulation for sound-hard scattering is not well-posed if the screen boundary has Hausdorff dimension greater than $n-2$. Our main contribution is to propose novel well-posed boundary integral equation and boundary value problem formulations, valid for arbitrary bounded screens. In fact, we show that for sufficiently irregular screens there exist whole families of well-posed formulations, with infinitely many distinct solutions, the distinct formulations distinguished by the sense in which the boundary conditions are understood. To select the physically correct solution we propose limiting geometry principles, taking the limit of solutions for a sequence of more regular screens converging to the screen we are interested in, this a natural procedure for those fractal screens for which there exists a standard sequence of prefractal approximations. We present examples exhibiting interesting physical behaviours, including penetration of waves through screens with ``holes'' in them, where the ``holes'' have no interior points, so that the screen and its closure scatter differently. Our results depend on subtle and interesting properties of fractional Sobolev spaces on non-Lipschitz sets.
\end{abstract}


\section{Introduction}
\label{sec:Intro}
This paper is concerned with the mathematical analysis of time-harmonic acoustic scattering problems modelled by the Helmholtz equation
\begin{equation} \label {eq:he}
\Delta u + k^2 u = 0,
\end{equation}
(or its inhomogeneous variant \eqref{eq:he_inhom} below) where $k>0$ is the {\em wavenumber}.
Our focus is on scattering by thin planar screens in $\R^n$ ($n=2$ or $3$), so that the domain in which \rf{eq:he} holds is $D:=\R^n\setminus \omS$, where $\mS$, the {\em screen}, is a bounded subset of the hyperplane $\Gamma_\infty:=\{\bx=(x_1,...,x_n)\in \R^n:x_n=0\}$, and the compact set $\omS$ is its closure.  As usual, the complex-valued function $u$ is to be interpreted physically as either the (total) complex acoustic pressure field or the velocity potential, and we write $u$ as $u=u^i+u^s$, where $u^i$ is a given incident field %
and $u^s:= u-u^i$ is the {\em scattered field} which is to be determined and is assumed to satisfy \eqref{eq:he} and the standard Sommerfeld radiation condition (equation \eqref{src} below). We suppose that either {\em sound-soft} or {\em sound-hard} boundary conditions hold, respectively that either
\begin{equation} \label{eq:bcs}
u = 0 \mbox{ or } \frac{\partial u}{\partial \bn} = 0
\end{equation}
on the screen in some appropriate sense, where $\bn$ is the unit normal pointing in the $x_n$ direction.

These are long-standing scattering problems, their mathematical study dating back at least to \cite[p.\ 139]{Rayleigh}, and it is well-known (e.g., \cite{NeittaanmakiRoach:87,wilcox:75}, and see \S\ref{subsec:bvps} for more detail) that, for arbitrary bounded $\mS\subset\Gamma_\infty$, these problems are well-posed (and the solutions depend only on the closure $\omS$) if the boundary conditions are understood in the standard weak senses that $u\in W^{1,\mathrm{loc}}_0(D)$ in the sound-soft case, that $u\in W^{1,\mathrm{loc}}(D)$ and
\begin{equation} \label{eq:wnc}
\int_D \left( v\Delta u + \nabla v\cdot \nabla u\right)\, \rd x = 0, \quad \mbox{for all } v\in W^{1,{\mathrm{comp}}}(D),
\end{equation}
in the sound-hard case.  We spell out these weak formulations more fully in Definitions \ref{def:sSsDw} and \ref{def:sSsNw} below using standard Sobolev space notations defined in \S\ref{sec:SobolevSpaces}.

In the well-studied case where $\mS$ is a relatively open\footnote{
For brevity we shall henceforth omit the word ``relatively'' when discussing relatively open subsets of $\Gamma_\infty$.
}
subset of $\Gamma_\infty$ that is Lipschitz or smoother, the alternative, classical formulation, dating back to the late 40s \cite{Meixner:49,Bouwkamp:54}, imposes the boundary conditions \eqref{eq:bcs} in a classical sense, and additionally  imposes ``edge conditions'' requiring locally finite energy, that $u$ and $\nabla u$ are square integrable in some neighbourhood of $\partial \Gamma$ (see Definition \ref{def:sSsD-cl} below). Equivalently, one can formulate boundary value problems (BVPs) for $u^s$ in a Sobolev space setting, seeking $u^s\in W^{1,\mathrm{loc}}(D)$ satisfying \eqref{eq:he} and the radiation condition, and imposing the boundary conditions \eqref{eq:bcs} in a trace sense, requiring that the Dirichlet or Neumann traces on $\Gamma_\infty$, $\gamma^\pm u^s$ and $\partial_\bn^\pm u^s$, satisfy $(\gamma^\pm u^s)|_\Gamma = g_\sD\in H^{1/2}(\Gamma)$ in the sound-soft case, $(\partial_\bn^\pm u^s)|_\Gamma = g_\sN\in H^{-1/2}(\Gamma)$ in the sound-hard case, where $g_\sD := -(\gamma^\pm u^i)|_\Gamma$ and $g_\sN := -(\partial_\bn^\pm u^i)|_\Gamma$ (see, e.g., \cite{Ste:87} and Definition \ref{def:sD-st} for details). Finally, it is well-known \cite{Ste:87,WeSt:90,Ha-Du:90,Ha-Du:92,Co:04} that for Lipschitz $\Gamma$ one can reformulate these BVPs as the boundary integral equations (BIEs)
\begin{equation} \label{eq:bie}
S[\partial_\bn u^s] = -g_\sD, \quad T[u] = g_\sN,
\end{equation}
in the sound-soft and sound-hard cases, respectively. In these equations the unknowns are the jumps across the screen in $u$ and its normal derivative, $[u]\in \tH^{1/2}(\mS)$ and $[\partial_\bn u^s]\in \tH^{-1/2}(\mS)$, and the isomorphisms $S:\tH^{-1/2}(\mS)\to H^{1/2}(\mS)$ and $T:\tH^{1/2}(\mS)\to H^{-1/2}(\mS)$ are the (acoustic) single-layer and hypersingular boundary integral operators (BIOs), respectively. Here $\tH^s(\Gamma)\subset H^s(\Gamma_\infty)$, for $s\in \R$, denotes the closure in $H^s(\Gamma_\infty)$ of $C_0^\infty(\Gamma)$. As is pointed out in \cite{CoercScreen2}, the BIEs \eqref{eq:bie} are well-posed ($S$ and $T$ are isomorphisms) for arbitrary open $\mS$.

The scattering problems we study may be long-standing, but there remain many open questions concerning the correct choice, well-posedness, and equivalence of mathematical formulations when $\Gamma$ is not a Lipschitz open set, and many interesting features arise in this case. Amongst the new results in this paper we will see that:
\begin{enumerate}
\item if the screen is sufficiently irregular, uniqueness fails for the classical and Sobolev space BVP formulations;
\item the BIEs \eqref{eq:bie} are well-posed for every open $\Gamma$ but their solutions differ from those of the weak BVP if $\tH^{\pm 1/2}(\Gamma)\subsetneqq H^{\pm 1/2}_{\overline \Gamma}$ (here $H^s_F$, for closed $F\subset \Gamma_\infty$, is the set of those $\phi\in H^s(\Gamma_\infty)$ supported in $F$);
\item whenever $\tH^{\pm 1/2}(\Gamma^\circ)\subsetneqq H^{\pm 1/2}_{\overline \Gamma}$ (where $\Gamma^\circ$ is the relative interior in $\Gamma_\infty$ of $\Gamma$) there exists, in fact, an infinite family of well-posed BVP and equivalent BIE formulations, which have infinitely many distinct solutions for generic boundary data, and distinct solution choices are appropriate in different physical limits.
\end{enumerate}

A main aim of this paper is to derive the correct mathematical formulations for screens that are fractal or have fractal boundary, and to understand the convergence of solutions as sequences of prefractals (as in Figures \ref{fig:Sierpinksi}, \ref{fig:CantorDust}, and \ref{fig:koch}) converge to a fractal limit. One motivation for such a study is that
fractal screen problems are of relevance to a number of areas of current engineering research, for example in the design of antennas for electromagnetic wave transmission/reception (see e.g.\ \cite{PBaRomPouCar:98,SriRanKri:11}), and in piezoelectric ultrasound transducers (see e.g.\ \cite{MuWa:11,Mu:11}).
The attraction of using fractal structures (in practice, high order prefractal approximations to fractal structures) for such applications is in their potential for wideband performance. Indeed, a key property of fractals is that they possess structure on every length scale, and the idea is to exploit this to achieve efficient transmission/reception of waves over a broad range of frequencies simultaneously. (In this direction, much earlier, Berry \cite{Berry79} urged the study of waves diffracted by fractal structures (termed {\em diffractals}), as a situation where distinctive high frequency asymptotics can be expected.)
Although this is a mature engineering technology (at least for electromagnetic antennas), as far as we are aware no analytical framework is currently available for such problems. %
Understanding well-posedness and convergence of prefractal solutions for the simpler acoustic case considered in the current paper can be regarded as a significant first step towards these applications. %

 Regarding related mathematical work, there is a substantial literature studying trace spaces on fractal boundaries (see \cite{JoWa84} and the references therein). This is an important ingredient in formulating and analysing BVPs, and this theory has been applied to the study of elliptic PDEs in domains with fractal boundaries for example in \cite{JoWa97}. However, a key assumption in these results is that the boundary satisfies a so-called {\em Markov inequality} \cite{JoWa84,JoWa97}. This assumption on the boundary $\partial \Omega$ of a domain $\Omega \subset \R^n$ (in the language of numerical analysis, a requirement that a type of inverse estimate holds for polynomials defined on $\partial \Omega$) requires a certain isotropy of $\partial \Omega$, and  does not hold if $\partial \Omega$ is an $(n-1)$-dimensional manifold (or part of such a manifold as in this paper). Further, this theory applies specifically to the case where $\partial \Omega$ is a $d$-set in the sense of \cite{JoWa84,JoWa97}. (Roughly speaking this is a requirement that $\partial \Omega$ has finite $d$-dimensional Hausdorff measure, uniformly across $\partial \Omega$. An example to which this theory applies is the boundary of the Koch snowflake (Figure \ref{fig:koch}), a $d$-set with $d=\log_3 4$.)

Similar constraints on $\partial \Omega$ apply to other studies of BVPs in domains with fractal boundary, for example recent work on regularity of PDE solutions in Koch snowflake domains and their prefractal approximations \cite{CapitanelliVivaldi15} and, closer to the specific problems we tackle in this paper, work on high frequency scattering by fractals \cite{SleemanHua92,SleemanHua01}. In these latter papers Sleeman and Hua address what we term above the weak scattering problems (with the Dirichlet condition understood as $u\in W^{1,\mathrm{loc}}_0(D)$, the Neumann condition as \eqref{eq:wnc}) in the case when the domain $D:=\R^n\setminus\overline{\Omega}$ and $\Omega$ is a bounded open set whose boundary has fractal dimension in the range $(n-1,n)$. They study the high frequency asymptotics of the so-called {\em scattering phase}, this closely related to the asymptotics of the eigenvalue counting function for the interior set $\Omega$, which has been widely studied both theoretically and computationally \cite{Lapidus91,SleemanHua95,LevitinVassiliev96,Neuberger06}, following the 1980 conjecture of Berry \cite{Berry79a,Berry80} on the dependence of these asymptotics on the fractal dimension of $\partial \Omega$.

  The particular case where $\Omega\subset \R^2$ is a so-called ramified domain with self-similar fractal boundary $\partial \Omega$ of fractal dimension greater than one has been extensively studied by Achdou and collaborators, including work on the formulation and numerical analysis of (interior) Poisson and Helmholtz problems with Dirichlet and Neumann boundary conditions \cite{Achdou06a,Achdou06,Achdou07}, numerical studies of the Berry conjecture \cite{Achdou07}, characterisation of trace spaces \cite{Achdou10}, study of convergence of prefractal to fractal problems \cite{Achdou06,Achdou06a,Achdou16} (as in our \S\ref{sec:domdep}), and  study of transmission problems \cite{Achdou16} (see also \cite{Bagnerini13}).

 This current paper deviates from the above-cited literature in a number of significant respects. Firstly, the above works all treat boundaries with fractal dimension in the range $(n-1,n)$. In contrast, the boundary of our domain $D$ is $\overline{\Gamma}$, a bounded subset of an $(n-1)$-dimensional manifold. If fractal, $\Gamma$ has Hausdorff dimension in the range $(0,n-1)$. Secondly, for us a deep study of trace spaces seems superfluous: our analysis only requires standard traces from the upper and lower half-spaces onto the plane $\Gamma_\infty$ containing the screen $\Gamma$. Thirdly, in this current paper a major theme is to explore the multiplicity of distinct formulations and solutions, and to point out the physical relevance of distinct solutions as limits of problems on more regular domains. Nothing of this flavour arises in the above literature. Finally, we note that, in contrast to the above studies, a significant focus in this paper is on BIEs on rough domains (including on fractals), and our proof of well-posedness of our novel BVP formulations is via analysis of their BIE equivalents.

Another motivation for this study %
is simply that BIE formulations are powerful and well-studied  for problems of acoustic scattering by screens, both theoretically (e.g.\ \cite{StWe84,Ste:87,WeSt:90,Ha-Du:90,Ha-Du:92,CoercScreen2}) and as a computational tool in applications (e.g.\ \cite{DLaWroPowMan:98,DavisChew08}), so that it is of intrinsic interest to extend this methodology to deal with general, not just Lipschitz or smoother, screens.
Our analysis of BIEs in \S\ref{subsec:bies} follows the spirit of previous studies (e.g.\ \cite{StWe84,Ste:87,WeSt:90}), in which  to determine solvability one needs to understand the BIOs in \eqref{eq:bie} as mappings between fractional Sobolev spaces defined on the screen.
While the mapping properties of the BIOs are well understood for Lipschitz screens, they have not been studied for less regular screens. In remedying  this we draw heavily on our own studies of Sobolev spaces on rough domains presented recently in \cite{ChaHewMoi:13,InterpolationCWHM,HewMoi:15,CoercScreen2}.

Our assumption that the screen is planar, rather than a subset of a more general $(n-1)$-dimensional submanifold, which we anticipate could be removed with non-trivial further work,
is made so as to simplify things in two respects.
Firstly, it means that Sobolev spaces on the screen can be defined concretely in terms of Fourier transforms on the hyperplane $\Gamma_\infty$, without the need for coordinate charts.
Secondly, and more importantly, it allows one to prove that the BIOs are \emph{coercive} operators on the relevant spaces, as has been shown recently (with wavenumber-explicit continuity and coercivity estimates) in \cite{CoercScreen2}, building on previous work in \cite{Ha-Du:90,Ha-Du:92,Co:04}.
As far as possible, anticipating extensions to non-planar screens, we will seek to argue without making use of this coercivity, but we do use results from \cite{HewMoi:15} that assume coercivity to analyse dependence on the boundary in \S\ref{sec:domdep}.

The structure of the remainder of this paper is as follows. In \S\ref{sec:SobolevSpaces} we summarise results on Sobolev spaces that we use throughout the paper, paying attention to the important distinctions between different Sobolev space definitions that arise for non-Lipschitz domains.
We also introduce the trace operators, layer potentials and novel BIOs that appear in our new BIE formulations, paying careful attention to the subtleties introduced when integration is over a screen that is not Lipschitz (indeed which may have zero surface measure).

Section \ref{sec:form} is the heart of the paper. We introduce first in \S\ref{subsec:bvps} the standard weak and classical formulations, the novelty that we study their interrelation for general (rather than Lipschitz) screens. In \S\ref{subsec:bvpsnovel} and \S\ref{subsec:bies}, this a key contribution of the paper, we introduce new infinite families of BVP and BIE formulations distinguished by the sense in which the boundary condition is to be enforced, prove their well-posedness, and study their relationship to the standard weak and classical formulations. If the screen is sufficiently regular these formulations collapse to single formulations, equivalent to the standard BVPs and BIEs, but generically these formulations have infinitely
many distinct solutions.

When $\Gamma$ has empty interior the incident field may not `see' the screen; the scattered field may be zero.  Section \ref{sec:zero} studies when this does and does not happen: a key consideration is whether a set $\mathfrak{S}\subset \Gamma_\infty$ is or is not $\pm 1/2$-null (a set $\mathfrak{S}\subset \Gamma_\infty$ is $s$-null if there are no $\phi\in H^s(\Gamma^\infty)$ supported in $\mathfrak{S}$), which we study using results from \cite{HewMoi:15}. In \S\ref{subsec:embarr} we establish the size (cardinality) of our sets of novel formulations, and prove that distinct formulations have distinct solutions, at least for plane wave incidence and almost all incident directions. In \S\ref{subsec:spacesequal} we investigate, mainly using recent results from \cite{ChaHewMoi:13}, a key criterion in answering many of our questions, namely: when is $\tH^{\pm 1/2}(\mS^\circ)= H^{\pm 1/2}_\omS$? In \S\ref{subsec:nonuni} we elucidate precisely for which screens $\mS$ the classical formulations of Definitions \ref{def:sSsD-cl} and \ref{def:sD-st} are or are not well-posed. In \S\ref{sec:domdep} we study dependence of the screen scattering problems on $\Gamma$, establishing continuous dependence results for weak notions of set convergence, and use these results to select, from the infinite set of solutions arising from the formulations of \S\ref{subsec:bvpsnovel} and \S\ref{subsec:bies}, physically relevant solutions by studying a general screen as the limit of a sequence of more regular screens. Finally, in \S\ref{sec:fract} we illustrate the results of the previous sections by a number of concrete examples, mainly examples where $\mS$ is fractal or has fractal boundary.

We remark that some of the results on the BVP and BIE formulations in this paper appeared previously in the conference papers \cite{HeCh:13,Ch:13} and the unpublished report \cite{CoercScreen}, and that elements of some of the results of \S\ref{sec:domdep} and part of Example \ref{ex:cantor} appeared recently in \cite{ChaHewMoi:13} (though the results in \cite{ChaHewMoi:13} are for $\Im(k)>0$ rather than $k$ real).

\section{Preliminaries} \label{sec:SobolevSpaces}

We first set some basic notation.
For any subset $E\subset\R^n$ we denote the complement of $E$ by $E^c:=\R^n\setminus E$, the closure of $E$ by $\overline{E}$, and the interior of $E$ by $E^\circ$. We denote by $\dimH(E)$ the Hausdorff dimension of $E$ (cf.\ e.g.\ \cite[\S5.1]{AdHe}).
For subsets $E_1,E_2\subset\R^n$ we denote by $E_1\ominus E_2$ the symmetric difference $E_1\ominus E_2:=(E_1\setminus E_2)\cup( E_2\setminus E_1)$.
We say that a non-empty open set $\Omega\subset \R^n$ is $C^0$ (respectively %
Lipschitz) if its boundary $\partial\Omega$ can at each point be locally represented as the graph (suitably rotated) of a $C^0$ (respectively %
Lipschitz) function from $\R^{n-1}$ to $\R$, with $\Omega$ lying only on one side of $\partial\Omega$.
For a more detailed definition see, e.g., \cite[1.2.1.1]{Gri}.
We note that for $n=1$ there is no distinction between these definitions: we interpret them both to mean that $\Omega$ is a countable union of open intervals whose closures are disjoint and whose endpoints have no limit points. We will use, for $r>0$ and $\bx\in \R^n$, the notations $B_r(\bx):= \{\by\in \R^n: |\by-\bx|< r\}$ and $B_r:= B_r(\boldsymbol{0})$.

For $n\in \N$, let $\scrD(\R^n):=C^\infty_0(\R^n)$.
For a non-empty open set $\Omega\subset \R^n$, let $\scrD(\Omega):=\{u\in\scrD(\R^n):\supp{u}\subset\Omega\}$, and let $\scrD^*(\Omega)$ denote the associated space of distributions (continuous \emph{antilinear} functionals on $\scrD(\Omega)$).
For $s\in \R$ let $H^s(\R^n)$ denote the Sobolev space of those tempered distributions
whose Fourier transforms are locally integrable and satisfy
\mbox{$\|u\|_{H^s(\R^n)}^2:=\int_{\R^n}(1+|\bxi|^2)^{s}\,|\hat{u}(\bxi)|^2\,\rd \bxi < \infty.$}
$\scrD(\R^n)$
is dense in $H^s(\R^n)$; indeed \cite[Lemma 3.24]{McLean}, for all $v\in H^s(\R^n)$ and $\epsilon>0$ there exists $u\in \scrD(\R^n)$ such that
\begin{equation} \label{eq:approx}
\|v-u\|_{H^s(\R^n)} < \epsilon \quad \mbox{and} \quad \supp{u} \subset \{\bx\in \R^n:|\bx-\by|<\epsilon \mbox{ and } \by\in \supp{v}\},
\end{equation}
where $\supp{v}$ denotes the support of the distribution $v$.
It is also standard that $H^{-s}(\R^n)$ provides a natural unitary realisation of $(H^s(\R^n))^*$, the dual space of bounded antilinear functionals on $H^s(\R^n)$,
with the duality pairing, for $u\in H^{-s}(\R^n)$, $v\in H^s(\R^n)$,
\begin{align}
\label{DualDef}
\left\langle u, v \right\rangle_{H^{-s}(\R^n)\times H^{s}(\R^n)}:= \int_{\R^n}\hat{u}(\bxi) \overline{\hat{v}(\bxi)}\,\rd \bxi=\int_{\R^n}u \overline{v}\,\rd \bx,
\end{align}
the second equality holding by Plancherel's theorem,
whenever $u$ is locally integrable and $v\in \scrD(\R^n)$ (or vice versa).

Given a closed set $F\subset \R^n$, we define, for $s\in \R$,
\begin{equation} \label{HsTdef}
H_F^s := \{u\in H^s(\R^n): \supp{u} \subset F\},
\end{equation}
a closed subspace of $H^s(\R^n)$.
Some of our later results will depend on whether or not $H_F^s$ is trivial (i.e.\ contains only the zero distribution), for a given $F$ and $s$. Following \cite{HewMoi:15}, for $s\in \R$ we will say that a set $E\subset\R^n$ is {\em $s$-null} if there are no non-zero elements of $H^{s}(\R^n)$ supported entirely in $E$. In this terminology, for a closed set $F$, $H_F^s=\{0\}$ if and only if $F$ is $s$-null.

Given a non-empty open set $\Omega\subset\R^n$, there are a number of ways to define Sobolev spaces on $\Omega$.
First, we have $H^s_{\overline{\Omega}}$, defined as in \rf{HsTdef}.
Next, we consider the closure of $\scrD(\Omega)$ in $H^s(\R^n)$, which we denote by
\begin{equation}
\label{eq:HstildeDef}
 \tH^s(\Omega):=\overline{\scrD(\Omega)}^{H^s(\R^n)}.
\end{equation}
By definition, $\widetilde H^s(\Omega)$ is, like $H^s_{\overline{\Omega}}$, a closed subspace of $H^s(\R^n)$, and it is easy to see that $\tH^s(\Omega)\subset H^s_{\overline\Omega}$ for all $s\in\R$. When $\Omega$ is sufficiently regular (for example, if $\Omega$ is $C^0$ - see \cite[Theorem 3.29]{McLean}) it holds that $\tH^s(\Omega)=H^s_{\overline\Omega}$; however, for general $\Omega$ the two spaces can be different.
We discuss this key issue in \S\ref{subsec:spacesequal}, using results from \cite{ChaHewMoi:13}.

Next, let
$H^s(\Omega):=\{u\in \scrD^*(\Omega): u=U|_\Omega \textrm{ for some }U\in H^s(\R^n)\}$,
where $U|_\Omega$ denotes the restriction of the distribution $U$ to $\Omega$ (see e.g.\ \cite{McLean}),
with norm%
\begin{align*}
\|u\|_{H^{s}(\Omega)}:=\inf_{\substack{U\in H^s(\R^n)\\ U|_{\Omega}=u}}\normt{U}{H^{s}(\R^n)}.
\end{align*}
Where $^\perp$ denotes the orthogonal complement in $H^s(\R^n)$ and $P:H^s(\R^n)\to (H^s_{\Omega^c})^\perp$ is orthogonal projection,  it holds that $U-PU\in H^s_{\Omega^c}$ for $U\in H^s(\R^n)$, so that
\begin{equation} \label{comm}
U|_\Omega = (PU)|_\Omega \mbox{ and } \normt{PU}{H^{s}(\R^n)}\leq \normt{U}{H^{s}(\R^n)}, \quad U\in H^s(\R^n).
\end{equation}
Thus for $U\in H^s(\R^n)$ we have $\|U|_\Omega\|_{H^{s}(\Omega)} =\normt{E(U|_\Omega)}{H^{s}(\R^n)}$, where
\begin{equation} \label{exts}
E(U|_\Omega) := PU
\end{equation}
is the extension of $U|_\Omega$ from $\Omega$ to $\R^n$ with minimum norm, so that the restriction operator $|_\Omega :(H^s_{\Omega^c})^\perp\to H^s(\Omega)$ is a unitary isomorphism, with inverse $E:H^s(\Omega)\to (H^s_{\Omega^c})^\perp$.
Hence $H^s(\Omega)$ can be identified with a closed subspace of $H^s(\R^n)$, namely $(H^s_{\Omega^c})^\perp$.
We also remark that $\scrD(\overline\Omega):=\{u\in C^\infty(\Omega):u=U|_\Omega \textrm{ for some }U\in\scrD(\R^n)\}$ is a dense subset of $H^s(\Omega)$. %

Central to our analysis will be the fact that for any closed subspace $V\subset H^s(\R^n)$ the dual space $V^*$ can be unitarily realised as a subspace of $H^{-s}(\R^n)$, with duality pairing inherited from $H^{-s}(\R^n)\times H^{s}(\R^n)$. Explicitly (for more details see \cite{ChaHewMoi:13}), let $\mathcal{I}:H^{-s}(\R^n)\to (H^s(\R^n))^*$ be the unitary isomorphism implied by the duality pairing \eqref{DualDef}, i.e.\ $\mathcal{I}u(v) := \left\langle u, v \right\rangle_{H^{-s}(\R^n)\times H^{s}(\R^n)}$, let $R:H^s(\R^n)\to (H^s(\R^n))^*$ denote the standard Riesz isomorphism, and let $j:H^s(\R^n)\to H^{-s}(\R^n)$ be the unitary isomorphism defined by $j:= \mathcal{I}^{-1}R$ ($j$ can be expressed explicitly as a Bessel potential operator, see \cite[Lemma 3.2 and its proof]{ChaHewMoi:13}). Then
$
V^*\cong j(V) = (V^a)^\perp$,
where
\begin{equation} \label{ann}
V^a:=\{u\in H^{-s}(\R^n) : \left\langle u, v \right\rangle_{H^{-s}(\R^n)\times H^{s}(\R^n)}=0 \,\, \mbox{for all } v\in V\}
\end{equation}
is the annihilator of $V$ in $H^{-s}(\R^n)$, $^\perp$ denotes the orthogonal complement in $H^{-s}(\R^n)$, and the duality pairing is
\begin{equation} \label{dualgen}
 \left\langle u, v \right\rangle_{(V^a)^\perp\times V} := \left\langle u, v \right\rangle_{H^{-s}(\R^n)\times H^{s}(\R^n)}, \qquad u\in (V^a)^\perp, v\in V.
 \end{equation}

Explicitly \cite[Lemma 3.2]{ChaHewMoi:13}, if $V = H^s_F$ for $F\subset\R^n$ closed, then
\begin{align} \label{annclosed}
V^a = \tH^{-s}(F^c), \quad \mbox{so that} \quad
(H^s_F)^*&\cong (\tH^{-s}(F^c))^\perp,
\end{align}
with the duality pairing \eqref{dualgen}.
Similarly, if $V= \tH^s(\Omega)$ for $\Omega\subset\R^n$ open,
\begin{align} \label{annopen}
V^a = H^{-s}_{\Omega^c} \quad \mbox{and} \quad
(\tH^s(\Omega))^*\cong (H^{-s}_{\Omega^c})^\perp \subset (\tH^{-s}(\overline{\Omega}^c))^\perp \cong (H^s_{\overline{\Omega}})^*,
\end{align}
again with the duality pairing \eqref{dualgen}.
We note that, since $|_\Omega :(H^{-s}_{\Omega^c})^\perp\to H^{-s}(\Omega)$ is a unitary isomorphism (as noted above), the first realisation in \rf{annopen} can be replaced by the more familiar unitary realisation %
\begin{align}
\label{isdual}
(\tH^s(\Omega))^*  \cong H^{-s}(\Omega) \quad \mbox{with} \quad
\langle u,v \rangle_{H^{-s}(\Omega)\times \tH^s(\Omega)}:=\langle U,v \rangle_{H^{-s}(\R^n)\times H^s(\R^n)},
\end{align}
where $U\in H^{-s}(\R^n)$ is any extension of $u\in H^{-s}(\Omega)$ with $U|_\Omega=u$.

Sobolev spaces can also be defined as subspaces of $L^2(\R^n)$ satisfying constraints on weak derivatives. In particular, given a non-empty open $\Omega\subset \R^n$, let
$$
W^1(\Omega) := \{u\in L^2(\Omega): \nabla u \in L^2(\Omega)\}, \quad \|u\|_{W^1(\Omega)} := \left(\|u\|_{L^2(\Omega)}^2+ \|\nabla u\|_{L^2(\Omega)}^2\right)^{1/2},
$$
where $\nabla u$ is the gradient in a distributional sense.
$W^1(\R^n)=H^1(\R^n)$; in fact $W^1(\Omega)= H^1(\Omega)$ (with equivalence of norms) whenever $\Omega$ is a Lipschitz open set \cite[Theorem 3.30]{McLean}, in which case $\scrD(\overline \Omega)$ is dense in $W^1(\Omega)$.
Similarly, we define
$$
W^1(\Omega;\Delta) := \{u\in W^1(\Omega): \Delta u \in L^2(\Omega)\},\; \|u\|_{W^1(\Omega;\Delta)} := \left(\|u\|_{W^1(\Omega)}^2+ \|\Delta u\|_{L^2(\Omega)}^2\right)^{1/2},
$$
where $\Delta u$ is the Laplacian in a distributional sense, and note that, when $\Omega$ is Lipschitz, $\scrD(\overline{\Omega})$ is also dense in $W^1(\Omega;\Delta)$ \cite[Lemma 1.5.3.9]{Gri}. We define, for $s\in \R$,
\begin{equation}\label{eq:W10}
H^s_0(\Omega):=
\overline{\scrD(\Omega)\big|_\Omega}^{H^s(\Omega)}, \quad W^1_0(\Omega):=
\overline{\scrD(\Omega)\big|_\Omega}^{W^1(\Omega)},
\end{equation}
and note that, for every open set $\Omega$, $W^1_0(\Omega)=H^1_0(\Omega)$, since the $W^1(\Omega)$ norm is equivalent to the $H^1(\R^n)$ norm on $\scrD(\Omega)$. We will also use the notation $W^{1,\mathrm{comp}}(\Omega) := W^1(\Omega)\cap L^2_{\mathrm{comp}}(\Omega)$, where
$L^2_{\mathrm{comp}}(\Omega)\subset L^2(\Omega)$ is the set of restrictions to $\Omega$ of those $u\in L^2(\R^n)$ that are compactly supported.

For wave scattering problems, in which functions decay only slowly at infinity, it is convenient to define also
\begin{eqnarray*}
W^{1,\mathrm{loc}}(\Omega) &:=& \{u\in L^2_{\mathrm{loc}}(\Omega): \nabla u \in L^2_{\mathrm{loc}}(\Omega)\},\\
 W^{1,\mathrm{loc}}(\Omega;\Delta) &:=& \{u\in W^{1,\mathrm{loc}}(\Omega): \Delta u \in L^2_{\mathrm{loc}}(\Omega)\},
\end{eqnarray*}
and
$$
W^{1,\mathrm{loc}}_0(\Omega) := \{u\in W^{1,\mathrm{loc}}(\Omega):\chi|_\Omega u\in W^1_0(\Omega), \mbox{ for every } \chi \in \scrD(\R^n)\},
$$
where $L^2_{\mathrm{loc}}(\Omega)$ is the set of locally integrable functions $u$ on $\Omega$ for which $\int_G|u(\bx)|^2 \rd \bx < \infty$ for every bounded measurable $G\subset \Omega$. Analogously, for $s\geq 0$,
$$
H^{s,\mathrm{loc}}(\Omega) := \{u\in L^2_{\mathrm{loc}}(\Omega): \chi|_\Omega u\in H^s(\Omega), \mbox{ for every } \chi \in \scrD(\R^n)\}.
$$
Clearly $H^{0,\mathrm{loc}}(\Omega)=L^2_{\mathrm{loc}}(\Omega)$ and $H^{1,\mathrm{loc}}(\R^n)=W^{1,\mathrm{loc}}(\R^n)$. It holds moreover (since $n=2$ or $3$) that
\begin{equation} \label{surp}
\quad W^{1,\mathrm{loc}}(\R^n;\Delta) = H^{2,\mathrm{loc}}(\R^n)\subset C(\R^n),
\end{equation}
since if $\chi\in \scrD(\R^n)$ and $u\in W^{1,\mathrm{loc}}(\R^n;\Delta)$, then $v:=\chi u\in H^1(\R^n)$ and $\Delta v\in L^2(\R^n)$, from which it follows, by elliptic regularity (e.g., \cite{GilbargTrudinger}), that $v\in H^2(\R^n)$, so that $v\in C(\R^n)$ by the Sobolev imbedding theorem (e.g., \cite[Theorem 3.26]{McLean}).

\subsection{Function spaces on $\Gamma_\infty$ and trace operators}
\label{subsec:FuncSpacesTrace}

Recall that the propagation domain in which the scattered field is assumed to satisfy \rf{eq:he} is $D:=\R^n\setminus \omS$ ($n=2$ or $3$), where $\mS$ (the screen) is a bounded subset of the hyperplane $\Gamma_\infty:=\{\bx=(x_1,...,x_n)\in \R^n:x_n=0\}$.

To define Sobolev spaces on $\Gamma_\infty$ we make the natural association of $\Gamma_\infty$ with $\R^{n-1}$, and set $H^s(\Gamma_\infty) := H^s(\R^{n-1})$, for $s\in\R$.
For an arbitrary subset $E\subset\Gamma_\infty$ we set $\widetilde{E}:=\{\widetilde \bx\in\R^{n-1}: (\widetilde \bx,0)\in E\}\subset \R^{n-1}$. Then for a closed subset $F\subset\Gamma_\infty$ we define $H^s_F:=H^s_{\widetilde{F}}$, and for an open subset $\Omega\subset \Gamma_\infty$ we set $\tH^{s}(\Omega):=\tH^{s}(\widetilde{\Omega})$, $H^{s}_{\overline{\Omega}}:=H^{s}_{\overline{\widetilde{\Omega}}}$, $H^{s}(\Omega):=H^{s}(\widetilde{\Omega})$, and $H^{s}_0(\Omega):=H^{s}_0(\widetilde{\Omega})$. The spaces $C^\infty(\Gamma_\infty)$, $\scrD(\Gamma_\infty)$, $\scrD(\overline\Omega)$ and $\scrD(\Omega)$ are defined analogously.
Letting $U^+:=\{\bx\in\R^n:x_n>0\}$ and $U^-:=\R^n\setminus \overline{U^+}$ denote the upper and lower half-spaces, respectively, %
we define trace operators  $\gamma^\pm:\scrD(\overline{U^\pm})\to \scrD(\Gamma_\infty)$ by $\gamma^\pm u := u|_{\Gamma_\infty}$, which extend to bounded linear operators $\gamma^\pm:W^1(U^\pm)\to H^{1/2}(\Gamma_\infty)$.
Similarly, we define normal derivative operators $\partial_{\bn}^\pm:\scrD(\overline{U^\pm})\to \scrD(\Gamma_\infty)$ by
$\partial_{\bn}^\pm u = \pdonetext{u}{x_n}|_{\Gamma_\infty}$ (so the normal points into $U^+$), which
extend  (see, e.g.,\ \cite{ChGrLaSp:11})
to bounded linear operators $\partial_\bn^\pm:W^1(U^\pm;\Delta)\to H^{-1/2}(\Gamma_\infty)=(H^{1/2}(\Gamma_\infty))^*$, satisfying  Green's first identity, that
\begin{equation} \label{green1st}
\langle \partial_\bn^\pm u, \gamma^\pm v\rangle_{H^{-1/2}(\Gamma_\infty)\times H^{1/2}(\Gamma_\infty)} = \mp \int_{U^\pm} \left( \nabla u \cdot \nabla \bar v + \bar v \Delta u \right) \rd \bx,
\end{equation}
for $u\in W^1(U^\pm;\Delta)$ and $v\in W^{1}(U^\pm)$. Of note is the fact that
\begin{eqnarray}
\label{W1DCharac}
W^1(D) &=& \{u\in L^2(D): u|_{U^\pm}\in W^1(U^\pm) \textrm{ and } \gamma^+u=\gamma^-u \textrm{ on } \Gamma_\infty\setminus \omS\},\\ \label{W1DDCharac}
W^1(D;\Delta) &=& \{u\in W^1(D): u|_{U^\pm}\in W^1(U^\pm;\Delta) \textrm{ and } \partial_\bn^+u=\partial_\bn^-u \textrm{ on } \Gamma_\infty\setminus \omS\}, \hspace*{5ex}
\end{eqnarray}
with \eqref{W1DDCharac} a consequence of \eqref{green1st} and \eqref{W1DCharac}.

For compact $K\subset \Gamma_\infty$, let
$$
\scrD_{1,K}:=\{\phi \in \scrD(\R^n): \phi=1 \mbox{ in some neighbourhood of }K\}.
$$
For
$u\in W^{1,\rm loc}(D;\Delta)$,
we define,
where $\chi$ is any element of $\scrD_{1,\omS}$, the jumps
\begin{align} \label{jumpu}
[u]:=\gamma^+(\chi u)-\gamma^-(\chi u)\in H^{1/2}_{\omS} \; \mbox{ and }\;
\left[\partial_\bn u\right]:=\partial^+_\bn(\chi u)-\partial^-_\bn(\chi u)\in H^{-1/2}_{\omS}.
\end{align}
These definitions are independent of the choice of $\chi\in\scrD_{1,\omS}$, and the fact that $[u]$ and $[\partial_\bn u]$ are supported in $\omS$ follows from \eqref{W1DCharac} and \eqref{W1DDCharac}.

It is convenient, for closed $T\subset \omS$, also to use the notation
\begin{equation} \label{eq:smooth}
C^\infty_T(\overline D) := \left\{v\in C^\infty(D): v|_{U^+} = w^+\mbox{ and } v|_{U^-} = w^-, \mbox{ for some } w^\pm\in C^\infty(\overline{U_\pm}\setminus T)\right\},
\end{equation}
where $C^\infty(\overline{U_\pm}\setminus T)$ denotes the set of all $u\in C^\infty(U^\pm)$ for which the partial derivatives of $u$ of all orders have continuous extensions from $U^\pm$ to $\overline{U^\pm}\setminus T$.

\subsection{Layer potentials and boundary integral operators}
\label{sec:PotsBIOs}
Let $\Phi(\bx,\by)$ denote the fundamental solution of the Helmholtz equation such that
$v:=\Phi(\cdot,\by)$ satisfies the Sommerfeld radiation condition, that
 \begin{align} \label{src}
 \frac{\partial v(\bx)}{\partial r} -\ri k v(\bx) = o\left(r^{(1-n)/2}\right),
 \end{align}
 as $r:= |\bx|\to \infty$, uniformly in $\widehat \bx:= \bx/|\bx|$. Explicitly,
\begin{align}
\label{FundSoln}
\Phi(\bx,\by) :=
\begin{cases}
\dfrac{\re^{\ri k|\bx-\by|}}{4\pi |\bx-\by|}, & n=3,\\[3mm]
\dfrac{\ri }{4}H_0^{(1)}(k|\bx-\by|), & n=2,
\end{cases}
\qquad \bx,\by \in \R^n.
\end{align}

We define the single and double layer potentials,
\begin{align*}
\label{}
\cS:H^{-1/2}_{\omS}\to C^2(D)\cap W^{1,\rm loc}(\R^n) \mbox{ and }
\cD:H^{1/2}_{\omS}\to C^2(D)\cap W^{1,\rm loc}(D),
\end{align*}
respectively, by (note that both sign choices in $\gamma^\pm$ and $\partial^\pm_\bn$ give the same result)
\begin{align*}
\label{}
\cS\phi (\bx)&:=\left\langle \gamma^\pm(\chi\Phi(\bx,\cdot)),\overline{\phi}\right\rangle_{H^{1/2}(\Gamma_\infty)\times H^{-1/2}(\Gamma_\infty)}, \qquad \bx\in D,\, \phi\in H^{-1/2}_{\omS},\\
\cD\psi (\bx)&:=\left\langle \psi, \overline{\partial^\pm_\bn(\chi\Phi(\bx,\cdot))}\right\rangle_{H^{1/2}(\Gamma_\infty)\times H^{-1/2}(\Gamma_\infty)}, \qquad \bx\in D,\, \psi\in H^{1/2}_{\omS},
\end{align*}
where $\chi$ is any element of $\scrD_{1,\omS}$ with $\bx\not\in \supp{\chi}$. Explicitly, by \eqref{DualDef},
\begin{align}
\label{SkPotIntRep}
\cS\phi (\bx) &= \int_\omS \Phi(\bx,\by) \phi(\by)\, \rd s(\by), \qquad \bx\in D,\\
\label{DkPotIntRep}
\cD\psi (\bx) &= \int_\omS \pdone{\Phi(\bx,\by)}{\bn(\by)} \psi(\by)\, \rd s(\by), \qquad \bx\in D,
\end{align}
but with the first of these equations holding only when $\phi\in H^{-1/2}_\omS$ is locally integrable, in which case $\phi\in L^1\left(\omS\right)$. Equation \eqref{DkPotIntRep} holds for all $\psi\in H^{1/2}_\omS$ since $H^{1/2}_\omS\subset L^1(\omS)$.

The following properties of $\cS$ and $\cD$ are standard when $\mS$ is a bounded Lipschitz open subset of $\Gamma_\infty$.
The extension to general bounded $\mS\subset\Gamma_\infty$ follows immediately, noting that layer potentials on $\omS$ can be thought of as layer potentials on any larger bounded open set $\Omega\supset \omS$;
for more details see \cite[Theorem 3.1]{CoercScreen}.
\begin{thm}
\label{LayerPotRegThm}
(i) For any $\phi\in H^{-1/2}_{\omS}$ and $\psi\in H^{1/2}_{\omS}$ the potentials $\cS\phi $ and $\cD\psi $ are infinitely differentiable in $D$, and satisfy the Helmholtz equation \rf{eq:he} in $D$ and the Sommerfeld radiation condition \eqref{src};

(ii) for any $\chi\in \scrD(\R^n)$ the following mappings are bounded:
\begin{align*}
\label{}
\chi\cS:H^{-1/2}_{\omS}\to W^1(\R^n),\qquad
\chi \cD:H^{1/2}_{\omS}\to  W^1(D);
\end{align*}
(iii) the following jump relations hold for all $\phi\in H^{-1/2}_{\omS}$, $\psi\in H^{1/2}_{\omS}$, and $\chi\in \scrD_{1,\omS}$:
\begin{align}
\label{JumpRelns1}
[\cS\phi]&=0,\\
\label{JumpRelns2}
\partial^\pm_\bn(\chi \cS\phi) &= \mp\phi/2, \quad \mbox{so that } \left[\partial_\bn \cS\phi\right]=-\phi, \\
\label{JumpRelns3}
\gamma^\pm(\chi \cD\psi) &=\pm \psi/2, \quad \mbox{so that }[\cD\psi]=\psi,\\ %
\label{JumpRelns4}
\left[\partial_\bn \cD\psi \right]&=0.
\end{align}
\end{thm}

We will obtain BIE formulations for our scattering problems that  can be expressed with the help of single-layer and hypersingular operators, $S_\infty$ and $T_\infty$ respectively, defined as mappings from $\scrD(\Gamma_\infty)$ to $C^\infty(\Gamma_\infty)$ by the standard formulae
\begin{align}
\label{STDef}
S_\infty\phi (\bx) = \int_{\Gamma_\infty} \Phi(\bx,\by) \phi(\by)\, \rd s(\by), \quad
T_\infty\psi (\bx) = \pdone{}{\bn(\bx)}\int_{\Gamma_\infty} \pdone{\Phi(\bx,\by)}{\bn(\by)} \psi(\by)\, \rd s(\by),
\end{align}
for $\bx\in \Gamma_\infty$. Fix $\chi\in\scrD_{1,\omS}$, in which case $\chi=1$ in some bounded open neighbourhood $\Gamma_\dag$ of $\omS$. It is standard (e.g.\ \cite{McLean}) that, if $\phi,\psi\in \scrD(\Gamma_\infty)$, then
\begin{align} \label{STrep}
S_\infty\phi =\gamma^\pm(\chi\cS\phi) \mbox{ and } T_\infty\psi =\partial^\pm_\bn(\chi\cD\psi)\quad \mbox{in }\Gamma_\dag,
\end{align}
 so that we can define mappings $S_\dag$ and $T_\dag$ from $\scrD(\Gamma_\dag)$ to $\scrD(\overline{\Gamma_\dag})$ by
\begin{equation} \label{Sstardefs}
S_\dag\phi := S_\infty\phi|_{\Gamma_\dag} = \gamma^\pm(\chi\cS\phi)|_{\Gamma_\dag}, \quad T_\dag\psi := T_\infty\psi|_{\Gamma_\dag}= \partial^\pm_\bn(\chi\cD\psi)|_{\Gamma_\dag},
\end{equation}
for $\phi, \psi\in \scrD(\Gamma_\dag)$. It is clear from the mapping properties of the trace operators, those of $\cS$ and $\cD$ in Theorem \ref{LayerPotRegThm}(ii), and the density of $\scrD(\Gamma_\dag)$ in $\tH^s(\Gamma_\dag)$, for $s\in\R$, that the representations \eqref{Sstardefs} extend the definitions of $S_\dag$ and $T_\dag$ to bounded operators $S_\dag:\tH^{-1/2}(\Gamma_\dag)\to H^{1/2}(\Gamma_\dag)$ and $T_\dag: \tH^{1/2}(\Gamma_\dag)\to H^{-1/2}(\Gamma_\dag)$. (These mapping properties are well-known in the case that $\Gamma_\dag$ is Lipschitz or smoother - see e.g.\ \cite{Ha-Du:90,Ha-Du:92}, where $\Gamma$ is assumed $C^\infty$.)

Recall that $H^{\pm1/2}(\Gamma_\dag)$ can be identified with the dual space $\left(\tH^{\mp 1/2}(\Gamma_\dag)\right)^*$ via the unitary mapping implied by the duality pairing \eqref{isdual}.
As noted in \S\ref{sec:SobolevSpaces}, an alternative natural unitary realisation of $\left(\tH^{\mp 1/2}(\Gamma_\dag)\right)^*$, via the duality pairing \eqref{dualgen}, is  $\left(H^{\pm 1/2}_{\Gamma_\dag^c}\right)^\perp\subset H^{\pm 1/2}(\Gamma_\infty)$. We can define versions of $S_\dag$ and $T_\dag$, $S_\ddag:\tH^{-1/2}(\Gamma_\dag)\to \left(H^{ 1/2}_{\Gamma_\dag^c}\right)^\perp$ and $T_\ddag:\tH^{1/2}(\Gamma_\dag) \to \left(H^{- 1/2}_{\Gamma_\dag^c}\right)^\perp$, which map to these alternate realisations of the dual spaces,  by
\begin{equation} \label{Sstardefs2}
S_\ddag\phi := P_{+}\gamma^\pm(\chi\cS\phi), \quad T_\ddag\psi := P_{-}\partial^\pm_\bn(\chi\cD\psi),
\end{equation}
for $\phi\in \tH^{-1/2}(\Gamma_\dag)$, $\psi\in \tH^{+1/2}(\Gamma_\dag)$, where $P_{\pm}$ denotes orthogonal projection onto $\left(H^{\pm 1/2}_{\Gamma_\dag^c}\right)^\perp$ in $H^{\pm 1/2}(\Gamma_\infty)$. Since $\varphi|_{\Gamma_\dag} =0$ for $\varphi\in H^{\pm 1/2}_{\Gamma_\dag^c}$, we see from \eqref{Sstardefs} and \eqref{Sstardefs2} that
\begin{equation} \label{SddSd}
S_\dag = |_{\Gamma_\dag} S_\ddag, \quad T_\dag = |_{\Gamma_\dag} T_\ddag,
\end{equation}
with the mapping $|_{\Gamma_\dag}:\left(H^{\pm 1/2}_{\Gamma_\dag^c}\right)^\perp\to H^{\pm 1/2}(\Gamma_\dag)$ a unitary isomorphism, as noted in \S\ref{sec:SobolevSpaces}, whose inverse \eqref{exts} takes $\varphi\in H^{\pm 1/2}(\Gamma_\dag)$ to its unique extension in $H^{\pm 1/2}(\Gamma_\infty)$ with minimum norm.

Thus $S_\dag \phi$ is simply the restriction of $S_\ddag \phi$ to $\Gamma_\dag$, and $S_\ddag \phi$ the minimum norm extension of $S_\dag \phi$; and the same relationship holds between $T_\dag \phi$ and $T_\ddag \phi$. Moreover, it is immediate from \eqref{STDef}, \eqref{STrep}, \eqref{Sstardefs}, and \eqref{SddSd}, that, for $\phi,\psi\in \scrD(\Gamma_\dag)$ and $\bx \in \Gamma_\dag$,
\begin{eqnarray}
\label{SdagDef}
S_\dag\phi (\bx) &=&  S_\ddag\phi(\bx) = \int_{\Gamma_\dag} \Phi(\bx,\by) \phi(\by)\, \rd s(\by),\\ \label{TdagDef}
T_\dag\psi (\bx) &=& T_\ddag\psi(\bx) =\pdone{}{\bn(\bx)}\int_{\Gamma_\dag} \pdone{\Phi(\bx,\by)}{\bn(\by)} \psi(\by)\, \rd s(\by).
\end{eqnarray}

To write down weak forms of BIEs, we introduce sesquilinear forms associated to these BIOs, defined by
\begin{align} \label{eq:sesD}
a_S(\phi,\psi) := \left\langle S_\dag \phi, \psi\right\rangle_{H^{1/2}(\Gamma_\dag)\times \tH^{-1/2}(\Gamma_\dag)} = \left\langle S_\ddag \phi, \psi\right\rangle_{H^{1/2}(\Gamma_\infty)\times H^{-1/2}(\Gamma_\infty)},
\end{align}
for $\phi,\psi\in \tH^{-1/2}(\Gamma_\dag)$, and
\begin{align}\label{eq:sesN}
a_T(\phi,\psi) := \left\langle T_\dag \phi, \psi\right\rangle_{H^{-1/2}(\Gamma_\dag)\times \tH^{1/2}(\Gamma_\dag)} = \left\langle T_\ddag \phi, \psi\right\rangle_{H^{-1/2}(\Gamma_\infty)\times H^{1/2}(\Gamma_\infty)},
\end{align}
for $\phi,\psi\in \tH^{1/2}(\Gamma_\dag)$. Explicitly, for $\phi,\psi\in \scrD(\Gamma_\dag)$ (which is dense in $\tH^{\pm 1/2}(\Gamma_\dag)$), it follows from \eqref{DualDef}, \eqref{STDef}, \eqref{STrep}, and \eqref{Sstardefs}, that
\begin{equation} \label{eq:sesD2}
a_S(\phi,\psi) = \int_{\Gamma_\dag}S_\infty \phi\, \bar\psi \, \rd s(\bx) \quad \mbox{and} \quad
a_T(\phi,\psi) = \int_{\Gamma_\dag}T_\infty \phi\, \bar\psi \, \rd s(\bx),
\end{equation}
with the actions of $S_\infty$ and $T_\infty$ given by \eqref{STDef}.
These sesequilinear forms are continuous and coercive, in the sense of the following theorem, taken from \cite{CoercScreen2} (and see  \cite{Ha-Du:92,Co:04}). We remark that coercivity in this sense is unusual for BIOs for scattering problems. More usual -- and in fact this would be enough for most of our later analysis -- is that the BIO is a compact perturbation of a coercive operator (where by a coercive operator we mean one whose associated sesquilinear form is coercive). We note that \cite{CoercScreen2} gives explicit expressions for the constants in this theorem, as functions of the dimension $n$ and $kL$,  where $L:=\mathrm{diam}(\Gamma_\dag)$.

\begin{thm}
\label{thm:ContCoerc1}
The sesquilinear forms $a_S$ and $a_T$ are {\em continuous} and {\em coercive}, i.e., there exist constants $c_S,C_S,c_T,C_T>0$ such that
\begin{align}
\label{}
&|a_S(\phi,\psi)| \leq C_S \|\phi\|_{\tH^{-1/2}(\Gamma_\dag)}\, \|\psi\|_{\tH^{-1/2}(\Gamma_\dag)},\quad
|a_S(\phi,\phi)|\geq c_S \|\phi\|^2_{\tH^{-1/2}(\Gamma_\dag)},
\end{align}
for all $\phi, \psi\in \tH^{-1/2}(\Gamma_\dag)$, and
\begin{align}
&|a_T(\phi,\psi)| \leq C_T \|\phi\|_{\tH^{1/2}(\Gamma_\dag)}\, \|\psi\|_{\tH^{1/2}(\Gamma_\dag)},\quad
|a_T(\phi,\phi)|\geq c_T \|\phi\|^2_{\tH^{1/2}(\Gamma_\dag)},
\end{align}
for all $\phi, \psi\in \tH^{1/2}(\Gamma_\dag)$.
\end{thm}

Since $H^{\pm 1/2}(\Gamma_\dag)$ is a {\em unitary} realisation of $(\tH^{\mp 1/2}(\Gamma_\dag))^*$ through the duality pairing \eqref{isdual}, the upper bounds in this theorem are equivalent to the bounds
\begin{equation} \label{norms}
\|S_\dag\| = \|S_\ddag\| \leq C_S, \quad  \|T_\dag\| = \|T_\ddag\| \leq C_S.
\end{equation}
Further, by Lax-Milgram, the above theorem implies that these operators are invertible, with
\begin{equation} \label{invnorms}
\|S^{-1}_\dag\| = \|S^{-1}_\ddag\| \leq c^{-1}_S, \quad  \|T^{-1}_\dag\| = \|T^{-1}_\ddag\| \leq c^{-1}_T.
\end{equation}

Our BIEs will be expressed in terms of single-layer and hypersingular operators associated to the screen $\mS$, defined analogously to \eqref{Sstardefs2}. Specifically, let $V^{\pm}$ denote any closed subspace of $H^{\pm 1/2}_\omS$ (so that $V^\pm\subset \tH^{\pm 1/2}(\Gamma_\dag)$ by \eqref{eq:approx}), and let $V^\mp_*:= ((V^\pm)^a)^\perp$ denote the natural unitary realisation of
$\left(V^{\pm}\right)^*$ implied by the duality pairing \eqref{dualgen}, where $(V^\pm)^a$ denotes the annihilator of $V^\pm$ in $H^{\mp 1/2}(\Gamma_\infty)$, defined by \eqref{ann}. We note in particular that if $V^\pm = \tH^{\pm1/2}(\Omega)$, for some open $\Omega\subset \omS$, or if $V^\pm = H^{\pm 1/2}_F$, for some closed $F\subset \omS$, then $V^\mp_*$ are given explicitly by \eqref{annopen} and \eqref{annclosed}, as
 \begin{equation} \label{eq:duals}
 V^\mp_* = \left(H^{\mp 1/2}_{\Omega^c}\right)^\perp\quad \mbox{ and } \quad V^\mp_* = \left(\tH^{\mp 1/2}(F^c)\right)^\perp,
\end{equation}
respectively.
Let $P_{V^\pm_*}$ denote orthogonal projection onto $V^\pm_*$ in $H^{\pm 1/2}(\Gamma_\infty)$. Then the operators in our BIE formulations will be the single-layer and hypersingular operators, $S:V^-\to V^+_*$ and $T:V^+\to V^-_*$, defined by
\begin{align} \label{eq:BIOs}
S\phi :=P_{V^{+}_*}\gamma^\pm(\chi\cS\phi), \qquad \phi\in V^{-}, \quad \quad
T\psi :=P_{V^{-}_*}\partial^\pm_\bn(\chi\cD\psi), \qquad \psi\in V^{+}.
\end{align}
These definitions are independent of the choice of the $\pm$ sign in the trace operators, by \eqref{JumpRelns1} and \eqref{JumpRelns4}, and are independent of the choice of $\chi\in\scrD_{1,\omS}$. (If $\chi_1,\chi_2\in \scrD_{1,\omS}$, then $\gamma^\pm((\chi_1-\chi_2)\cS\phi)\in (V^\pm)^a$, so that $P_{V^{+}_*}\gamma^\pm((\chi_1-\chi_2)\cS\phi)=0$; similarly, $P_{V^{-}_*}\partial^\pm_\bn((\chi_1-\chi_2)\cD\psi)=0$.)

Since $V^\pm$ are closed subspaces of $\tH^{\pm 1/2}(\Gamma_\dag)$, it follows from \eqref{Sstardefs2} and \eqref{eq:BIOs} that
\begin{equation} \label{STdag}
S = P_{V^{+}_*}S_\ddag|_{V^-}, \quad T = P_{V^{-}_*}T_\ddag|_{V^+},
\end{equation}
so that, since $\|P_{V^\pm_*}\|=1$, \eqref{norms} implies that
\begin{equation} \label{norms2}
\|S\| \leq C_S, \quad  \|T\| \leq C_T.
\end{equation}
Further, it is immediate from the definitions of $V^\pm_*$ and $P_{V^\pm_*}$ that
\begin{equation} \label{rest1}
\langle S\phi,\psi\rangle_{V^+_*\times V^-} = \langle S\phi,\psi\rangle_{H^{1/2}(\Gamma_\infty)\times H^{-1/2}(\Gamma_\infty)} = a_S(\phi,\psi),
\quad \phi,\psi\in V^-,
\end{equation}
\begin{equation} \label{rest2}
\langle T\phi,\psi\rangle_{V^-_*\times V^+} = \langle T\phi,\psi\rangle_{H^{-1/2}(\Gamma_\infty)\times H^{1/2}(\Gamma_\infty)} = a_T(\phi,\psi),
\quad \phi,\psi\in V^+.
\end{equation}
In other words, the sesquilinear forms corresponding to $S$ and $T$ are just the restrictions of $a_S$ and $a_T$ to the subspaces $V^-$ and $V^+$, respectively. Thus these sesquilinear forms are coercive with the same constants, and $S$ and $T$ are invertible by Lax-Milgram with
\begin{equation} \label{invert}
\|S^{-1}\| \leq c_S^{-1}, \quad  \|T^{-1}\| \leq c_T^{-1}.
\end{equation}

\section{Formulating screen scattering problems} \label{sec:form}

\subsection{Standard BVP formulations and their interrelation}  \label{subsec:bvps}

In this section we study the standard formulations for screen scattering from the literature.
We will see that these formulations are equivalent for screens that occupy open sets in $\Gamma_\infty$ with Lipschitz boundaries (this is well-known), but that some of these standard formulations (Problems \SDcl, \SNcl, \Dst, and \Nsts below) fail to have unique solutions in less regular cases (this is explored in \S\ref{subsec:nonuni} below). In the next section, \S\ref{subsec:bvpsnovel}, we will introduce new families of formulations that are well-posed in all cases. %

We begin by stating precisely the standard weak formulations of the sound-soft (Dirichlet) and sound-hard (Neumann) scattering problems that we have referred to in the introduction. In each case the problem is to find the scattered field $u^s$, or equivalently the total field $u=u^i+u^s$, given the incident field $u^i$.
\begin{defn}[Problem \SDw] \label{def:sSsDw}
Given $u^i\in W^{1,\rm loc}(\R^n;\Delta)$, find $u\in W^{1,\rm loc}(D;\Delta)\cap W_0^{1,\rm loc}(D)$ such that $u^s:= u-u^i$ satisfies the Helmholtz equation \rf{eq:he} in $D$ and the Sommerfeld radiation condition \eqref{src}.
\end{defn}

\begin{defn}[Problem \SNw] \label{def:sSsNw}
Given $u^i\in W^{1,\rm loc}(\R^n;\Delta)$, find $u\in W^{1,\rm loc}(D;\Delta)$ such that $u^s:= u-u^i$ satisfies the Helmholtz equation \rf{eq:he} in $D$ and the Sommerfeld radiation condition \eqref{src}, and such that $u$ satisfies the weak sound-hard boundary condition \eqref{eq:wnc}.
\end{defn}

\begin{rem} \label{rem:local} It is easy to see that the solutions $u^s$ to Problems \SDws and \SNws depend only on $u^i$ in a neighbourhood of $\omS$. Precisely, if $u^s$ is a solution to \SDws (\SNw), then $u^s$ is also a solution to \SDws (\SNw) with $u^i$ replaced with $u^i_\sharp$, provided $u^i=u^i_\sharp$ in some neighbourhood of $\omS$. In particular, without changing the set of solutions $u^s$, we can replace $u^i$ by $u^i_\sharp:= \chi u^i$, for any $\chi\in \scrD_{1,\omS}$, in which case $u^i_\sharp\in W^{1}(\R^n;\Delta)$ and is compactly supported, and $u_\sharp := u^i_\sharp + u^s\in W^{1,\rm loc}(D;\Delta)$ satisfies \eqref{eq:he} outside the support of $\chi$ and the Sommerfeld radiation condition \eqref{src}.
\end{rem}

\begin{rem} \label{rem:equiv} An alternative way of formulating the scattering problem \SDws is to start with a given $f\in L^2(D)=L^2(\R^n)$ that has bounded support and seek $u\in W^{1,\rm loc}(D;\Delta)\cap W_0^{1,\rm loc}(D)$ which satisfies
\begin{equation} \label{eq:he_inhom}
\Delta u + k^2 u = f
\end{equation}
in $D$ and the Sommerfeld radiation condition \eqref{src}. This is equivalent to \SDws in the sense that if $u$ satisfies this formulation and we define $u^i\in W^{1,\rm loc}(\R^n;\Delta)$ to be the unique solution of $\Delta u + k^2 u = f$ in $\R^n$ which satisfies \eqref{src}, explicitly
\begin{equation} \label{eq:ui}
u^i(x) = -\int_{\R^n}\Phi(\bx,\by) f(\by) \rd \by, \quad \bx\in \R^n,
\end{equation}
then $u$ and $u^s:= u-u^i$ satisfy \SDw. Conversely, if $u^i\in W^{1,\rm loc}(\R^n;\Delta)$ is given and $u^s$ satisfies \SDw, then, by Remark \ref{rem:local}, $u^s$ also satisfies \SDws with $u^i$ replaced by $u^i_\sharp := \chi u^i$, for any $\chi\in \scrD_{1,\omS}$, and defining $u:= u^i_\sharp + u^s$, $u$ satisfies \eqref{src} and $\Delta u + k^2 u = f:= (\Delta + k^2)u^i_\sharp$, which is in $L^2(D)$ and has bounded support. Identical remarks apply regarding the alternative formulation of \SNw.
\end{rem}

The following well-posedness is classical, and can be established by combining the equivalence of formulations in Remark \ref{rem:equiv} with results in \cite[Corollary 4.5]{wilcox:75} for \SNw, in \cite{NeittaanmakiRoach:87} for \SDw. In each case uniqueness follows from Green's first theorem and a result of Rellich (cf.\ proof of Lemma \ref{lem:uni} below), and existence from uniqueness and local compactness and limiting absorption arguments. These require in the Neumann case that the domain $D$ satisfies a local compactness condition, which it does as $D$ satisfies Wilcox's finite tiling property - see \cite[Theorem 4.3 and p.62]{wilcox:75}.

\begin{thm} \label{thm:bvp_wwp} Problems \SDws and \SNws have exactly one solution for every $u^i\in W^{1,\mathrm{loc}}(D;\Delta)$.
\end{thm}

The following lemma uses the notations introduced above \eqref{eq:BIOs}, so that $V^\pm$ is any closed subspace of $H^{\pm 1/2}_\omS$ and $P_{V^{\mp}_*}$ orthogonal projection onto the realisation $V^{\mp}_*\subset H^{\mp 1/2}(\Gamma_\infty)$ of its dual space. This lemma will allow us to make connections between \SDw, \SNws and the other formulations we introduce below.

\begin{lem} \label{lem:equiv} If $u$ satisfies \SDws or \SNws then $u\in C(D)$, $u^s\in C^\infty(D)$, $[u]\in H^{1/2}_\omS$, $[\partial_\bn u]\in H^{-1/2}_\omS$, and
\begin{equation} \label{eq:jumpszero}
[u^s]=[u] = 0 \mbox{ if $u$ satisfies \SDw, while } [\partial_\bn u^s]=[\partial_\bn u] = 0 \mbox{ if $u$ satisfies \SNw}.
\end{equation}
Further, for all $\chi\in \scrD_{1,\omS}$,
\begin{equation} \label{eq:C1}
P_{V_*^{+}}\gamma^\pm(\chi u)=0 \mbox{ with } V^- = H^{-1/2}_\omS, \mbox{ which implies that } \gamma^\pm(\chi u)|_{\omS^\circ} = 0,
\end{equation}
 if $u$ satisfies \SDw, while
\begin{equation} \label{eq:C2}
P_{V_*^{-}}\partial_\bn^\pm(\chi u)=0 \mbox{ with }V^+=H^{1/2}_\omS, \mbox{ which implies that } \partial_\bn^\pm(\chi u)|_{\omS^\circ} = 0,
\end{equation}
 if $u$ satisfies \SNw.
\end{lem}
\begin{proof} If $u$ satisfies \SDws or \SNws then $u^s\in C^\infty(D)$ by \eqref{eq:he} and standard elliptic regularity (e.g., \cite{GilbargTrudinger}).  Further it follows from \eqref{surp} that $u^i\in C(\R^n)$, so that $u\in C(D)$.
That $[u]\in H^{1/2}_\omS$ and $[\partial_\bn u]\in H^{-1/2}_\omS$ follows since $u\in W^{1,\rm loc}(D;\Delta)$, as noted below \eqref{jumpu}. Note also that $[u^i]=0$ and $[\partial_\bn u^i]=0$, as $u^i\in W^{1,\rm loc}(\R^n;\Delta)$, so that $[u^s]=[u]$ and $[\partial_\bn u^s]=[\partial_\bn u]$.

If $u$ satisfies \SDw, then also $[u]=0$ by density as $[v]=0$ if $v\in \scrD(D)$. Further, if $v\in \scrD(D)$, then $\gamma^\pm(\chi v)$ is in the annihilator of $H^{-1/2}_\omS$, i.e.\ $P_{V_*^{+}}\gamma^\pm(\chi v)=0$ with $V^- = H^{-1/2}_\omS$, and the same holds for $u$ by density. As $\tH^{-1/2}(\omS^\circ)\subset H^{-1/2}_\omS$, this implies $P_{V_*^{+}}\gamma^\pm(\chi v)=0$ with $V^- = \tH^{-1/2}(\omS^\circ)$ and $V_*^+ = (H^{1/2}_{(\omS^\circ)^c})^\perp$, which is equivalent, recalling \eqref{comm} and that $|_{\omS^\circ}:V_*^+\to H^{1/2}(\omS^\circ)$ is an isomorphism, to $\gamma^\pm(\chi u)|_{\omS^\circ} = 0$.

 Suppose now that $u$ satisfies \SNw. To show $[\partial_\bn u]=0$ it is enough, given the density \eqref{eq:approx}, to show that $\langle [\partial_\bn u],\phi\rangle_{H^{-1/2}(\Gamma_\infty)\times H^{1/2}(\Gamma_\infty)}=0$ for all $\phi\in \scrD(\Gamma_\infty)$.
But if  $\phi\in \scrD(\Gamma_\infty)$, choosing $v\in W^{1,\mathrm{comp}}(\R^n)\subset W^{1,\mathrm{comp}}(D)$ so that $\gamma^\pm \bar v = \phi$, and $\chi\in \scrD_{1,\omS}$ such that $\chi=1$ in a neighbourhood of the support of $v$, it follows from \eqref{green1st} that
\begin{eqnarray*}
\langle [\partial_\bn u],\phi\rangle_{H^{-1/2}(\Gamma_\infty)\times H^{1/2}(\Gamma_\infty)} &=& -\int_{D} \left( \nabla (\chi u) \cdot \nabla v + v \Delta (\chi u) \right) \rd x\\ &=& -\int_{D} \left( \nabla u \cdot \nabla v + v \Delta  u \right) \rd x = 0.
\end{eqnarray*}
Arguing similarly, given $\phi\in H_\omS^{1/2}$ it is clear from \eqref{W1DCharac} that one can choose $v\in W^{1,\mathrm{comp}}(D)$ so that $\gamma^+ \bar v = \phi$ and $v=0$ in $U_-$, and deduce that
$$
\langle \partial^+_\bn (\chi u),\phi\rangle_{H^{-1/2}(\Gamma_\infty)\times H^{1/2}(\Gamma_\infty)} = -\int_{D} \left( \nabla u \cdot \nabla v + v \Delta  u \right) \rd x = 0,
$$
so that $\partial^+_\bn(\chi u)$, and also $\partial^-_\bn(\chi u)$ as $[\partial_\bn u]=0$, is in the annihilator of $H^{1/2}_\omS$, i.e.\ $P_{V_*^{-}}\partial_\bn^\pm(\chi v)=0$ with $V^+ = H^{1/2}_\omS$. This implies, arguing as above for the Dirichlet case, that $\partial^\pm_\bn(\chi u)|_{\omS^\circ}=0$.
\end{proof}

The next lemma is immediate from standard elliptic regularity results up to the boundary (e.g., \cite{GilbargTrudinger}).

\begin{lem} \label{lem:reggen}
Suppose that $v\in W^{1,\rm loc}(D;\Delta)$ satisfies the Helmholtz equation \eqref{eq:he} in $D$ and that, for some $\chi\in \scrD_{1,\omS}$, either $\gamma^\pm(\chi v)|_{\mS^\circ}\in C^\infty(\mS^\circ)$ or $\partial_{\bn}^\pm(\chi v)\in C^\infty(\mS^\circ)$. Then $v$ is smooth up to the boundary away from $\partial \mS$, i.e., $v\in C^\infty_{\partial \mS}(\overline D)$ in the notation of \eqref{eq:smooth}.
\end{lem}

From Lemma \ref{lem:reggen} %
and Lemma \ref{lem:equiv} we have the following corollary.

\begin{cor} \label{cor:reg} If $u$ satisfies \SDws or \SNws and $\Delta u^i + k^2 u^i = 0$ in a neighbourhood of $\omS$ (so that $u^i$ is $C^\infty$ in a neighbourhood of $\omS$), then $u^s\in C^\infty_{\partial \omS}(\overline D)$.
\end{cor}

\begin{rem} Related to Corollary \ref{cor:reg}, common choices for the incident field in Problems \SDws and \SNws are the plane wave
\begin{equation} \label{eq:pw}
u^i(\bx) = \exp(\ri k\bd \cdot \bx), \quad \bx \in \R^n,
\end{equation}
for some unit vector $\bd\in \R^n$, and the incident field \eqref{eq:ui}, for some $f\in L^2(D)$ compactly supported in $D$, both satisfying $(\Delta + k^2)u^i=0$ in a neighbourhood of $\omS$.  In particular if, for some $\by\in D$, $c\in \C$, and $\varepsilon < \mathrm{dist}(\by,\omS)$, we define $f(\bx) := c$ for $|\by-\bx| < \varepsilon$, $f(\bx):=0$, otherwise, then \eqref{eq:ui} implies
\begin{equation} \label{eq:cw}
u^i(\bx) = C\Phi(\bx,\by), \quad \mbox{for }|\bx-\by|\geq \varepsilon,
\end{equation}
where $C$ depends on $c$, $\varepsilon$, and $k$, this an incident cylindrical (spherical) wave for $n=2$ ($n=3$).
\end{rem}

\SDws and \SNws are formulations of screen scattering with the boundary conditions understood in weak, generalised senses. It is also possible to impose the boundary conditions in a classical sense, if $u^i$ is sufficiently smooth near $\mS$. The following is the obvious generalisation to an arbitrary screen $\mS$ of early BVP formulations for diffraction by screens (see \cite{Bouwkamp:54} and the references therein), in which there is a (usually implicit) assumption of smoothness of the solution up to the boundary away from the screen edge, and an assumption of finite energy density (in other words finite $W^1$ norm of $u^s$) in some neighbourhood of the screen boundary, this the Meixner \cite{Meixner:49} {\em edge condition}.

\begin{defn}[Problems \SDcls and \SNcl] \label{def:sSsD-cl}
Given $u^i\in W^{1,\rm loc}(\R^n;\Delta)$ such that $\Delta u^i + k^2 u^i = 0$ in a neighbourhood of $\omS$, find $u^s\in C^\infty_{\partial \mS}(\overline D)$ that satisfies the Helmholtz equation \rf{eq:he} in $D$, the Sommerfeld radiation condition \eqref{src}, the boundary condition for $\bx\in \mS^\circ$ that
\begin{equation} \label{eq:bc}
\begin{array}{cc}
  \displaystyle{\lim_{\by\in D,\;\by \to \bx} u^s(\by)}=-u^i(\bx) & \mbox{for Problem \SDcl}, \\
  \\
     \displaystyle{\lim_{\by\in D,\;\by \to \bx}\frac{\partial u^s(\by)}{\partial \bn}}=-\displaystyle{\frac{\partial u^i(\bx)}{\partial \bn}} & \mbox{for Problem \SNcl},
\end{array}
\end{equation}
and the edge condition that $\int_N (|\nabla u^s|^2 + |u^s|^2)\rd \bx <\infty$, where $N:= \{\bx\in D: \mathrm{dist}(\bx,\partial \mS) < \epsilon\}$, for some $\epsilon >0$.
\end{defn}

Problems \SDcls and \SNcls are phrased as Dirichlet and Neumann BVPs, respectively, with boundary data in \eqref{eq:bc} in terms of $u^i$. We can also study Dirichlet and Neumann BVPs with more general boundary data. The following is a standard Sobolev space formulation (e.g., \cite{Ste:87}).

\begin{defn}[Problems \Dsts and \Nst] \label{def:sD-st}
Given $\widetilde g_\sD\in H^{1/2}(\mS^\circ)$ and $\widetilde g_\sN\in H^{-1/2}(\mS^\circ)$, find $v\in C^2\left(D\right)\cap  W^{1,\mathrm{loc}}(D)$ satisfying the Helmholtz equation \rf{eq:he} in $D$, the Sommerfeld radiation condition \eqref{src}, and the boundary conditions,  for some $\chi\in \scrD_{1,\omS}$, that
\begin{equation} \label{eq:bcst}
\begin{array}{cc}
  \displaystyle{\gamma^\pm(\chi v)|_{\mS^\circ}=\widetilde g_{\sD}},\;\; & \mbox{for Problem \Dst}, \\
  \\
     \displaystyle{\partial_{\bn}^\pm(\chi v)|_{\mS^\circ}=\widetilde g_{\sN}},\;\; & \mbox{for Problem \Nst}.
\end{array}
\end{equation}
\end{defn}

The following lemma is immediate from Lemma \ref{lem:reggen}, noting that if $u\in W^{1,{\rm loc}}(D;\Delta)$ and $\chi u\in C^\infty_{\partial \mS}(\overline D)$, for some $\chi\in \scrD_{1,\omS}$, then the Dirichlet (Neumann) boundary condition holds in \eqref{eq:bc} if and only if $\gamma^\pm(\chi u)|_{\mS^\circ}=0$ ($\partial^\pm_\bn(\chi u)|_{\mS^\circ}=0$).

\begin{lem} \label{lem:equivclst} Suppose that $u^i$ satisfies the conditions of Problems \SDcls  and \SNcl.  Then $u^s$ satisfies \SDcls (\SNcl) if and only if $u^s$ satisfies \Dsts (\Nst) with $\widetilde g_{\sD} := -u^i|_{\mS^\circ}$ ($\widetilde g_\sN := -\frac{\partial u^i}{\partial \bn}|_{\mS^\circ}$).
\end{lem}

\begin{cor}\label{lem:uniequiv} Problem \SDcls (\SNcl) has at most one solution if and only if the same holds for \Dsts (\Nst).
\end{cor}

The following lemma is immediate from Lemmas \ref{lem:equiv} and \ref{lem:equivclst}. The corollary follows from Lemma \ref{lem:wcl} and Theorem \ref{thm:bvp_wwp}.

\begin{lem} \label{lem:wcl} Suppose that $u^i$ satisfies the conditions of Problems \SDcls  and \SNcl. If $u^s$ satisfies \SDws (\SNw), then $u^s$ satisfies \Dsts (\Nst) with $\widetilde g_{\sD} := -u^i|_{\mS^\circ}$ ($\widetilde g_\sN := -\frac{\partial u^i}{\partial \bn}|_{\mS^\circ}$), and satisfies \SDcls (\SNcl).
\end{lem}

\begin{cor} \label{cor:atleastS} Problems \SDcls and \SNcls have at least one solution.
\end{cor}

If $\mS$ is an open set with sufficiently regular boundary, we will see in \S\ref{subsec:nonuni} that \SDcls and \SNcls are equivalent to \SDws and \SNw, so that they both have exactly one solution.  But we will show that uniqueness does not hold for \SDcls and \SNcls (nor, by Corollary \ref{lem:uniequiv}, for \Dsts and \Nst) for general $\Gamma$, unless further constraints are imposed. In particular, uniqueness can fail if $\mS^\circ$ is empty, when \eqref{eq:bc} is empty. But we will also see (Theorem \ref{thm:wpst}) that uniqueness fails when $\mS$ is an open set if $\partial \mS$ is sufficiently wild.

\subsection{Novel families of BVPs for screen scattering} \label{subsec:bvpsnovel}

We now introduce some new families of BVP formulations for screen scattering. Our new formulations are defined in terms of the notations introduced above \eqref{eq:BIOs}, so that $V^\pm$ is any closed subspace of $H^{\pm 1/2}_\omS$ and $P_{V^{\mp}_*}$ orthogonal projection onto the realisation $V^{\mp}_*\subset H^{\mp 1/2}(\Gamma_\infty)$ of its dual space.
For any such spaces $V^\pm$ we define the following BVPs and associated scattering problems. In all these definitions the choice of $\chi\in \scrD_{1,\omS}$ is arbitrary.

\begin{defn}[Problem $\sD(V^-)$]
Given $g_{\sD}\in V^{+}_*$, find $v\in C^2\left(D\right)\cap  W^{1,\mathrm{loc}}(D)$ satisfying the Helmholtz equation \rf{eq:he} in $D$, the Sommerfeld radiation condition \eqref{src}, and the boundary conditions
\begin{align}
\label{a3}P_{V_*^{+}}\gamma^\pm(\chi v)&=g_{\sD}, \qquad \mbox{ for some }\chi\in \scrD_{1,\omS},\\
\label{a1}[v]&=0,\\
\label{a2}\left[\partial_\bn v\right] &\in V^{-}.
\end{align}
\end{defn}

\begin{defn}[Problem $\sN(V^+)$]
Given $g_{\sN}\in V^{-}_*$, find $v\in C^2\left(D\right)\cap  W^{1,\mathrm{loc}}(D)$ satisfying the Helmholtz equation \rf{eq:he} in $D$, the Sommerfeld radiation condition \eqref{src}, and the boundary conditions
\begin{align}
\label{b3}P_{V_*^{-}}\partial_{\bn}^\pm(\chi v)&=g_{\sN}, \qquad \mbox{ for some }\chi\in \scrD_{1,\omS},\\
\label{b1}\left[\partial_\bn v\right]&=0,\\
\label{b2}[v] &\in V^{+}.
\end{align}
\end{defn}

\begin{defn}[Problem $\sS\sD(V^-)$]
Given $u^i\in W^{1,\rm loc}(\R^n;\Delta)$ find $u^s$ satisfying Problem $\sD(V^-)$ with
\begin{equation} \label{gd}
g_{\sD} := -P_{V_*^{+}}\gamma^\pm(\chi u^i), \qquad \mbox{ for some }\chi\in \scrD_{1,\omS}.
\end{equation}
\end{defn}

\begin{defn}[Problem $\sS\sN(V^+)$]
Given $u^i\in W^{1,\rm loc}(\R^n;\Delta)$ find $u^s$ satisfying Problem $\sN(V^+)$ with
\begin{equation} \label{gn}
g_{\sN} :=-P_{V_*^{-}}\partial_\bn^\pm(\chi u^i), \qquad \mbox{ for some }\chi\in \scrD_{1,\omS}.
\end{equation}
\end{defn}

\begin{rem} In general, conditions \eqref{a1}, \eqref{a2}, \eqref{b1}, and \eqref{b2} in the above formulations are essential to ensure uniqueness. But, as we will see later in Theorem \ref{thm:superlf} (and see Remark \ref{rem:acsuperfl}), if $\mS$ is sufficiently regular and $V^\pm$ is constrained by \eqref{eq:selection} below, these conditions are superfluous.
\end{rem}

Each of $\sS\sD(V^-)$ and $\sS\sN(V^+)$ is a family of formulations indexed by the subspace $V^\pm$.  In the case that $V^\pm=H^{\pm 1/2}_\omS$ we will show (Corollaries \ref{DirEquivThmm2} and \ref{NeuEquivThmm2}) that $\sS\sD(V^-)$ and $\sS\sN(V^+)$ are equivalent to \SDws and \SNw, respectively. But we will also show (Theorems \ref{DirEquivThmm} and \ref{NeuEquivThmm} and Corollaries \ref{DirEquivThmm2} and \ref{NeuEquivThmm2}) that these problems are well-posed for every choice of $V^\pm$; that (Theorems \ref{thm:S61} and \ref{thm:s62}) distinct choices of $V^\pm$ lead to  distinct solutions; and  that (Remark \ref{rem:useful}) many of these solutions can be interpreted as valid physical solutions for distinct screens $\mS$ with the same closure.

But many choices of $V^\pm$ do not correspond to physical scattering problems, for example  choices where $V^\pm$ is finite-dimensional (though these formulations may be relevant as numerical approximations). Thus, in much (but not all) of our discussion below, we will constrain $V^\pm$ to satisfy a {\em physical selection principle},
\begin{equation} \label{eq:selection}
\tH^{\pm 1/2}(\mS^\circ) \subset V^\pm \subset H^{\pm 1/2}_\omS,
\end{equation}
intended to ensure that $\sD(V^-)$ and $\sN(V^+)$ are physically reasonable.
In \eqref{eq:selection} we understand $\tH^{\pm 1/2}(\mS^\circ)$ to mean $\{0\}$ if $\mS^\circ=\emptyset$ (in which case \eqref{eq:selection} provides no constraint).
The point of \eqref{eq:selection} is the following lemma.

\begin{lem} \label{lem:imp} If \eqref{eq:selection} holds and $v\in W^{1,\mathrm{loc}}(D;\Delta)$ then \eqref{a3} and \eqref{b3} imply that
\begin{equation} \label{eq:standard}
\gamma^\pm(\chi v)|_{\mS^\circ}=g_{\sD}|_{\mS^\circ} \quad \mbox{ and } \quad \partial_{\bn}^\pm(\chi v)|_{\mS^\circ}=g_{\sN}|_{\mS^\circ},
\end{equation}
respectively. Indeed, if $V^\pm = \tH^{\pm 1/2}(\mS^\circ)$, \eqref{a3} and \eqref{b3} are equivalent to these equations.
\end{lem}
\begin{proof} If \eqref{eq:selection} holds then \eqref{a3} and \eqref{b3} imply that $P_{V_*^{+}}(\gamma^\pm(\chi v)-g_{\sD})=0$ and $P_{V_*^{-}}(\partial_{\bn}^\pm(\chi v)-g_{\sN})=0$, respectively, with $V^\pm=\tH^{\pm 1/2}(\mS^\circ)$. This is equivalent to \eqref{eq:standard} (cf.\ proof of Lemma \ref{lem:equiv}).
\end{proof}

\begin{rem} \label{rem:acsuperfl} Lemma \ref{lem:imp} implies that, if $V^- = \tH^{-1/2}(\mS^\circ)$, then $\sD(V^-)$ is the standard formulation \Dsts augmented by the additional constraints \eqref{a1} and \eqref{a2}. Similarly, if $V^+ = \tH^{1/2}(\mS^\circ)$, then $\sN(V^+)$ is the standard formulation \Nsts augmented by the additional constraints \eqref{b1} and \eqref{b2}. We will show in Theorem \ref{thm:superlf} that, if $\mS$ is sufficiently regular, these additional constraints are superfluous, which in turn will imply (Theorem \ref{thm:wpst}) that the standard formulations \Dsts and \Nsts are then well-posed.
\end{rem}

\begin{cor} \label{cor:imp} If \eqref{eq:selection} holds and $v$ satisfies $\sD(V^-)$ ($\sN(V^+)$), then $v$ satisfies \Dsts with $\widetilde g_\sD = g_\sD|_{\mS^\circ}$ (\Nsts  with $\widetilde g_\sN = g_\sN|_{\mS^\circ}$).
\end{cor}

\begin{rem} \label{rem:bcs} Arguing as in the proof of Lemma \ref{lem:imp}, we have also that if \eqref{eq:selection} holds and $g_\sD$ and $g_\sN$ are given as in Problems $\sS\sD(V^-)$ and $\sS\sN(V^+)$, then, for every $\chi\in \scrD_{1,\omS}$,
$$
g_\sD|_{\mS^\circ} = -\gamma^\pm(\chi u^i)|_{\mS^\circ} \quad \mbox{and} \quad
g_{\sN}|_{\mS^\circ} =-\partial_\bn^\pm(\chi u^i)|_{\mS^\circ}.
$$
\end{rem}

Corollary \ref{cor:imp}, Remark \ref{rem:bcs}, and Lemma \ref{lem:equivclst} imply:

\begin{cor} \label{lem:satcl} If \eqref{eq:selection} holds, $u^s$ satisfies $\sS\sD(V^-)$ ($\sS\sN(V^+)$), and $(\Delta + k^2)u^i=0$ in a neighbourhood of $\omS$, then $u^s$ satisfies \SDcls (\SNcl).
\end{cor}

The following result is one half of a proof of well-posedness of $\sD(V^-)$ and $\sN(V^+)$ that we will complete in Theorems \ref{DirEquivThmm} and \ref{NeuEquivThmm} below.

\begin{thm} \label{lem:uni} Problems $\sD(V^-)$ and $\sN(V^+)$ (and hence also
$\sS\sD(V^-)$ and $\sS\sN(V^+)$) have at most one solution.
\end{thm}
\begin{proof} Suppose that $v$ satisfies $\sD(V^-)$ with $g_\sD=0$, and choose real-valued $\chi_*,\chi\in \scrD_{1,\omS}$ such that $\chi_*=1$ in a neighbourhood of the support of $\chi$. Then, by Green's first theorem \eqref{green1st} applied with $u$ and $v$ replaced by $\chi_* v$ and $\chi v$, respectively, and using \eqref{a1} and \eqref{eq:he},
$$
\langle [\partial_\bn v], \gamma^\pm (\chi v)\rangle_{H^{-1/2}(\Gamma_\infty)\times H^{1/2}(\Gamma_\infty)} = - \int_{D} \left( \nabla v \cdot \nabla (\chi \bar v) - k^2 \chi |v|^2 \right) \rd \bx.
$$
The duality pairing on the left hand side vanishes, since $[\partial_\bn v]\in V^-$ and $P_{V_*^{+}}\gamma^\pm(\chi v)=0$, so that $\gamma^\pm(\chi v)\in (V_*^+)^\perp$, i.e.\ is in the annihilator of $V^-$. Thus
\begin{equation} \label{eq:uni}
\Im \int_{D} \bar v\nabla v \cdot \nabla \chi \rd \bx = 0.
\end{equation}
Arguing similarly, but applying Green's first theorem (e.g.\ \cite[(3.4)]{CoKr:83}) in the bounded domain $B_R$
to $(1-\chi_\sharp) u$ and  $(1-\chi)u$ (both in $C^2(\R^n)$), where $\chi_\sharp\in \scrD_{1,\omS}$ is chosen so that $\chi=1$ in a neighbourhood of the support of $\chi_\sharp$ and $R$ is large enough so that the support of $\chi$ is in $B_R$, we see that
$$
\int_{\partial B_R} \bar v \frac{\partial v}{\partial r}\, \rd s = \int_{B_R} \left( \nabla v \cdot \nabla ((1-\chi) \bar v) - k^2 (1-\chi) |v|^2 \right) \rd \bx.
$$
Taking imaginary parts and using \eqref{eq:uni} we see that
$
\Im \int_{\partial B_R} \bar v \frac{\partial v}{\partial r}\, \rd s = 0
$
for all sufficiently large $R$. But this, together with the radiation condition \eqref{src}, implies that $\int_{\partial B_R} |v|^2 \rd s \to  0$
as $R\to \infty$, which implies by the Rellich lemma (e.g., \cite[Lemma 3.11]{CoKr:83}) that $v=0$ in $D$.

An almost identical argument proves uniqueness for $\sN(V^+)$.
\end{proof}

The following corollary is immediate from Lemma \ref{lem:equiv}.

\begin{cor} If  $u^s$ satisfies \SDw, then $u^s$ satisfies $\sS\sD(V^-)$ with $V^-=H^{-1/2}_\omS$. Similarly, if  $u^s$ satisfies \SNw, then $u^s$ satisfies $\sS\sN(V^+)$ with $V^+=H^{1/2}_\omS$.
\end{cor}

Combining this result with Theorems \ref{thm:bvp_wwp} and \ref{lem:uni}, we obtain:

\begin{thm} \label{thm:Dexist} For $V^-=H^{-1/2}_\omS$, Problem $\sS\sD(V^-)$ has exactly one solution  which is the unique solution of \SDw. Similarly, for $V^+=H^{1/2}_\omS$, Problem $\sS\sN(V^+)$ has exactly one solution  which is the unique solution of \SNw.
\end{thm}

\subsection{BIEs, well-posedness, and equivalence of formulations} \label{subsec:bies}
We now study the reformulation as BIEs of the various BVPs we have introduced above. We will also use these BIEs to complete proofs of well-posedness and to complete our study of the connections between the various formulations.

The operators in these BIEs will be the single layer and hypersingular operators, $S:V^-\to V^+_*$ and $T:V^+\to V^-_*$, that we introduced in \eqref{eq:BIOs}, where, as above,  $V^\pm$ is some closed subspace of $H^{\pm 1/2}_\omS\subset H^\pm(\Gamma_\infty)$, and $V^\mp_*\subset H^{\mp 1/2}(\Gamma_\infty)$ the natural realisation of its dual space. These BIOs may seem exotic, especially in cases where $\omS$ has empty interior or even zero Lebesgue measure, but we emphasise that these BIOs are nothing but restrictions to subspaces of operators that are completely familiar in screen scattering problems. Explicitly, from \eqref{SddSd} and \eqref{STdag},
$$
S = P_{V_*^+}E^+S_\dag|_{V^-} \quad \mbox{and} \quad T = P_{V_*^-}E^-T_\dag|_{V^+},
$$
where $S_\dag:\tH^{-1/2}(\Gamma_\dag)\to H^{1/2}(\Gamma_\dag)$ and $T_\dag:\tH^{1/2}(\Gamma_\dag)\to H^{-1/2}(\Gamma_\dag)$ are the familiar operators defined by \eqref{Sstardefs}, \eqref{SdagDef} and \eqref{TdagDef}, and $E^\pm:H^{\pm 1/2}(\Gamma_\dag)\to H^{\pm 1/2}(\R^n)$ are the minimum norm extension operators introduced in \eqref{exts}.

We first reformulate as BIEs the Dirichlet and Neumann BVPs $\sD(V^-)$ and $\sN(V^+)$, and prove well-posedness of these BVPs via well-posedness of the BIEs. We omit the proof of Theorem \ref{NeuEquivThmm} which is almost identical to that of Theorem \ref{DirEquivThmm}. Recall that the sesquilinear forms $a_S$ and $a_T$ are defined in \eqref{eq:sesD}, \eqref{eq:sesN}, and \eqref{eq:sesD2}.
\begin{thm}
\label{DirEquivThmm}
Problem $\sD(V^-)$ has a unique solution, which satisfies
\begin{align}
\label{eqn:SLPRep}
v(\bx )= -\cS\left[\partial_\bn v\right](\bx), \qquad\bx\in D,
\end{align}
with $[\partial_\bn v]\in V^{-}$ the unique solution of the BIE
\begin{align}
\label{BIE_sl}
S\left[\partial_\bn v\right] = -g_{\sD}.
\end{align}
Further, \eqref{BIE_sl} is equivalent to the variational problem: find $[\partial_\bn v]\in V^{-}$ such that
\begin{align}
\label{BIE_slw}
a_S\left([\partial_\bn v],\psi\right) = -\langle g_{\sD},\psi\rangle_{H^{1/2}(\Gamma_\infty)\times H^{-1/2}(\Gamma_\infty)}, \quad \mbox{for all }\psi \in V^{-}.
\end{align}
For every $V^-$ satisfying \eqref{eq:selection}, the solution of $\sD(V^-)$ is a solution of \Dsts with $\widetilde g_\sD= g_\sD|_{\mS^\circ}$.
\end{thm}
\begin{proof} We have seen in Theorem \ref{lem:uni} that $\sD(V^-)$ has at most one solution. Further, we have observed above \eqref{invert} that $S:V^-\to V^+_*$ is coercive and so invertible. Defining $\phi := -S^{-1}g_\sD\in V^-$ and $v:= -\cS\phi$, it is immediate from Theorem \ref{LayerPotRegThm} and \eqref{eq:BIOs} that $v$ satisfies $\sD(V^-)$ with $[\partial_\bn v]=\phi$. The equivalence of \eqref{BIE_sl} and \eqref{BIE_slw} is clear from \eqref{rest1} and the fact that $V^+_*$ is a realisation of the dual space of $V^-$ via the duality pairing on the left-hand side of \eqref{rest1}. The last sentence follows from Corollary \ref{cor:imp}.
\end{proof}

\begin{thm}
\label{NeuEquivThmm}
Problem $\sN(V^+)$ has a unique solution, which satisfies %
\begin{align}
\label{eqn:DLPRep}
 v(\bx ) =  \cD[u](\bx), \qquad\bx\in D,
\end{align}
with $[v]\in V^{+}$ the unique solution of the BIE
\begin{align}
\label{BIE_hyp}
T[u] = g_{\sN}.
\end{align}
Further, \eqref{BIE_hyp} is equivalent to the variational problem: find $[v]\in V^{+}$ such that
\begin{align}
\label{BIE_hypw}
a_T\left([v],\psi\right) = \langle g_{\sN},\psi\rangle_{H^{-1/2}(\Gamma_\infty)\times H^{1/2}(\Gamma_\infty)}, \quad \mbox{for all }\psi \in V^{+}.
\end{align}
For every $V^+$ satisfying \eqref{eq:selection}, the solution of $\sN(V^+)$ is a solution of \Nsts with $\widetilde g_\sN= g_\sN|_{\mS^\circ}$.
\end{thm}

The following corollary is immediate from Theorem \ref{DirEquivThmm}, noting that $[\partial_\bn u^i]=0$ as observed in the proof of Lemma \ref{lem:equiv}. The exception is the penultimate sentence which is a restatement of Theorem \ref{thm:Dexist}, and the last sentence which follows from Lemma \ref{lem:satcl}.
\begin{cor}
\label{DirEquivThmm2}
Problem $\sS\sD(V^-)$ has a unique solution, which satisfies (where $u:=u^i+u^s$)
\begin{align}
\label{eqn:SLPRep2}
u(\bx)= u^i(\bx)-\cS\left[\partial_\bn u\right](\bx), \qquad\bx\in D,
\end{align}
with $[\partial_\bn u]\in V^{-}$ the unique solution of the BIE
\begin{align}
\label{BIE_sl2}
S\left[\partial_\bn u\right] = P_{V_*^{+}}\gamma^\pm(\chi u^i),
\end{align}
where $\chi\in \scrD_{1,\omS}$ is arbitrary.
Further,
\eqref{BIE_sl2} is equivalent to the variational problem: find $[\partial_\bn u]\in V^{-}$ such that
\begin{align}
\label{BIE_slw2}
a_S\left([\partial_\bn u],\psi\right) = \langle \gamma^\pm(\chi u^i),\psi\rangle_{H^{1/2}(\Gamma_\infty)\times H^{-1/2}(\Gamma_\infty)}, \quad \mbox{for all }\psi \in V^{-}.
\end{align}
If $V^-=H^{-1/2}_\omS$ then $\sS\sD(V^-)$ and \SDws have the same unique solution. For every $V^-$ satisfying \eqref{eq:selection}, the solution of $\sS\sD(V^-)$ is a solution of \SDcls if $(\Delta + k^2)u^i=0$ in a neighbourhood of $\omS$.
\end{cor}

Similarly, the following corollary follows from Theorem \ref{NeuEquivThmm}, with the penultimate sentence a consequence of Theorem \ref{thm:Dexist}, and the last sentence a consequence of Lemma \ref{lem:satcl}.
\begin{cor}
\label{NeuEquivThmm2}
Problem $\sS\sN(V^+)$ has a unique solution, and this solution satisfies (where $u:= u^i+u^s$)
\begin{align}
\label{eqn:DLPRep2}
 u(\bx ) = u^i(\bx) + \cD[u](\bx), \qquad\bx\in D,
\end{align}
with $[u]\in V^{+}$ the unique solution of the BIE
\begin{align}
\label{BIE_hyp2}
T[u] = -P_{V_*^{-}}\partial_\bn^\pm(\chi u^i),
\end{align}
where $\chi\in \scrD_{1,\omS}$ is arbitrary.
Further,
\eqref{BIE_hyp2} is equivalent to the variational problem: find $[u]\in V^{+}$ such that
\begin{align}
\label{BIE_hypw2}
a_T\left([u],\psi\right) = -\langle \partial_\bn^\pm(\chi u^i),\psi\rangle_{H^{-1/2}(\Gamma_\infty)\times H^{1/2}(\Gamma_\infty)}, \quad \mbox{for all }\psi \in V^{+}.
\end{align}
If $V^+=H^{1/2}_\omS$ then $\sS\sN(V^+)$ and \SNws have the same unique solution. For every $V^+$ satisfying \eqref{eq:selection}, the solution of $\sS\sN(V^+)$ is a solution of \SNcls if $(\Delta + k^2)u^i=0$ in a neighbourhood of $\omS$.
\end{cor}

\section{When is the scattered field just $u^s=0$?} \label{sec:zero}

From a variety of perspectives, including that of inverse scattering, a fundamental question is: does the incident field `see' the screen, by which we mean simply: is $u^s\neq 0$? %

We first note from Corollaries \ref{DirEquivThmm2} and \ref{NeuEquivThmm2} that a necessary condition for the solution of $\sS\sD(V^-)$ ($\sS\sN(V^+)$) to be non-zero is $V^-\neq\{0\}$ ($V^+\neq \{0\}$). And, trivially, there exists a subspace $V^\pm$ of $H^{\pm 1/2}_\omS$ with $V^\pm\neq \{0\}$ if and only if $H^{\pm 1/2}_\omS\neq \{0\}$. So one relevant question is: for which compact sets $K\subset \Gamma_\infty$ is $H^{\pm 1/2}_K= \{0\}$?  That is, using the terminology introduced below \eqref{HsTdef}: for which compact sets $K$ is $K$ $\pm 1/2$-null? We address this question in Theorem \ref{thm:null},  which will be a key tool in much of our later analysis, using results from \cite{HewMoi:15}.

Before stating the theorem we note that, as will be of no surprise to readers familiar with potential theory (e.g., \cite{Kellogg}, and \cite[Theorem 2.7.4]{AdHe}),
for the Dirichlet problem a key role is played by the capacity, defined for a compact set $F\subset \R^n$ by $\mathrm{cap}(F) := \inf \{\|u\|^2_{H^1(\R^n)}\}$, where the infimum is over all $u\in \scrD(\R^n)$ such that $u\geq 1$ in a neighbourhood of $F$. For an open set $\Omega\subset \R^n$, and for an arbitrary Borel set $E\subset \R^n$,
\begin{align*}
\label{}
\mathrm{cap}(\Omega):=\sup_{\substack{F\subset \Omega\\ F \textrm{ compact}}} \mathrm{cap}(F), \quad \mathrm{cap}(E):=\inf_{\substack{\Omega\supset E\\ \Omega \textrm{ open}}} \mathrm{cap}(\Omega).
\end{align*}
This last definition for arbitrary Borel sets applies, in particular, in the cases $E$ compact and $E$ open, for which it coincides with the immediately preceding definitions for these cases (as shown e.g.\ in  \cite[\S3]{HewMoi:15}). Also, we note that in Theorem \ref{thm:null} and the rest of the paper we use the notation $m(E)$ to denote the $(n-1)$-dimensional Lebesgue measure of $E$, for measurable $E\subset \Gamma_\infty$.
Finally, we remark that illustrations of the last sentence of the theorem are given in Examples \ref{ex:sier}-\ref{ex:swiss}.

\begin{thm} \label{thm:null} Let $E$ be a Borel subset of $\Gamma_\infty$. Then:
\begin{enumerate}
  \item[(a)] $E$ is $-1/2$-null if and only if $\mathrm{cap}(E)=0$;
  \item[(b)] if $E$ is closed then $E$ is $-1/2$-null if and only if $W^1_0(\R^n\setminus E)=W^1(\R^n)$;
  \item[(c)] if $E^\circ$ is non-empty then $E$ is not $\pm 1/2$-null;
  \item[(d)] if $\dimH(E) < n-2$ then $E$ is $-1/2$-null, and if  $\dimH(E) > n-2$ then $E$ is not $-1/2$-null;
  \item[(e)] if $m(E)=0$ or $E$ is $-1/2$-null, then $E$ is $1/2$-null;
  \item[(f)] if $E=\partial \Omega$ and $\Omega$ is in the algebra of subsets of $\R^n$ generated by all $C^0$ open sets, then $E$ is $1/2$-null: if $\Omega$ is in the algebra generated by all Lipschitz open sets, then $E$ is $-1/2$-null;
  \item[(g)] if $E$ is countable, then $E$ is $\pm 1/2$-null; if $E$ is a countable union of Borel $-1/2$-null sets, then $E$ is $-1/2$-null;
  \item[(h)] if $E$ is $1/2$-null, $F\subset \Gamma_\infty$ is $1/2$-null, and $F$ has no limit points in $E\setminus F$ (this holds, in particular, if $F$ is closed), then $E\cup F$ is $1/2$-null.
\end{enumerate}
Further, there exists a compact set $F\subset \Gamma_\infty$ with $F^\circ=\emptyset$ and $F$ not $1/2$-null, and a set $F$ with $m(F)=0$ and $F$ not $-1/2$-null.
\end{thm}
\begin{proof} A proof of (a) can be found in \cite[\S4]{ChaHewMoi:13}; part (b) follows from (a) and results in \cite[\S13.2]{Maz'ya}; and the remaining results apart from (f) are proved in \cite{HewMoi:15}. Let $\cI^\pm$ denote the set of subsets $A\subset \R^n$ for which $\partial A$ is $\pm 1/2$-null. Note that $\R^n\in \cI^+$; and $\partial A=\partial (A^c)$, so that $A^c\in \cI^+$ if $A\in \cI^+$; and $A\cup B\in \cI^+$ if $A,B\in \cI$ by (h), as $\partial(A\cup B)\subset (\partial A)\cup(\partial B)$. Thus $\cI^+$ is an algebra of subsets of $\R^n$. As \cite{HewMoi:15} $\Omega\in \cI^+$ if $\Omega$ is $C^0$, the algebra $\cI^+$ contains that generated by the $C^0$ open sets. Similarly, but using (g) in place of (h), we see that $\cI^-$ is an algebra, and that $\cI^-$ contains the algebra generated by the Lipschitz open sets as these sets are in $\cI^-$ by \cite{HewMoi:15}.
\end{proof}

Another approach to the question `Is $u^s=0$?' is to observe, since $\sS\sD(V^-)$ and $\sS\sN(V^+)$ are well-posed, that the solutions to $\sS\sD(V^-)$ and $\sS\sN(V^+)$ are $u^s=0$ if and only if  $g_\sD$ and $g_\sN$, given by \eqref{gd} and \eqref{gn}, vanish. This clearly happens for some choices of incident field $u^i$: for example,  $g_\sD=g_\sN=0$ if $u^i=0$ in a neighbourhood of $\omS$. But we will see, in Theorems \ref{cor:Vm} and \ref{strength}, that for a large class of incident fields of interest, including the plane wave \eqref{eq:pw}, this does not happen, as long as $V^\pm\neq \{0\}$.

We first prove two preliminary lemmas. In both we assume that $V$ is a closed subspace of $ H^s(\Gamma_\infty)$, for some $s\in\R$, and, as usual, we denote by $V^a$ the annihilator of $V$ in $H^{-s}(\Gamma_\infty)$ given by \eqref{ann}, by $V^*$ the dual space realisation $V^*=(V^a)^\perp\subset H^{-s}(\Gamma_\infty)$ (with duality pairing \eqref{dualgen}), and by $P_{V^*}$ orthogonal projection in $H^{-s}(\Gamma_\infty)$ onto $V^*$.
In proving the lemmas we will use the fact that, if $s\in \R$ and $u\in H^{s}(\Gamma_\infty)$ is compactly supported, then $\hat u$ is an entire function and
\begin{equation} \label{eq:FT}
(2\pi)^{(n-1)/2} \hat u(\bxi) = \overline{\left\langle \chi e_{\bxi},u \right\rangle_{H^{-s}(\Gamma_\infty)\times H^{s}(\Gamma_\infty)}} = \left\langle u,\chi e_{\bxi}\right\rangle_{H^{s}(\Gamma_\infty)\times H^{-s}(\Gamma_\infty)}, \qquad \bxi\in \Gamma_\infty,
\end{equation}
for every $\chi\in \scrD(\Gamma_\infty)$ such that $\chi=1$ in a neighbourhood of $\supp{u}$, where $e_{\bxi}\in C^\infty(\Gamma_\infty)$ is defined by $e_{\bxi}(\bx) = \exp(\ri \bxi\cdot  \bx)$, for $\bx,\bxi\in \Gamma_\infty$.

\begin{lem} \label{lem:projV*1} Suppose that $s\in\R$ and that $V\subset H^s(\Gamma_\infty)$ is a closed subspace of $H^s_\omS$. Suppose that there exists $v\in V$, with $v\neq 0$ and $\chi v \in V$ for all $\chi\in \scrD(\Gamma_\infty)$. Suppose further that $g\in \scrD(\Gamma_\infty)$ and that $g(\bx)\neq 0$ for $\bx\in \omS$. Then $P_{V^*}g\neq 0$.
\end{lem}
\begin{proof}
We have noted above that $\hat v$ is an entire function and, since $v\neq 0$, $\hat v(\bxi)\neq 0$, for some $\bxi\in \Gamma_\infty$.
Choose $\chi\in \scrD_{1,\omS}$ with $g\neq 0$ on the support of $\chi$, and define $\phi\in V$ by $\phi := (e_{-\bxi}\bar \chi/\bar g)v$.  Then, abbreviating $\left\langle \cdot, \cdot \right\rangle_{H^{-s}(\Gamma_\infty)\times H^{s}(\Gamma_\infty)}$ by $\left\langle \cdot, \cdot \right\rangle$,
$$
\langle P_{V^*}g,\phi\rangle = \langle g,\phi\rangle = \langle e_{\bxi} \chi, v\rangle = (2\pi)^{(n-1)/2}\overline{\hat v(\bxi)}\neq 0,
$$
so that $P_{V^*}g\neq 0$.
\end{proof}

The proof of the next lemma uses the result that if the zero set of an entire function of one complex variable has a limit point, then the function is identically zero. This implies that if the set of real zeros of an entire function $f(s,t)$ of two complex variables has positive two-dimensional Lebsegue measure, then the function is identically zero. (For if $\chi(s,t)$ is the characteristic function of the zero set in $\R^2$ and $\int\int \chi(s,t) \rd s\rd t >0$ then, by Tonelli's theorem, $\int \chi(s,t) \rd t >0$ for all $s\in E\subset \R$, with $E$ of positive (one-dimensional) measure. This implies that $f(s,\cdot) = 0$ for $s\in E$ by the one-dimensional result applied with $s$ fixed, and, by the one-dimensional result applied with $t$ fixed, we deduce that $f(s,t)\equiv 0$.)

\begin{lem} \label{lem:nonvanish} Suppose that $s\in\R$, and that $\{0\}\neq V$ is a closed subspace of $H^s_\omS$.
Let $\chi\in \scrD_{1,\omS}$ and define $g\in \scrD(\Gamma_\infty)$ by $g(\bx) = \chi(\bx) \exp(\ri k \bd \cdot \bx)$, for $\bx\in \Gamma_\infty$.
Then $P_{V^*}g\neq 0$ for almost all $\bd\in\mathbb{S}^n:= \{\bd\in \R^n:|\bd|=1\}$, indeed for all except finitely many $\bd$ if $n=2$.
\end{lem}
\begin{proof} We prove the result by contradiction. Where $g(\bx) := \chi(\bx) \exp(\ri k \bd \cdot \bx)$, for $\bx\in \Gamma_\infty$, suppose that $P_{V^*}g=0$  for all $\bd$ in some  subset of $\mathbb{S}^n$ of positive surface measure, which implies that $P_{V^*}(\chi e_{\bxi})=0$ for all $\bxi$ in some $\Xi\subset \{\bxi\in \Gamma_\infty:|\bxi|\leq k\}$ which has positive Lebesgue measure. Then, for all $v\in V$ and all $\bxi\in \Xi$, it follows from \eqref{eq:FT} that
\[
(2\pi)^{(n-1)/2}\overline{\hat v(\bxi)} = \left\langle \chi e_{\bxi},v\right\rangle_{H^{-s}(\Gamma_\infty)\times H^{s}(\Gamma_\infty)} =
\left\langle P_{V^*}(\chi e_{\bxi}),v\right\rangle_{H^{-s}(\Gamma_\infty)\times H^{s}(\Gamma_\infty)} = 0,
\]
so that $\hat v=0$ (and hence $v=0$) since $\hat v$ is entire. Thus $V=\{0\}$, a contradiction. In the case $n=2$ we achieve the same contradiction just assuming that $P_{V^*}g=0$  for infinitely many  $\bd\in\mathbb{S}^n$.
\end{proof}

\begin{rem} \label{rem:all are subspaces} Regarding the application of the above lemmas, it is important to note that any closed subspace $W\subset H^{-s}(\Gamma_\infty)$ can be written (uniquely) as $W=V^*:=(V^a)^\perp$, for some closed subspace $V\subset H^s(\Gamma_\infty)$, and that $V\neq  \{0\}$ if $W\neq \{0\}$. Explicitly, $V=j^{-1}(W)$, where $j:H^s(\Gamma_\infty)=H^s(\R^{n-1})\to H^{-s}(\Gamma_\infty) = H^{-s}(\R^{n-1})$ is the unitary isomorphism introduced above \eqref{ann}.
\end{rem}

Lemma \ref{lem:nonvanish} implies that, for any non-zero $V^\pm$ (even where $V^\pm$ is only one-dimensional), the solutions of $\sS\sD(V^-)$  and $\sS\sN(V^+)$ are non-zero for almost all incident plane waves.

\begin{thm} \label{cor:Vm} Let $V^-\neq \{0\}$ ($V^+\neq \{0\}$), and let $u^s$ be the solution to $\sS\sD(V^-)$ ($\sS\sN(V^+)$), with $g_\sD$ given by \eqref{gd} ($g_\sN$ given by \eqref{gn}), where $u^i$ is the plane wave \eqref{eq:pw}. Then, for almost all plane wave directions $\bd\in\mathbb{S}^n$ (all except finitely many if $n=2$), it holds that $u^s\neq 0$.
\end{thm}
\begin{proof} If $u^i$ is given by \eqref{eq:pw} and $g_\sD$ and $g_\sN$ by \eqref{gd} and \eqref{gn}, then $g_\sD=-P^+_*g$ and $g_\sN = -\ri k d_n P^-_*g$, where $g$ is as given in Lemma \ref{lem:nonvanish} and $d_n$ is the component of $\bd$ in the $x_n$ direction. It follows from Lemma \ref{lem:nonvanish} that $g_\sD\neq 0$ and $g_\sN\neq 0$ for almost all $\bd\in\mathbb{S}^n$ (all except finitely many if $n=2$), and the result follows.
\end{proof}

The above result can be strengthened, using Lemma \ref{lem:projV*1}, in the case $V^\pm = H^{\pm 1/2}_\omS$, indeed in the case that $V^\pm$ satisfies \eqref{eq:selection}, provided in this latter case that $\mS^\circ \neq \emptyset$.

\begin{thm} \label{strength} Suppose that $u^i$ is $C^\infty$ in a neighbourhood of $\omS$, and let $u^s$ be the solution of \SDws (\SNw). Then, if $\omS$ is not $-1/2$-null and $u^i(\bx)\neq 0$ for all $\bx\in \omS$ (not $1/2$-null and $\frac{\partial u^i}{\partial x_n}(\bx)\neq 0$ for $\bx\in \omS$), it holds that $u^s\neq 0$. In particular, these conditions on $u^i$ are satisfied by the incident plane wave \eqref{eq:pw} (provided $d_n\neq 0$ for the Neumann problem \SNw), and by the incident cylindrical/spherical wave \eqref{eq:cw} (provided $y_n\neq 0$ for the Neumann problem \SNw). Further, if $\mS^\circ \neq \emptyset$, and $V^\pm$ satisfies \eqref{eq:selection}, then the above statements hold with \SDws and \SNws replaced by $\sS\sD(V^-)$ and $\sS\sN(V^+)$, respectively.
\end{thm}
\begin{proof} If $v\in H^{\pm1/2}_\omS$ then $\chi v\in H^{\pm1/2}_\omS$, for all $\chi\in \scrD(\Gamma_\infty)$, so that the first part of this result follows from Lemma \ref{lem:projV*1}. Clearly the conditions on $u^i$ are satisfied by \eqref{eq:pw} and \eqref{eq:cw} (recalling that the Hankel function $H_\nu^{(1)}(t)\neq 0$ for $t>0$ and $\nu=0,1$). If $\mS^\circ\neq \emptyset$ and \eqref{eq:selection} holds then $V^\pm \supset\tH^{\pm 1/2}(\mS^\circ)\neq \{0\}$ and if $v\in \tH^{\pm 1/2}(\mS^\circ)\subset V^\pm$ then $\chi v\in \tH^{\pm 1/2}(\mS^\circ)\subset V^\pm$, for all $\chi\in \scrD(\Gamma_\infty)$, so that the rest of the result also follows from Lemma \ref{lem:projV*1}.
\end{proof}

The following theorem summarises, and/or follows immediately from the results in, Theorems \ref{thm:null}, \ref{cor:Vm} and \ref{strength}. We defer discussion of examples illustrating this theorem until \S\ref{sec:fract}.

\begin{thm} \label{thm:main} $\omS$ is $-1/2$-null if and only if $\mathrm{cap}(\omS)=0$, which holds if and only if $W_0^1(D)=W^1(\R^n)$. If $\omS$ is $-1/2$-null, which holds in particular if $n=3$ and $\dimH(\omS) < 1$, then the solution to \SDws and to $\sS\sD(V^-)$ is $u^s=0$. If $\omS$ is not $-1/2$-null, equivalently $\mathrm{cap}(\omS)>0$, which holds in particular if $\dimH(\omS) > n-2$, and certainly if $\omS^\circ$ is non-empty, then: (i) if $u^i$ satisfies the conditions of Theorem \ref{strength} for the Dirichlet case, in particular if $u^i$ is the incident plane wave \eqref{eq:pw} or the cylindrical/spherical wave \eqref{eq:cw},  then the solution $u^s$ to \SDws does not vanish, and nor does the solution to $\sS\sD(V^-)$ as long as $V^-$ satisfies \eqref{eq:selection} and $\mS^\circ$ is non-empty; (ii) if $u^i$ is the plane wave \eqref{eq:pw} and $V^-\neq \{0\}$ then the solution $u^s$ to $\sS\sD(V^-)$ is non-zero for almost all incident directions $\bd$ (all but finitely many directions $\bd$ if $n=2$).

Similarly, if $\omS$ is $1/2$-null, in particular if $m(\omS)=0$, then the solution to \SNws and to $\sS\sN(V^+)$ is $u^s=0$. If $\omS$ is not $1/2$-null, in particular if $\omS^\circ$ is non-empty (though this is not necessary, see Example \ref{ex:swiss}),  then: (i) if $u^i$ satisfies the conditions of Theorem \ref{strength} for the Neumann case,  then the solution $u^s$ to \SNws does not vanish, and nor does the solution to $\sS\sN(V^+)$ as long as $V^+$ satisfies \eqref{eq:selection} and $\mS^\circ$ is non-empty; (ii) if $u^i$ is the plane wave \eqref{eq:pw} and $V^+\neq \{0\}$ then the solution $u^s$ to $\sS\sN(V^+)$ is non-zero for almost all incident directions $\bd$ (all but finitely many directions if $n=2$).
\end{thm}

\section{Do all our formulations have the same solution?} \label{subsec:embarr}

We focus in this section on the formulations $\sD(V^-)$, $\sN(V^+)$, $\sS\sD(V^-)$, and $\sS\sN(V^+)$ introduced in \S\ref{subsec:bvpsnovel}, with $V^\pm$ satisfying the physical selection principle \eqref{eq:selection}, and on the standard formulations introduced in \S\ref{subsec:bvps}, addressing the question of the section title. We will show that either the formulations $\sS\sD(V^-)$ and $\sS\sN(V^+)$ satisfying \eqref{eq:selection} coincide for all choices of $V^+$ and $V^-$ satisfying \eqref{eq:selection}, or, in each case, the cardinality of the set of distinct formulations is that of the continuum. Further we show that, in the specific case of plane wave incidence, for almost all directions of incidence, there are, in both the sound-soft and sound-hard cases, infinitely many distinct solutions $u^s$ to these formulations.

The main results of the section are Theorems \ref{thm:S61} and \ref{thm:s62}. For their proof we shall appeal to a number of preliminary results. The first, an approximation lemma, is used to prove Lemma \ref{lem:bigspace}, and is a special case of \cite[Lemma 3.22]{ChaHewMoi:13}.

\begin{lem} \label{lem:vj} Suppose that $N\in \N$ and $\bx_1,...,\bx_N\in \R^n$ are distinct.
Then there exists a family $(v_j)_{j\in\N}\subset C^\infty(\R^n)$ such that: for all $j\in \N$,
$v_j(\bx) = 0$, if $|\bx-\bx_i|< 1/j$ for some $i\in \{1,...,N\}$; for all $\phi\in H^s(\R^n)$ with $|s|\leq 1/2$, $\|v_j\phi-\phi\|_{H^s(\R^n)} \to 0$ as $j\to\infty$.
\end{lem}

\begin{lem} \label{lem:bigspace} If $|s|\leq 1/2$ and $\tH^{s}(\mS^\circ)\neq H^{s}_\omS$, then the quotient space $H^{s}_\omS/\tH^{s}(\mS^\circ)$ is an infinite-dimensional separable Hilbert space.
\end{lem}
\begin{proof} It is standard that the quotient space $H^{s}_\omS/\tH^{s}(\mS^\circ)$ is a Hilbert space (e.g., \cite[\S2.1]{ChaHewMoi:13}), and it is separable as $H^{s}(\R^n)\supset H^{s}_\omS$ is separable.

Suppose now that $|s|\leq 1/2$ and $\tH^{s}(\mS^\circ)\neq H^{s}_\omS$. We show first that for every $v\in H^{s}_\omS \setminus \tH^{s}(\mS^\circ)$ there exists a set $K\subset \omS$ with at least countably many points such that each point $\bx\in K$ has the following property:
\begin{equation} \label{eq:proper}
\mbox{for all $\epsilon>0$ there exists $\chi\in \scrD(B_\epsilon(\bx))$ with $\chi v\in H^{s}_\omS \setminus \tH^{s}(\mS^\circ)$.}
\end{equation}

For suppose this is not true. Then, for some $v\in H^{s}_\omS \setminus \tH^{s}(\mS^\circ)$ the set $X$ of points $\bx\in \omS$ with this property is finite.  Choose a sequence $(v_\ell)_{\ell=1}^\infty\subset C^\infty(\R^n)$ as follows: set $v_\ell\equiv 1$ for $\ell\in \N$ if $X= \emptyset$; if $X=\{\bx_1,...,\bx_N\}$, for distinct $\bx_1,...,\bx_N\in \omS$, choose $(v_\ell)$ to have the properties of Lemma \ref{lem:vj}. For every $\bx\in K_- :=\omS\setminus X$ there exists $\epsilon(\bx)>0$ such that  $\chi v\in \tH^{s}(\mS^\circ)$ for all $\chi\in \scrD(B_{\epsilon(\bx)}(\bx))$. Now, for each $\ell\in \N$, $\{B_{\epsilon(\bx)}(\bx):\bx\in K_-\}$ is an open cover for $\supp(v_\ell v)$ which has a finite subcover $\{B_{\epsilon(\by_j)}(\by_j):j\in \{1,...,M\}\}$, for some finite subset $\{\by_1,...,\by_M\}\subset K_-$. Let $\chi_1,..., \chi_M$ be a partition of unity for $\supp(v_\ell v)$ such that $\chi_j \in \scrD(B_{\epsilon(\by_j)}(\by_j))$, for $j=1,...,M$; such a partition of unity exists by  \cite[Theorem 2.17]{grubb}. Then $\chi_jv_\ell v \in \tH^{s}(\mS^\circ)$, for $j=1,...,M$, so that $v_\ell v= \sum_{j=1}^M \chi_j v_\ell v \in \tH^{s}(\mS^\circ)$. But this implies, taking the limit $\ell\to\infty$, since $\tH^{s}(\mS^\circ)$ is closed, that $v\in \tH^{s}(\mS^\circ)$, a contradiction.

Now suppose that $v\in H^{s}_\omS \setminus \tH^{s}(\mS^\circ)$ and let $K = \{\bx_1,\bx_2,...\}\subset \omS$, with distinct points $\bx_j$, be such that each point $\bx\in K$ has the property \eqref{eq:proper}. Let
$$
W := \{\chi v + \tH^{s}(\mS^\circ): \chi\in \scrD(\R^n)\}.
$$
Then $W$ is a linear subspace of $H^{s}_\omS/\tH^{s}(\mS^\circ)$. Further, for every $N\in \N$, choosing $\epsilon>0$ such that $\epsilon < |\bx_j-\bx_\ell|/3$, for $j,\ell=1,...,N$, and $\chi_j\in\scrD(B_\epsilon(\bx_j))$ with $\chi_j v\in H^{s}_\omS \setminus \tH^{s}(\mS^\circ)$ for $j=1,...,N$, it is clear that
$$
\{\chi_1v+ \tH^{s}(\mS^\circ),...,\chi_Nv+ \tH^{s}(\mS^\circ)\}\subset W
$$
is linearly independent, for if $w:= \sum_{j=1}^N \alpha_j \chi_j v \in \tH^{s}(\mS^\circ)$, for some $\alpha_1,...,\alpha_N\in \C$, then, for $\ell\in \{1,...,N\}$, choosing $\chi\in \scrD(\R^n)$ such that $\chi=1$ in a neighbourhood of $B_\epsilon(\bx_\ell)$ and $\supp(\chi)\subset B_{2\epsilon}(\bx_\ell)$, we see that $\tH^{s}(\mS^\circ)\ni\chi w = \alpha_\ell \chi \chi_\ell v = \alpha_\ell \chi_\ell v$, so $\alpha_\ell=0$ since $\chi_\ell v\not\in \tH^{s}(\mS^\circ)$. Thus $\dim(W) \geq N$, for each $N$, so that $H^{s}_\omS/\tH^{s}(\mS^\circ)$ is infinite-dimensional.
\end{proof}

The following lemma uses cardinal arithmetic from ZFC set theory \cite{Halmos}.
\begin{lem} \label{lem:card} Let $H$ be a separable Hilbert space of dimension at least two. Then the cardinality of the set of closed subspaces of $H$ is $\mathfrak{c}$, the cardinality of $\R$.
\end{lem}
\begin{proof}
Suppose $H$ is a separable Hilbert space. Then its cardinality is at least $\mathfrak{c}$ if its dimension is at least two, for if $v_1,v_2\in H$ are orthogonal, $\{\alpha(\cos(\theta)\, v_1 + \sin(\theta)\, v_2):\alpha\in \C\}$ is a distinct closed subspace of $H$ for each $0\leq \theta<\pi$. Further, $H$ has a countable orthonormal basis, so that the cardinality of $H$ itself is no larger than that of the set of sequences of complex numbers, which is $\mathfrak{c}$, as $|\C^{\N}| = (2^{\aleph_0})^{\aleph_0}=2^{\aleph_0}=\mathfrak{c}$.  Finally, since there is an injection $H\to \R$,  and each closed subspace $V$ is characterised by a countable set of orthonormal basis vectors in $H$, the cardinality of the set of closed subspaces of $H$ is no larger than $|\R^{\N}| = \mathfrak{c}$.
\end{proof}

\begin{lem} \label{lem:cardc} If $\tH^{\pm 1/2}(\mS^\circ)\neq H^{\pm 1/2}_\omS$, then the set of closed subspaces $V^\pm$ satisfying \eqref{eq:selection} has cardinality $\mathfrak{c}$.
\end{lem}
\begin{proof}
$H^{\pm 1/2}_\omS/\tH^{\pm 1/2}(\mS^\circ)$ is unitarily isomorphic to $\mathcal{W}$, the orthogonal complement of $\tH^{\pm 1/2}(\mS^\circ)$ in $H^{\pm 1/2}_\omS$, so the set of closed subspaces of $\mathcal{W}$ has cardinality $\mathfrak{c}$ by Lemmas \ref{lem:bigspace} and \ref{lem:card}. The result follows as there is a one-to-one correspondence between the set of closed subspaces of $\mathcal{W}$ and the closed subspaces $V^\pm$ of $H^{\pm 1/2}_\omS$ that satisfy \eqref{eq:selection}, given by $L\mapsto \tH^{\pm 1/2}(\mS^\circ)\oplus L$, for $L$ a closed subspace of $\mathcal{W}$.
\end{proof}

We are now ready to state and prove the main theorems in this section.

\begin{thm} \label{thm:S61} If $\tH^{-1/2}(\mS^\circ)= H^{-1/2}_\omS$ and $V^-$ satisfies \eqref{eq:selection} then $V^-=H^{-1/2}_\omS$, so that there is only one formulation $\sD(V^-)$, and one formulation $\sS\sD(V^-)$, with $V^-$ satisfying \eqref{eq:selection}. Further, in this case $\sS\sD(V^-)$ and \SDws have the same unique solution. If $\tH^{-1/2}(\mS^\circ)\neq H^{-1/2}_\omS$ then the set of subspaces $V^-$  satisfying \eqref{eq:selection} has cardinality $\mathfrak{c}$. Further, if $u^i$ is the incident plane wave \eqref{eq:pw}, then:

(a) if $V_1\neq V_2$ are any two subspaces $V^-$
and, for $j=1,2$, $u^s_j$ is the solution to $\sS\sD(V_j)$, it holds that $u^s_1\neq u^s_2$ for almost all incident directions $\bd\in \mathbb{S}^n$ (all but finitely many $\bd$ if $n=2$);

(b) for almost all $\bd\in \mathbb{S}^n$ (all but countably many $\bd$ if $n=2$) there are infinitely many distinct  solutions $u^s$ to the set of scattering problems $\sS\sD(V^-)$ with $V^-$ satisfying \eqref{eq:selection}.
\end{thm}
\begin{proof} The first sentence is clear, and the second is a restatement of a result in Corollary \ref{DirEquivThmm2}. If $\tH^{-1/2}(\mS^\circ)\neq H^{-1/2}_\omS$, that the set of subspaces $V^-$  satisfying \eqref{eq:selection} has cardinality $\mathfrak{c}$ is Lemma \ref{lem:cardc}.

To see (a), suppose that $V_1\neq V_2$ are any two such subspaces and, for $j=1,2$, let $u^s_j$ be the solution to $\sS\sD(V_j)$, given explicitly by \eqref{eqn:SLPRep2} as $u^s_j = -\cS\phi_j$, where $\phi_j$ is the unique solution of \eqref{BIE_sl2}, i.e. $S_j\phi_j = P_jg$ where $g := \gamma^\pm (\chi u^i)$, and $S_j$, $V^*_j$, and $P_j$ denote $S$, $V^+_*$, and $P_{V^+_*}$, respectively, when $V^-=V_j$. Now, for $j=1,2$, $S_j:V_j\to V^*_j$ is an isomorphism, so that, if $V_1\subset V_2$, $V^*_2 = S_2(V_1)\oplus W$, where $W$ is the non-empty orthogonal complement of $S_2(V_1)$ in $V^*_2$. Let $P$ denote orthogonal projection onto $W$ in $H^{1/2}(\Gamma_\infty)$. Then, by Lemma \ref{lem:nonvanish} and Remark \ref{rem:all are subspaces}, $Pg\neq 0$ for almost all  $\bd\in \mathbb{S}^n$ (all but finitely many $\bd$ if $n=2$), which implies that $\phi_2\not\in V_1$, so that $\phi_1\neq \phi_2$. In the general case that $V_1\neq V_2$, it holds for $j=1,2$ that $V_j = V\oplus W_j$, for some orthogonal subspaces $V$, $W_1$, and $W_2$ of $H^{-1/2}_\omS$, with at least one of $W_1$ and $W_2$ non-trivial. Then the argument just made shows that, if $W_j$ is non-trivial, then $\phi_j\not\in V$, which implies that $\phi_1\neq \phi_2$, for almost all  $\bd\in \mathbb{S}^n$ (all but finitely many $\bd$ if $n=2$). But if $\phi_1\neq \phi_2$ then $u^s_1\neq u^s_2$ by \eqref{JumpRelns2}.

If $\{V_j:j\in \Z\}$ is a countably infinite set of subspaces $V^-$ satisfying \eqref{eq:selection}, with $V_i\neq V_j$ for $i\neq j$, then, since a countable union of sets of measure zero has measure zero, and a countable union of finite sets is countable, statement (b) follows from (a).
\end{proof}

The proof of the following corresponding theorem for the sound-hard case is essentially identical to that of Theorem \ref{thm:S61} for the sound-soft case and is omitted.

\begin{thm} \label{thm:s62} If $\tH^{1/2}(\mS^\circ)= H^{1/2}_\omS$ and $V^+$ satisfies \eqref{eq:selection} then $V^+=H^{1/2}_\omS$, so that there is only one formulation $\sN(V^+)$, and one formulation $\sS\sN(V^+)$, with $V^+$ satisfying \eqref{eq:selection}. Further, $\sS\sN(V^+)$ and \SNws have the same unique solution. If $\tH^{1/2}(\mS^\circ)\neq H^{1/2}_\omS$ then the set of subspaces $V^+$  satisfying \eqref{eq:selection} has cardinality $\mathfrak{c}$. Further, if $u^i$ is the incident plane wave \eqref{eq:pw}, then:

(a) if $V_1\neq V_2$ are any two subspaces $V^+$ and, for $j=1,2$, $u^s_j$ is the solution to $\sS\sN(V_j)$, it holds that $u^s_1\neq u^s_2$ for almost all incident directions $\bd\in \mathbb{S}^n$ (all but finitely many $\bd$ if $n=2$);

(b) for almost all $\bd\in \mathbb{S}^n$ (all but countably many $\bd$ if $n=2$) there are infinitely many distinct  solutions $u^s$ to the set of scattering problems $\sS\sN(V^+)$ with $V^+$ satisfying \eqref{eq:selection}.
\end{thm}

\subsection{When is $\tH^{\pm 1/2}(\mS^\circ)= H^{\pm 1/2}_\omS$?} \label{subsec:spacesequal}

To complete the picture from Theorems \ref{thm:S61} and \ref{thm:s62}, we examine the question: for which screens $\mS$ does it hold that $\tH^{\pm 1/2}(\mS^\circ)= H^{\pm 1/2}_\omS$?  Many of our results are expressed in terms of the $\pm 1/2$-nullity of Borel subsets of $\Gamma_\infty$, which was defined below \eqref{HsTdef} and related to other set properties in Theorem \ref{thm:null}.
 Our starting point is the following two theorems. These, when combined with Theorem \ref{thm:null}, also lead to explicit examples of distinct subspaces $V^\pm$ satisfying \eqref{eq:selection}.

\begin{thm}[{\cite[Proposition 2.11]{HewMoi:15}}]
\label{thm:Hs_equality_closed}
Suppose that $\mS^\circ \subset F_j\subset \omS$ and $F_j$ is closed, for $j=1,2$. Then $\tH^{\pm 1/2}(\mS^\circ)\subset H^{\pm 1/2}_{F_j} \subset H^{\pm 1/2}_\omS$, for $j=1,2$, and $H^{\pm 1/2}_{F_1}=H^{\pm 1/2}_{F_2}$ if and only if the symmetric difference $F_1\ominus F_2$ is $\pm 1/2$-null.
\end{thm}

\begin{thm}[{\cite[Theorem 3.12]{ChaHewMoi:13}}]
\label{thm:Hs_equality_open}
Suppose that $\mS^\circ\subset \Gamma_j\subset \omS$, for $j=1,2$, with $\Gamma_j$ open. Then $\tH^{\pm 1/2}(\mS^\circ)\subset \tH^{\pm 1/2}(\Gamma_j) \subset H^{\pm 1/2}_\omS$, for $j=1,2$, and  $\tH^{\pm 1/2}(\Gamma_1)=\tH^{\pm 1/2}(\Gamma_2)$ if and only if $\Gamma_1\ominus\Gamma_2$ is $\mp 1/2$-null.
\end{thm}

Applying Theorems \ref{thm:Hs_equality_closed} and \ref{thm:Hs_equality_open} with $F_1=\overline{\mS^\circ}$, $F_2=\omS$, and $\Gamma_1=\mS^\circ$, $\Gamma_2=\omS^\circ$, we obtain:
\begin{cor} \label{lem:equalityNullity} Let $\mS\subset \Gamma_\infty$. Then:
\begin{enumerate}[(i)]
\item $H^{\pm 1/2}_{\overline{\mS^\circ}}= H^{\pm 1/2}_{\omS}$ if and only if $\omS\setminus \overline{\mS^\circ}$ is $\pm 1/2$-null.
\item If
$\omS^\circ\setminus \mS^\circ$ is not $(\mp 1/2)$-null, then $\tH^{\pm 1/2}(\mS^\circ)\neq H^{\pm 1/2}_{\omS}$. If
 $\tH^{\pm 1/2}(\omS^\circ) = H^{\pm 1/2}_{\omS}$,
then $\tH^{\pm 1/2}(\mS^\circ) = H^{\pm 1/2}_{\omS}$ if and only if $\omS^\circ\setminus \mS^\circ$ is $(\mp 1/2)$-null.
\end{enumerate}
\end{cor}

The next theorem provides sufficient conditions on $\mS^\circ$ ensuring that $\tH^{\pm 1/2}(\mS^\circ) = H^{\pm 1/2}_{\omS}$.
Here, and henceforth, when we say that the open set $\Omega\subset \R^n$ is $C^0$ except at countably many points $P\subset \partial \Omega$, we mean that its  boundary $\partial\Omega$ can at each point in $\partial \Omega\setminus P$ be locally represented as the graph (suitably rotated) of a $C^0$ function from $\R^{n-1}$ to $\R$, with $\Omega$ lying only on one side of $\partial\Omega$.
(In more detail we mean that $\Omega$ satisfies the conditions of \cite[Definition 1.2.1.1]{Gri}, but for every $\bx \in \partial \Omega \setminus P$ rather than for every $\bx \in \partial \Omega$.) Examples of such sets $\Omega$ include prefractal appoximations to the Sierpinski triangle (Figure \ref{fig:Sierpinksi}).
That the theorem holds when $\mS^\circ$ is $C^0$ is well known (e.g.\ \cite[Theorem 3.29]{McLean}), and that it holds in the generality stated here follows from \cite[Theorem 3.24]{ChaHewMoi:13}.

\begin{thm} \label{thm:C0} If $\mS^\circ$ is $C^0$, or is $C^0$ except at countably many points $P\subset \partial \Omega$, with $P$ having only a finite set of limit points, then
 $\tH^{\pm 1/2}(\mS^\circ)= H^{\pm 1/2}_{\overline{\mS^\circ}}$.
\end{thm}

By combining the results in Corollary~\ref{lem:equalityNullity}, Theorem \ref{thm:C0} and Theorem \ref{thm:null}, we can derive the following corollary, which explores the case where $\omS^\circ$ is $C^0$, except perhaps at countably many points, but where $\mS^\circ \subsetneqq \omS^\circ$. We illustrate this corollary with examples where $\mS$ is open, so that $\mS^\circ = \mS$ and $\overline{\mS^\circ}=\omS$, but $\mS^\circ \neq \omS^\circ$, in Examples \ref{ex:openminuscantor} and \ref{ex:openminuscheese}.
\begin{cor}
\label{prop:TildeSubscript}
Suppose that $\mS^\circ\subsetneqq \omS^\circ$ and that $\omS^\circ$ is  $C^0$, or $C^0$ except at a countable set of points $P$ that has only finitely many limit points. Then:
\begin{enumerate}[(i)]

\item if $\omS^\circ \setminus \mS^\circ$ has interior points then $\tH^{\pm 1/2}(\mS^\circ) \subsetneqq H^{\pm 1/2}_{\omS}$;
\item \label{ts1}
if $\omS^\circ\setminus \mS^\circ$ is countable, or $\omS^\circ\setminus \mS^\circ\subset \partial \Omega$, where $\Omega\subset \Gamma_\infty$ is a Lipschitz open set,
then $\tH^{\pm 1/2}(\mS^\circ) = H^{\pm 1/2}_{\omS}$;

\item if $\tH^{1/2}(\mS^\circ)=H^{1/2}_{\omS}$ or $m(\omS^\circ\setminus \mS^\circ)=0$, then $\tH^{-1/2}(\mS^\circ)=H^{-1/2}_{\omS}$;

\item $\tH^{1/2}(\mS^\circ)=H^{1/2}_{\omS}$ if and only if $\mathrm{cap}(\omS^\circ \setminus \mS^\circ) = 0$;

\item if $d:=\dimH(\omS^\circ\setminus \mS^\circ)<n-2$ then $\tH^{1/2}(\mS^\circ) = H^{1/2}_{\omS}$, while $\tH^{1/2}(\mS^\circ) \subsetneqq H^{1/2}_{\omS}$ for  $d>n-2$.

\end{enumerate}
\end{cor}

\section{Well-posedness of standard formulations}\label{subsec:nonuni}

In this section we investigate for which screens $\mS$ the standard formulations \SDcl, \SNcl, \Dst, and \Nsts are well-posed. Since we know already that these formulations have at least one solution (Corollary \ref{cor:atleastS} and Theorems \ref{DirEquivThmm} and \ref{NeuEquivThmm}), we need only consider the question of uniqueness. By Corollary \ref{lem:uniequiv} it is enough to consider this question for \Dsts and \Nst.

We noted in Remark \ref{rem:acsuperfl} that $\sD(V^-)$ with $V^-=\tH^{-1/2}(\mS^\circ)$ is equivalent to \Dst, augmented by the additional constraints \eqref{a1} and \eqref{a2}, and that $\sN(V^+)$ with $V^+ = \tH^{1/2}(\mS^\circ)$ is equivalent to \Nsts augmented by \eqref{b1} and \eqref{b2}.
Since $\sD(V^-)$ and $\sN(V^+)$ are well-posed (Theorem \ref{DirEquivThmm}), asking whether \Dsts has more than one solution is equivalent to asking whether \eqref{a1} and \eqref{a2} are superfluous, i.e.\ whether $\sD(V^-)$ remains well-posed if \eqref{a1} and \eqref{a2} are deleted. Similarly, asking whether \Nsts has more than one solution is equivalent to asking whether \eqref{b1} and \eqref{b2} are superfluous, i.e.\ whether $\sN(V^+)$ remains well-posed if \eqref{b1} and \eqref{b2} are deleted.

\begin{thm} \label{thm:superlf} Equation \eqref{a2} is superfluous in $\sD(V^-)$
if and only if $V^-=H^{-1/2}_\omS$.
Similarly, \eqref{b2} is superfluous in $\sN(V^+)$
if and only if $V^+=H^{1/2}_\omS$.
Equation \eqref{a1} is superfluous in $\sD(V^-)$ if and only if $(V^+_*)^\perp\cap H^{1/2}_\omS = \{0\}$. Similarly, \eqref{b1} is superfluous in $\sN(V^+)$
if and only if $(V^-_*)^\perp\cap H^{-1/2}_\omS = \{0\}$.
In particular, if $V^\pm$ satisfies \eqref{eq:selection} then \eqref{a1} is superfluous if $\partial \mS$ is $1/2$-null, and \eqref{b1} is superfluous if $\partial \mS$ is $-1/2$-null.
If $V^\pm = \tH^{\pm 1/2}(\mS^\circ)$ then \eqref{a1} is superfluous if and only if $\partial \mS$ is $1/2$-null, and \eqref{b1} is superfluous if and only if $\partial \mS$ is $-1/2$-null.

\end{thm}
\begin{proof} If $v\in W^{1,\mathrm{loc}}(D)$ and satisfies \eqref{eq:he} in $D$, so that $v\in W^{1,\mathrm{loc}}(D;\Delta)$, then $[v]\in H^{1/2}_\omS$ and $[\partial_\bn v]\in H^{-1/2}_\omS$. Thus \eqref{a2} and \eqref{b2} are superfluous if $V^\pm = H^{\pm 1/2}_\omS$. If also $v$ satisfies \eqref{a3} then $P_{V^+_*} [v]=0$, i.e.\ $[v]\in (V^+_*)^\perp$, so that $[v]\in (V^+_*)^\perp\cap H^{1/2}_\omS$. Thus \eqref{a1} is superfluous if $(V^+_*)^\perp\cap H^{1/2}_\omS = \{0\}$. Similarly, if $v$ satisfies \eqref{b3}, then $[\partial_\bn v]\in (V^-_*)^\perp\cap H^{-1/2}_\omS$, so \eqref{b1} is superfluous if $(V^-_*)^\perp\cap H^{-1/2}_\omS = \{0\}$.

Conversely, suppose that $V_1\neq V_2$, where $V_1=V^-$ and $V_2=H^{-1/2}_\omS$. Then, by Theorem \ref{thm:S61}(a), there exists an incident wave direction $\bd$ such that $u^s_1\neq u^s_2$, where $u^s_j$ is the solution to $\sS\sD(V_j)$ for $j=1,2$. Since $V_1\subset V_2$, $u^s_2$ satisfies all the conditions of $\sS\sD(V^-)$ except for \eqref{a2}. So $v:= u^s_1-u^s_2\neq 0$ is a solution to $\sD(V^-)$ with $g_{\sD}=0$ and \eqref{a2} deleted. Thus uniqueness fails for $\sD(V^-)$ if \eqref{a2} is deleted and $V^-\neq H^{-1/2}_\omS$. Similarly, arguing using Theorem \ref{thm:s62}(a), uniqueness fails for $\sN(V^+)$ if \eqref{b2} is deleted and $V^+\neq H^{1/2}_\omS$.

Next, suppose that $(V^+_*)^\perp\cap H^{1/2}_\omS \neq \{0\}$, choose a non-zero $\psi \in(V^+_*)^\perp\cap H^{1/2}_\omS$, and set $v:= \cD\psi$. Then, by Theorem \ref{LayerPotRegThm}, except that \eqref{a1} is not satisfied as $[v]=\psi$, $v$ satisfies $\sD(V^-)$ with $g_{\sD}=0$. Thus uniqueness fails for $\sD(V^-)$ if \eqref{a1} is deleted and $(V^+_*)^\perp\cap H^{1/2}_\omS \neq \{0\}$. Similarly, uniqueness fails for $\sN(V^+)$ if \eqref{b1} is deleted and $(V^-_*)^\perp\cap H^{-1/2}_\omS \neq \{0\}$, arguing in this case by defining $v:= \cS\phi$, where $0\neq \phi \in(V^-_*)^\perp\cap H^{-1/2}_\omS$.

If $V^\pm=\tH^{\pm 1/2}(\mS^\circ)$, then (recall \eqref{eq:duals})  $V_*^\mp = \left(H^{\mp 1/2}_{(\mS^\circ)^c}\right)^\perp$ and $(V^+_*)^\perp\cap H^{\mp 1/2}_\omS = H^{\mp 1/2}_{\partial \mS}$, so that $\{0\} = (V^+_*)^\perp\cap H^{\mp 1/2}_\omS$ if and only if $\partial \mS$ is $\mp1/2$-null. If $V^\pm$ satisfies \eqref{eq:selection}, then $V^\pm\supset \tH^{\pm 1/2}(\mS^\circ)$ and, as noted below \eqref{dualgen}, $V^\mp_* \supset \left(H^{\mp 1/2}_{(\mS^\circ)^c}\right)^\perp$, so that $(V^+_*)^\perp\cap H^{\mp 1/2}_\omS \subset H^{\mp 1/2}_{\partial \mS}$, so that $\{0\} = (V^+_*)^\perp\cap H^{\mp 1/2}_\omS$ if $\partial \mS$ is $\mp1/2$-null.
\end{proof}

By applying Theorem \ref{thm:superlf} with $V^\pm=\tH^{\pm 1/2}(\mS^\circ)$, and recalling Theorem \ref{thm:null}, and the results of \S\ref{subsec:spacesequal}, we can now clarify when the standard scattering and BVP formulations are well-posed.

\begin{thm} \label{thm:wpst} (a) Problems \SDcls and \Dsts are well-posed if and only if $\tH^{-1/2}(\mS^\circ)$ $= H^{-1/2}_\omS$ and $\partial \mS$ is $1/2$-null. In particular, any of the following conditions is sufficient to ensure that \SDcls and \Dsts are well-posed:
\begin{enumerate}[(i)]
\item \label{cond1} $\mS$ is open and $C^0$, or $C^0$ except at a countable set of points that has only finitely many limit points.
\item Condition \rf{cond1} holds for $\mS^\circ$,
and $\mathrm{cap}(\omS\setminus \overline{\mS^\circ})=0$. The latter condition holds, when $n=2$, for example if $\omS\setminus \overline{\mS^\circ}$ is countable and, when $n=3$, for example if $\dimH(\omS\setminus \overline{\mS^\circ})<1$ or if $\omS\setminus \overline{\mS^\circ}\subset \cup_{j=1}^M \partial\Omega_j$, with each $\Omega_j\subset\Gamma_\infty$ a Lipschitz open set.
\item Condition \rf{cond1} holds for $\omS^\circ$,
and $m(\partial \mS)=0$.
\end{enumerate}
On the other hand, \SDcls and \Dsts are {\em not} well-posed if $\omS^\circ\setminus \mS^\circ$ is not $1/2$-null, or if $\omS\setminus \overline{\mS^\circ}$ is not $-1/2$-null, in particular if $\dimH(\omS\setminus \overline{\mS^\circ})>n-2$.

(b) Problems \SNcls and \Nsts are well-posed if and only if $\tH^{1/2}(\mS^\circ)= H^{1/2}_\omS$ and $\partial \mS$ is $-1/2$-null. In particular, \SNcls and \Nsts are well-posed if $\mS$ is a Lipschitz open set, or if $\mS$ is $C^0$ except at a countable set of points that has only finitely many limit points and also $\partial \mS\subset \cup_{j=1}^\infty \partial \Omega_j$, with each $\Omega_j\subset\Gamma_\infty$ a Lipschitz open set.  On the other hand, \SNcls and \Nsts fail to be well-posed if $\mathrm{cap}(\partial \mS)>0$, in particular if $\dimH(\partial \mS)>n-2$.
\end{thm}
\begin{proof}
The first sentences of (a) and (b) follow from the remarks before Theorem \ref{thm:superlf} and Theorem \ref{thm:superlf} applied with  $V^\pm = \tH^{\pm 1/2}(\mS^\circ)$. The remainder of (b) follows from Theorems \ref{thm:C0} and \ref{thm:null}(a)(d), (f) and (g). Part (a)(i) follows from Theorems \ref{thm:C0} and \ref{thm:null}(f)-(h), part (a)(ii) additionally from Theorem \ref{thm:null}(d), and part (a)(iii) from Corollary \ref{prop:TildeSubscript}(iii) and Theorem \ref{thm:null}(e). The remainder of (a) follows from Corollary \ref{lem:equalityNullity}, and from Theorems \ref{thm:C0} and \ref{thm:null}(d).
\end{proof}

\begin{rem}
We note that if \Dsts is well-posed then its solution coincides with that of $\sD(V^-)$, and that if \Nsts is well-posed then its solution coincides with that of $\sN(V^+)$, this for all $V^\pm$ satisfying \eqref{eq:selection} (which means in fact for $V^\pm = \tH^{\pm1/2}(\mS^\circ)= H^{\pm1/2}_\omS$), as a consequence of Theorems \ref{DirEquivThmm} and \ref{NeuEquivThmm}. Similarly, as long as $(\Delta + k^2)u^i=0$ in a neighbourhood of $\omS$, if \SDcls is well-posed then its solution coincides with that of \SDw, and if \SNcls is well-posed then its solution coincides with that of \SNw, by Corollaries \ref{DirEquivThmm2} and \ref{NeuEquivThmm2}.
\end{rem}

A further case of Theorem \ref{thm:superlf} worthy of note is $V^\pm = H^{\pm 1/2}_\omS$. In this case, recent results in \cite{ChaHewMoi:13} reveal that the question of whether or not \eqref{a1} and \eqref{b1} are superfluous can be expressed in terms of properties of the space $H^s_0(\Omega)$ (defined in \rf{eq:W10}).
\begin{thm} If $V^\pm = H^{\pm 1/2}_\omS$,
 then \eqref{a1}, \eqref{a2} and \eqref{b2} are superfluous in $\sD(V^-)$ and $\sN(V^+)$. Further \eqref{b1} is superfluous in $\sN(V^+)$ if and only if $H^{1/2}_0(\omS^c) = H^{1/2}(\omS^c)$, which holds if $\partial \omS$ is $-1/2$-null (equivalently if $\mathrm{cap}(\partial \omS)=0$), and certainly if
 $\partial \omS \subset \cup_{j=1}^M \partial\Omega_j$, with each $\Omega_j\subset\Gamma_\infty$ a Lipschitz open set.
\end{thm}
\begin{proof}
If $V^\pm = H^{\pm 1/2}_\omS$ then the fact that \eqref{a2} and \eqref{b2} are superfluous was stated in Theorem \ref{thm:superlf}.
Further, recalling \eqref{eq:duals}, $V_*^\mp = \left(\tH^{\mp 1/2}(\omS^c)\right)^\perp$ so that $(V^\mp_*)^\perp\cap H^{\mp 1/2}_\omS = \tH^{\mp 1/2}(\omS^c)\cap H^{\mp 1/2}_{\partial \omS}$. Then \eqref{a2} and \eqref{b2} and superfluous if and only if $\{0\} = (V^\mp_*)^\perp\cap H^{\mp 1/2}_\omS$, which holds if and only if $\{0\} = \tH^{\mp 1/2}(\omS^c)\cap H^{\mp 1/2}_{\partial \omS}$,  and, by \cite[Corollary 3.29]{ChaHewMoi:13}, this holds if and only if $H^{\pm 1/2}_0(\omS^c) = H^{\pm 1/2}(\omS^c)$. Further, since \cite[Corollary 3.29]{ChaHewMoi:13} $H^{-1/2}_0(\omS^c) = H^{-1/2}(\omS^c)$, we conclude that \eqref{a1} is always superfluous. Finally, we note that \cite[Corollary 3.29]{ChaHewMoi:13} proves that also $H^{1/2}_0(\omS^c) = H^{1/2}(\omS^c)$ (so that \eqref{b1} is superfluous) if $\partial \omS$ is $-1/2$-null; by Theorem \ref{thm:null} this holds if and only if $\mathrm{cap}(\partial \omS)=0$, which certainly holds if $\partial \omS \subset \cup_{j=1}^M \partial\Omega_j$, with each $\Omega_j\subset\Gamma_\infty$ a Lipschitz open set.
\end{proof}

\section{Dependence on domain and limiting geometry principles} \label{sec:domdep}

In \S\ref{subsec:bvpsnovel} we introduced novel families of formulations for the screen scattering problems, each family member well-posed by Theorems \ref{DirEquivThmm} and \ref{NeuEquivThmm} and Corollaries \ref{DirEquivThmm2} and \ref{NeuEquivThmm2}. If we constrain these formulations by our physical selection principle \eqref{eq:selection} and $\tH^{\pm1/2}(\mS^\circ) = H^{\pm 1/2}_\omS$, then these formulations collapse onto single formulations. By Theorem \ref{thm:C0} this happens, in particular, if $\mS$ is $C^0$ or is $C^0$ except at a countable set of points having only finitely many limit points.

However, for an arbitrary screen $\mS$ (simply a bounded subset of $\Gamma_\infty$), the results of \S\ref{subsec:spacesequal} make clear that, in general, $\tH^{\pm1/2}(\mS^\circ) \neq H^{\pm 1/2}_\omS$, even if we constrain $\mS$ further, by requiring that $\mS$ is open or closed. And if indeed $\tH^{\pm1/2}(\mS^\circ) \neq H^{\pm 1/2}_\omS$, then, by Theorems \ref{thm:S61} and \ref{thm:s62}, there are infinitely many distinct formulations (with cardinality $\mathfrak{c}$) satisfying \eqref{eq:selection}, and, at least for the case of plane wave incidence and almost all incident directions, infinitely many corresponding solutions to the scattering problems.

One approach to selecting the ``physically correct'' solution
from this multitude is to think of $\mS$ as the limit of a sequence of screens $(\mS_j)_{j\in \N}$, where each screen $\mS_j$ is sufficiently regular so that the correct choice of solution is clear. This is a natural approach for recursively generated fractal structures. For example the open set $\mS$ whose boundary is the Koch snowflake is usually generated as the limit of a sequence $\mS_1\subset \mS_2\subset ...$ (see Figure \ref{fig:koch}), where each $\mS_j$ is a Lipschitz open set. Likewise the closed set $\mS$ which is the Sierpinski triangle is usually generated as the limit of a sequence of closed sets $\mS_1\supset \mS_2\supset ...$ (see Figure \ref{fig:Sierpinksi}), where, for each $j$, $\mS_j^\circ$ is $C^0$ (indeed Lipschitz) except at a finite set of points. Our first theorem (cf.\ \cite[Theorem 4.3 and Proposition 4.5]{ChaHewMoi:13}) deals with these and related cases.

\begin{thm} \label{thm:limits} (a) Suppose that $\mS_j\subset \Gamma_\infty$ is open for $j\in \N$, that $\mS_1\subset \mS_2 \subset ...$, and that $\mS = \cup_{j=1}^\infty \mS_j$ is bounded. Further (as usual) let $D:= \R^n\setminus \omS$. Let $u^s$ denote the solution to $\sS\sD(V^-)$ with $V^-=\tH^{-1/2}(\mS)$, and $u^s_j$ the solution to $\sS\sD(V^-)$ with $V^-=\tH^{-1/2}(\mS_j)$. Then, for every $\chi\in \scrD(\R^n)$,
\begin{equation} \label{eq:conv}
\|\chi(u^s-u^s_j)\|_{W^1(\R^n)} \to 0 \quad \mbox{ as } \quad j\to\infty.
\end{equation}
Similarly, if $u^s$ denotes the solution to $\sS\sN(V^+)$ with $V^+=\tH^{1/2}(\mS)$, and $u^s_j$ the solution to $\sS\sN(V^+)$ with $V^+= \tH^{1/2}(\mS_j)$, then, for every $\chi\in \scrD(\R^n)$, $\|\chi(u^s-u^s_j)\|_{W^1(D)} \to 0$ as $j\to\infty$.

(b)
Suppose that $\mS_j\subset \Gamma_\infty$ is compact for $j\in \N$, that $\mS_1\supset \mS_2 \supset ...$, and that $\mS$ is given by $\mS = \cap_{j=1}^\infty \mS_j$. Let $u^s$ denote the solution to $\sS\sD(V^-)$ with $V^-=H^{-1/2}_\mS$, and $u^s_j$ the solution to $\sS\sD(V^-)$ with $V^-= H^{-1/2}_{\mS_j}$. Then \eqref{eq:conv} holds for every $\chi\in \scrD(\R^n)$. Similarly,
let $u^s$ denote the solution to $\sS\sN(V^+)$ with $V^+=H^{1/2}_\mS$, and $u^s_j$ the solution to $\sS\sN(V^+)$ with $V^+= H^{1/2}_{\mS_j}$. Then, for every $\chi\in \scrD(\R^n)$ and every open $\Omega\subset \Gamma_\infty$ with $\mS\subset \Omega$, $\|\chi(u^s-u^s_j)\|_{W^1(\widetilde{D})} \to 0$ as $j\to\infty$, where $\widetilde{D} := \R^n \setminus \overline{\Omega}$.
\end{thm}
\begin{proof} Part (a). In the first case, by Corollary \ref{DirEquivThmm2}, $u^s=-\cS[\partial_\bn u]$ and $u^s_j =  -\cS[\partial_\bn u_j]$, where $[\partial_\bn u]\in V^-$ is the unique solution of \eqref{BIE_slw2}
with $V^-=\tH^{-1/2}(\mS)$, and $[\partial_\bn u_j]$ the unique solution with $V^-=\tH^{-1/2}(\mS_j)$. Since, by Theorem \ref{thm:ContCoerc1}, $a_S$ is coercive  and, by \cite[Proposition 3.33]{ChaHewMoi:13},
$$
\tH^{\pm 1/2}(\mS) = \overline{\bigcup_{j=1}^\infty \tH^{\pm 1/2}(\mS_j)},
$$
it follows from C\'ea's lemma (cf.\ \cite[(8)]{ChaHewMoi:13}) that $\|[\partial_\bn u]-[\partial_\bn u_j]\|_{\tH^{-1/2}(\mS)}\to 0$ as $j\to\infty$, and then \eqref{eq:conv} follows from Theorem \ref{LayerPotRegThm}(ii). The rest of part (a) follows similarly, using Corollary \ref{NeuEquivThmm2} and the coercivity of $a_T$ from Theorem \ref{thm:ContCoerc1}.

Part (b). In the first case, arguing as in part (a), $u^s=-\cS[\partial_\bn u]$ and $u^s_j =  -\cS[\partial_\bn u_j]$, where $[\partial_\bn u]\in V^-$ is the unique solution of \eqref{BIE_slw2}
with $V^-=H^{-1/2}_\mS$, and $[\partial_\bn u_j]$ the unique solution with $V^-=H^{-1/2}_{\mS_j}$. Since $a_S$ is coercive  and (e.g.\ \cite[Proposition 3.34]{ChaHewMoi:13})
$$
H^{\pm 1/2}_\mS = \bigcap_{j=1}^\infty H^{\pm 1/2}_{\mS_j},
$$
it follows from \cite[Lemma 2.4]{ChaHewMoi:13} that $\|[\partial_\bn u]-[\partial_\bn u_j]\|_{H^{-1/2}(\Gamma_\infty)}\to 0$ as $j\to\infty$, and then \eqref{eq:conv} follows from Theorem \ref{LayerPotRegThm}(ii). The rest of part (b) follows similarly, using Corollary \ref{NeuEquivThmm2} and the coercivity of $a_T$, and noting that, if $\Omega\subset \Gamma_\infty$ is open and $\mS\subset \Omega$, then $\mS_j\subset \Omega$ for all $j$ sufficiently large.
\end{proof}

\begin{rem} \label{rem:pw} We note that, if $u^s_j\to u^s$ in any of the senses indicated in the above theorem, then also, by elliptic regularity arguments, $u^s_j\to u^s$ uniformly on compact subsets of $D$. To see this, let $F\subset D$ be any such compact subset, choose $\chi\in \scrD(D)$ with $\chi = 1$ in a neighbourhood of $F$, and let $v_j:= \chi(u^s-u^s_j)$. Then $(\Delta + k^2)v_j = f_j\in L^{2}_\mathrm{comp}(\R^n)$, which implies that \eqref{eq:ui} holds with $u^i$ and $f$ replaced by $v_j$ and $f_j$. From this, noting also that $\|f_j\|_{L^2(\R^n)}\to0$ as $j\to 0$ and $\supp(f_j)\subset \supp(\chi)$, we see that, uniformly for $\bx\in F$, $|u^s(\bx)-u^s_j(\bx)|=|v_j(\bx)|\to 0$ as $j\to \infty$.
\end{rem}

Theorem \ref{thm:limits} and Remark \ref{rem:pw} suggest the following {\em limiting geometry} criteria for selecting physically appropriate solutions for bounded screens $\mS$ that are open and closed, respectively. \footnote{We note that this general approach, defining the solution to a BVP for an irregular domain $\Omega$ by taking the limit of the solutions for a sequence $(\Omega_j)_{j=1}^\infty$ of regular domains, is familiar from potential theory, dating back to Wiener \cite{Wi:24} (and see \cite[p.317]{Kellogg}, \cite{Bjorn:07}). In that context the approximating sequence $\Omega_j$ approximates $\Omega$ in the sense that $\Omega_1\subset \Omega_2 \subset \ldots$ with $\Omega = \cup_{j=1}^\infty \Omega_j$, just as in Definition \ref{def:lgopen}.}

\begin{defn}[Limiting Geometry Solution for an Open Screen] \label{def:lgopen}If $\mS\subset \Gamma_\infty$ is bounded and open, we call the scattered field $u^s$ a {\em limiting geometry solution for $\mS$ for sound-soft (sound-hard) scattering} if there exists a sequence $(\mS_j)_{j\in \N}$ of open subsets of $\Gamma_\infty$ such that: (i) $\mS_1\subset \mS_2 \subset ...$ and $\mS = \cup_{j=1}^\infty \mS_j$; (ii) for $j\in \N$, $\tH^{\pm 1/2}(\mS_j) = H^{\pm 1/2}_{\overline{\mS_j}}$, so that the formulations $\sS\sD(V^-)$ ($\sS\sN(V^+)$) satisfying \eqref{eq:selection} collapse to a single formulation with a well-defined unique solution $u^s_j$; (iii) for $\bx\in D:= \R^n\setminus \mS$,
$u^s(\bx)= \lim_{j\to\infty} u^s_j(\bx)$.
\end{defn}

\begin{defn}[Limiting Geometry Solution for a Closed Screen] \label{def:lgclosed}If $\mS\subset \Gamma_\infty$ is compact, call the scattered field $u^s$ a {\em limiting geometry solution for $\mS$ for sound-soft (sound-hard) scattering} if there exists a sequence $(\mS_j)_{j\in \N}$ of compact subsets of $\Gamma_\infty$ such that: (i) $\mS_1\supset \mS_2 \supset ...$ and $\mS = \cap_{j=1}^\infty \mS_j$; (ii) for $j\in \N$, $\tH^{\pm 1/2}(\mS_j^\circ) = H^{\pm 1/2}_{\mS_j}$, so that the formulations $\sS\sD(V^-)$ ($\sS\sN(V^+)$) satisfying \eqref{eq:selection} collapse to a single formulation with a well-defined unique solution $u^s_j$; (iii) for $\bx\in D:= \R^n\setminus \mS$,
$u^s(\bx)= \lim_{j\to\infty} u^s_j(\bx)$.
\end{defn}

The existence and uniqueness of such limiting geometry solutions is the subject of the following corollary, which is a consequence of Theorem \ref{thm:limits} and Remark \ref{rem:pw}.

\begin{cor} \label{cor:lgs} (a) For every bounded open screen $\mS\subset \Gamma_\infty$ there exists a unique limiting geometry solution for sound-soft (sound-hard) scattering by $\mS$, and this solution is the unique solution $u^s$ to $\sS\sD(V^-)$ ($\sS\sN(V^+)$) with $V^-=\tH^{-1/2}(\mS)$ ($V^+=\tH^{1/2}(\mS)$).

(b) For every compact screen $\mS\subset \Gamma_\infty$ there exists a unique limiting geometry solution for sound-soft (sound-hard) scattering by $\mS$, and this solution is the unique solution $u^s$ to \SDws (\SNw), equivalently the unique solution to $\sS\sD(V^-)$ ($\sS\sN(V^+)$) with $V^-=H^{-1/2}_\mS$ ($V^+=H^{1/2}_\mS$).
\end{cor}
\begin{proof}
Part (a) follows from Theorem \ref{thm:limits}(a) and Remark \ref{rem:pw}, provided there exists a sequence $(\mS_j)$ satisfying the conditions of Definition \ref{def:lgopen}. One such sequence can be constructed as follows. For $j\in \N$ and $\mathbf{\ell}\in \Z^{n-1}$ let $\Gamma_{\ell,j}:= \{\bx \in\Gamma_\infty: \ell_i<2^j x_i < \ell_{i}+1, \mbox{ for } i=1,...,n-1\}$, and let $I_j:= \{\mathbf{\ell}\in \Z^{n-1}: \overline{\Gamma_{\ell,j}}\subset \mS\}$. Then $\mS_j$, defined to be the interior of the set $\cup_{\ell\in I_j} \overline{\Gamma_{\ell,j}}$,
satisfies the required conditions. In particular, $\mS_j$ is $C^0$ except at a finite number of points, so that $\tH^{\pm 1/2}(\mS_j) = H^{\pm 1/2}_{\overline{\mS_j}}$ by Theorem \ref{thm:C0}.

Similarly, part (b) follows from Theorem \ref{thm:limits}(b) and Remark \ref{rem:pw}, provided there exists a sequence $(\mS_j)$ satisfying the conditions of Definition \ref{def:lgclosed}. One such sequence can be constructed as follows. Let $J_j:= \{\mathbf{\ell}\in \Z^{n-1}: \mS\cap \overline{\Gamma_{\ell,j}}\neq \emptyset \}$. Then
$\mS_j := \cup_{\ell\in J_j} \overline{\Gamma_{\ell,j}}$
satisfies the required conditions.
\end{proof}

\begin{rem} \label{rem:otherscreens} The limiting geometry principles in Definitions \ref{def:lgopen} and \ref{def:lgclosed} provide criteria for selecting physically relevant solutions when $\Gamma$ is either compact, or bounded and open. Other cases can also be considered. For example, suppose that $\mS = F\cup \Omega$, where $F$, $\Omega\subset \Gamma_\infty$ have disjoint closures,
$F$ is compact with empty interior, and $\Omega$ is bounded and open. We can construct a limiting geometry solution for $\mS^\circ$ according to Definition \ref{def:lgopen}, which corresponds by Corollary \ref{cor:lgs}(a) to computing the solution to $\sS\sD(V^-)$ or $\sS\sN(V^+)$ with $V^\pm = \tH^{\pm 1/2}(\mS^\circ)$, but this solution ignores the component $F$ since $\mS^\circ = \Omega$. Alternatively, we can construct a limiting geometry solution for $\omS = F \cup \overline{\Omega}$ according to Definition \ref{def:lgclosed}, which corresponds by Corollary \ref{cor:lgs}(b) to computing the solution to $\sS\sD(V^-)$ or $\sS\sN(V^+)$ with $V^\pm = H^{\pm 1/2}_\omS$, but this ignores the difference between $\Omega$ and $\overline{\Omega}$, which is important if $\tH^{\pm 1/2}(\Omega)\neq H^{\pm 1/2}_{\overline{\Omega}}$ (cf.\ Examples \ref{ex:openminuscantor} and \ref{ex:openminuscheese} below).

It may be that a more appropriate notion of a limiting geometry solution in this case can be constructed by thinking of $\mS$ as the limit of a sequence of screens $\mS_j:= F_j\cup \Omega_j$, with $F_1 \supset F_2 ...$, $\Omega_1\subset \Omega_2 ...$, each $F_j$ compact, each $\Omega_j$ open and bounded, and $F=\cap_{j=1}^\infty F_j$, $\Omega = \cup_{j=1}^\infty \Omega_j$. And we might conjecture that the limiting geometry solutions this gives rise to are the solutions to $\sS\sD(V^-)$ or $\sS\sN(V^+)$ with $V^\pm = H^{\pm 1/2}_F + \tH^{\pm 1/2}(\Omega)$, which spaces satisfy $\tH^{\pm 1/2}(\mS^\circ) \subsetneqq V^\pm \subsetneqq H^{\pm 1/2}_\omS$ if $F$ is not $\pm 1/2$-null and $\tH^{\pm 1/2}(\Omega)\neq H^{\pm 1/2}_{\overline{\Omega}}$. But we leave full analysis of these and other cases to future work.
\end{rem}

\begin{rem} \label{rem:useful} For a bounded screen $\mS\subset \Gamma_\infty$ with $\tH^{\pm 1/2}(\mS^\circ) \neq H^{\pm 1/2}_\omS$,
there are infinitely many (with cardinality $\mathfrak{c}$) distinct formulations $\sS\sD(V^-)$ and $\sS\sN(V^+)$ satisfying \eqref{eq:selection} (Theorems \ref{thm:S61} and \ref{thm:s62}). Definition \ref{def:lgopen} gives physical meaning to the solutions for $V^\pm = \tH^{\pm 1/2}(\mS^\circ)$ as the limiting geometry solutions for the open set $\mS^\circ$. Similarly, Definition \ref{def:lgclosed} gives physical meaning to the solutions for $V^\pm = H^{\pm 1/2}_\omS$ as the limiting geometry solutions for the closed set $\omS$. But what physical meaning, if any, do the other solutions have?

If $\Gamma_0$ is open and $\mS^\circ \subset \Gamma_0 \subset \omS$, then $V^\pm = \tH^{\pm 1/2}(\Gamma_0)$ satisfies \eqref{eq:selection} and the corresponding solutions are limiting geometry solutions for $\Gamma_0$ in the sense of Definition \ref{def:lgopen}: further, by Theorem \ref{thm:Hs_equality_open} and Corollary \ref{lem:equalityNullity}, $V^\pm$ is different to both $\tH^{\pm 1/2}(\mS^\circ)$ and $H^{\pm 1/2}_\omS$ if $\Gamma_0\setminus \mS^\circ$ and $\omS^\circ\setminus \Gamma_0$ are both not $\mp 1/2$-null. Similarly, if $\Gamma_0$ is closed and $\mS^\circ \subset \Gamma_0 \subset \omS$, then $V^\pm = H^{\pm 1/2}_{\Gamma_0}$ satisfies \eqref{eq:selection} and the corresponding solutions are limiting geometry solutions for $\Gamma_0$ in the sense of Definition \ref{def:lgclosed}: further, by Theorems \ref{thm:Hs_equality_closed} and \ref{thm:Hs_equality_open}, $V^\pm$ is different to both $\tH^{\pm 1/2}(\mS^\circ)$ and $H^{\pm 1/2}_\omS$ if $\Gamma_0^\circ\setminus \mS^\circ$ is not $\mp 1/2$-null and $\omS\setminus \Gamma_0$ is not $\pm 1/2$-null.

It may also well be the case -- see Remark \ref{rem:otherscreens} -- that solving $\sS\sD(V^-)$ or $\sS\sN(V^+)$ with some $V^\pm$ intermediate between $\tH^{\pm 1/2}(\mS^\circ)$ and $H^{\pm 1/2}_\omS$ is of physical interest as a limiting geometry solution in some sense different from that of Definitions \ref{def:lgopen} or \ref{def:lgclosed}.
\end{rem}

\section{Scattering by fractals and other examples} \label{sec:fract}

In this final section we illustrate the results of the previous sections by some concrete examples, including a number of examples where the screen is fractal or has a fractal boundary.

Our first three examples consider scattering by screens that are compact sets with empty interior. While the standard formulations \SDws and \SNws are physically relevant in this case, in particular are limiting geometry solutions in the sense of Definition \ref{def:lgclosed}, the formulations \SDcls and \SNcls lack boundary conditions: equations \eqref{eq:bc} are empty. In Examples \ref{ex:sier} and \ref{ex:cantor} the screen has zero surface measure and the incident field fails to see the screen for sound-hard scattering, while Example \ref{ex:swiss} is a screen with empty interior but positive surface measure where the scattered field, defined as a limiting geometry solution by Definition \ref{def:lgclosed}, is non-zero for both sound-soft and sound-hard scattering.

\begin{figure}[!t]
\begin{center}
\includegraphics[height=25mm]{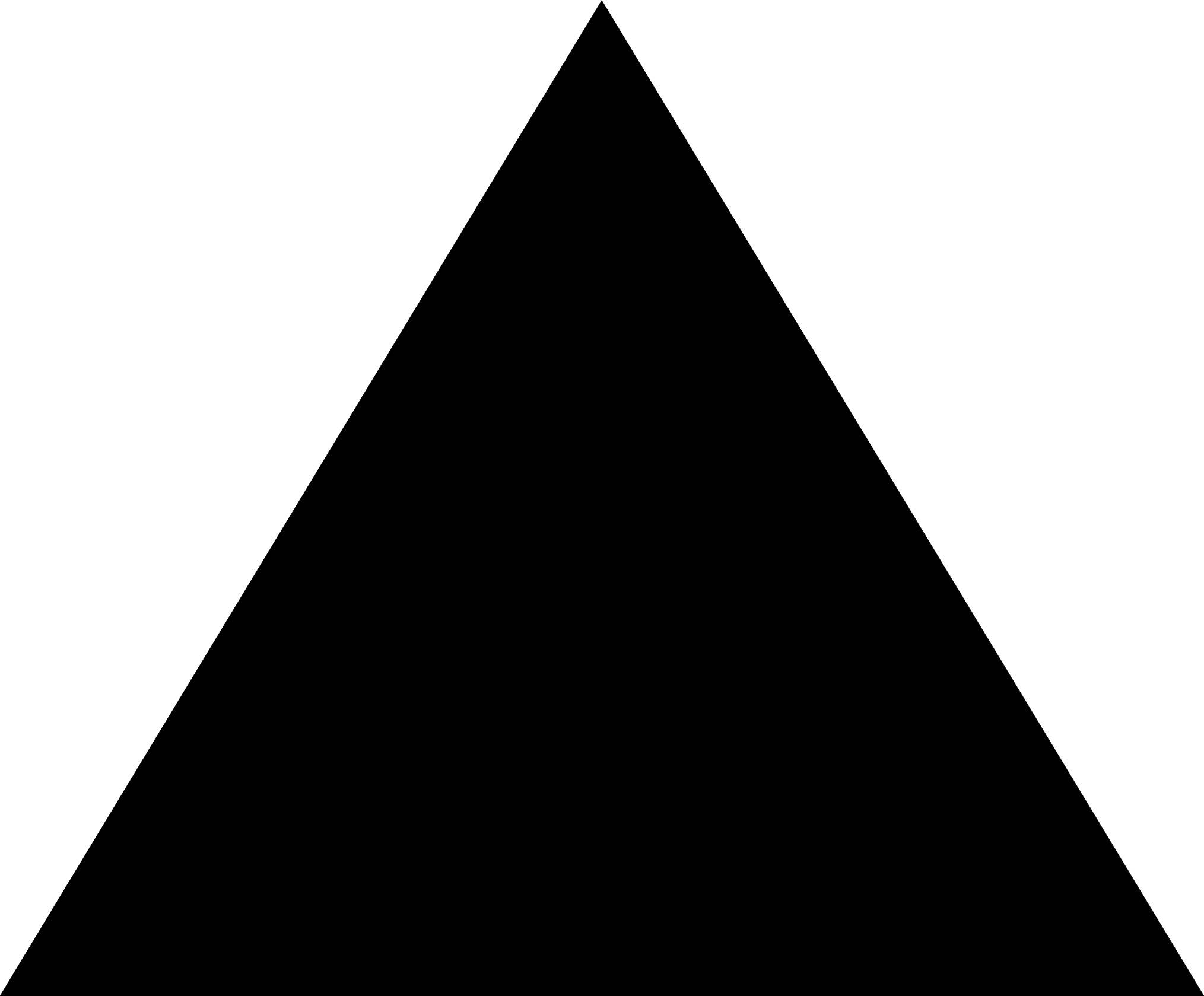}\hs{2}
\includegraphics[height=25mm]{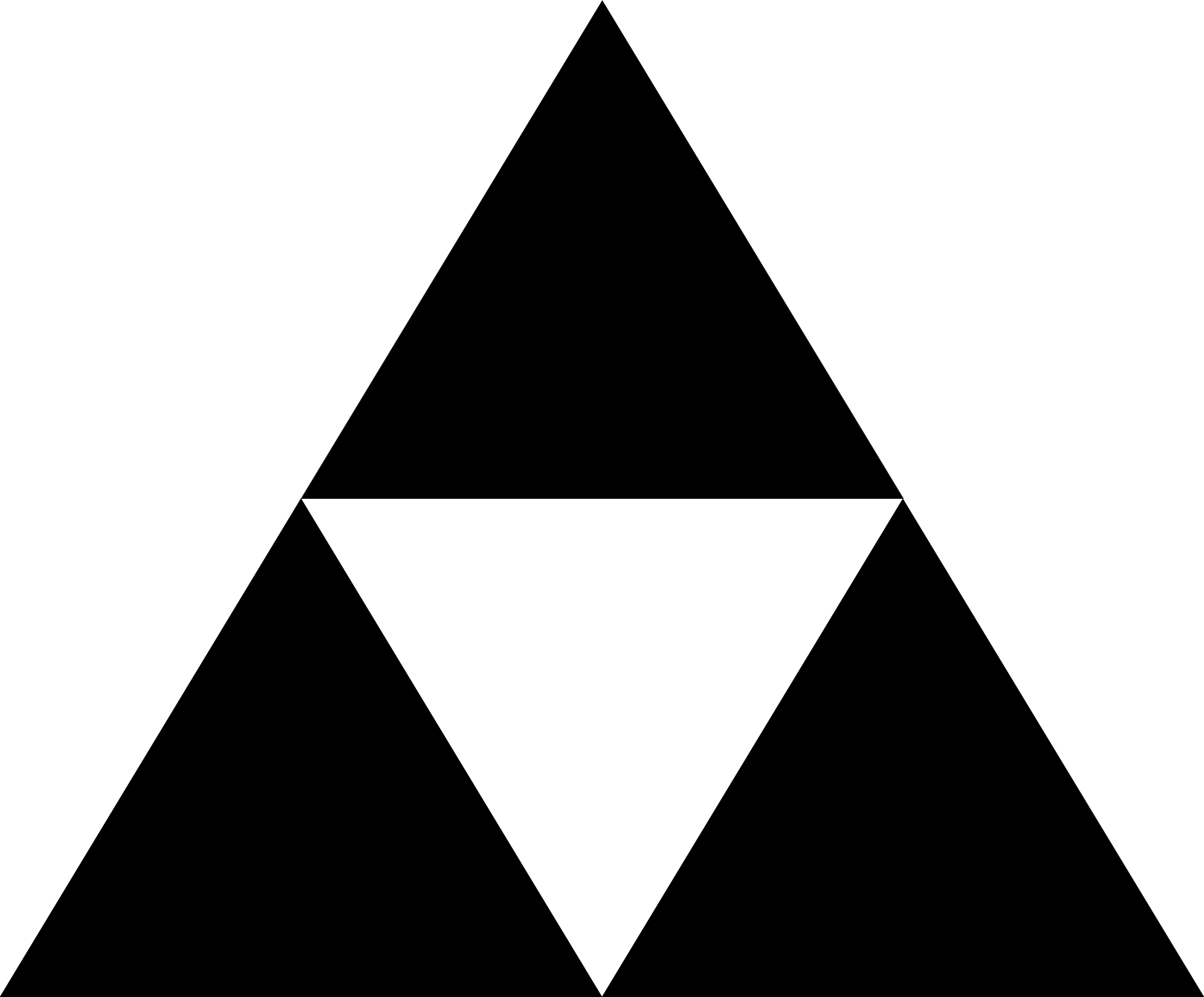}\hs{2}
\includegraphics[height=25mm]{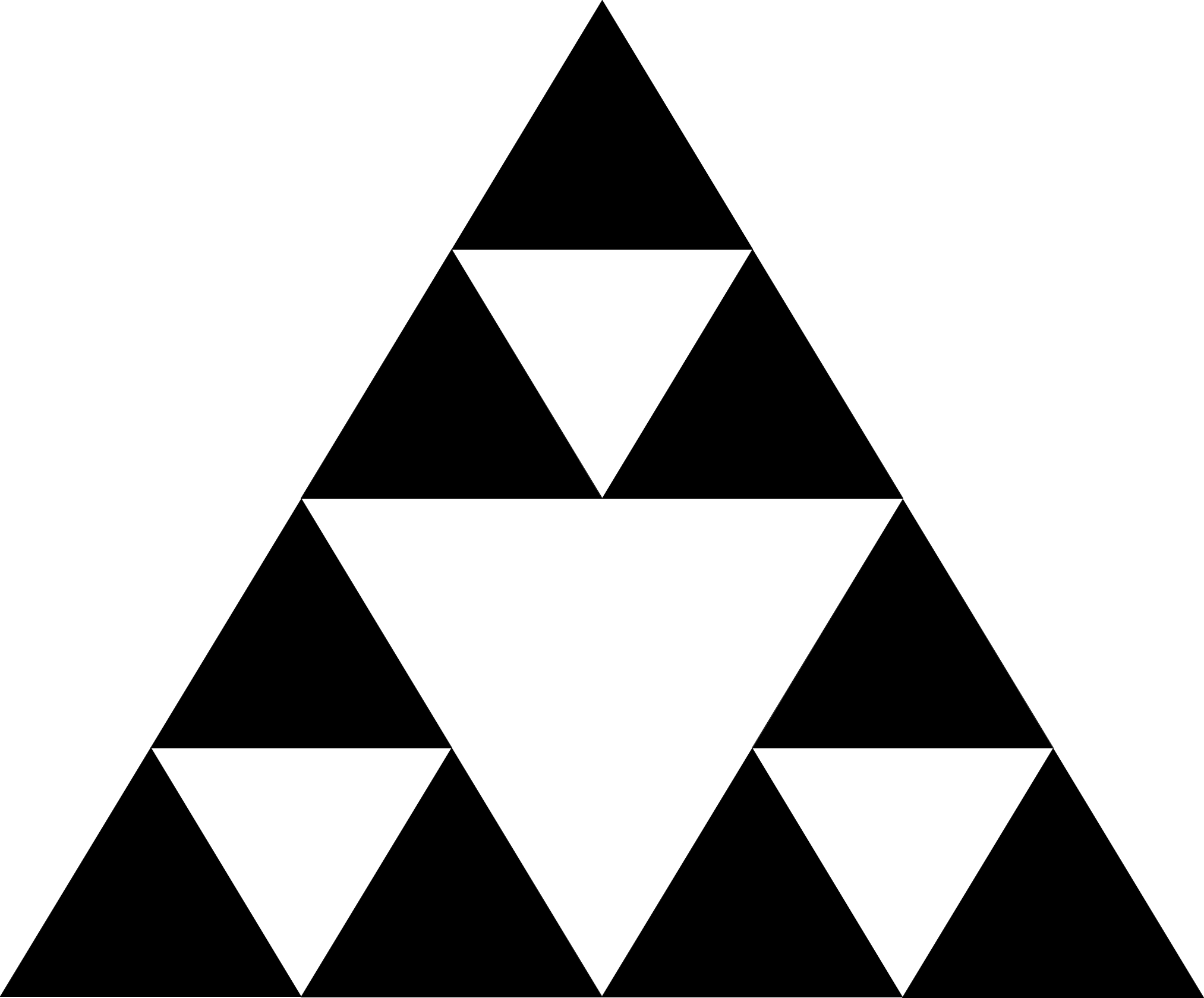}\hs{2}
\includegraphics[height=25mm]{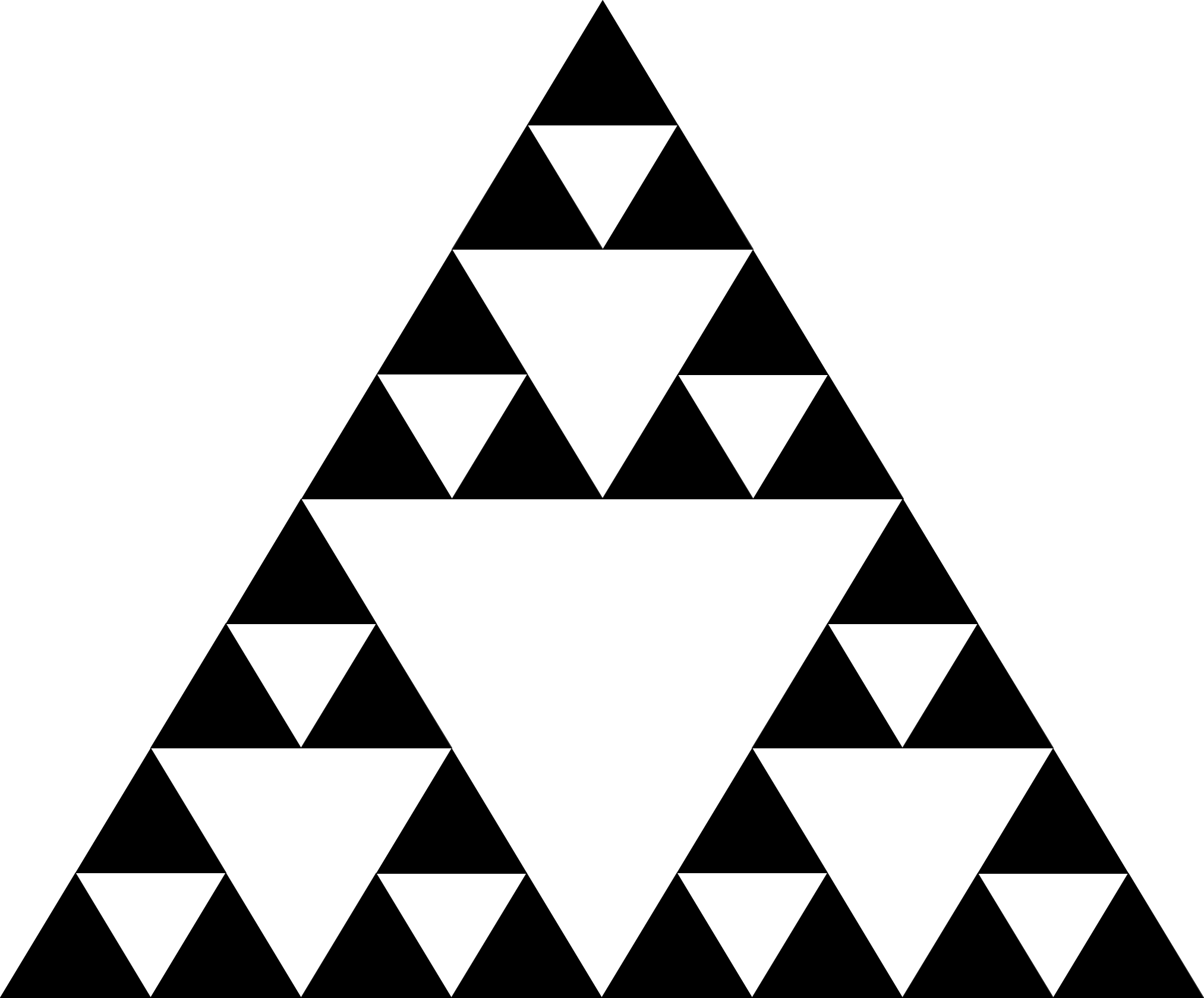}
\end{center}
\caption{\label{fig:Sierpinksi} The first four prefractal approximations to the Sierpinksi triangle.}
\end{figure}

\begin{example}[Scattering by a Sierpinski triangle and its prefractal approximations] \label{ex:sier} Suppose that $n=3$ and $\mS\subset \Gamma_\infty$ is a Sierpinski triangle, the compact set defined by $\mS = \cap_{j=1}^\infty \mS_j$ with $\mS_1\supset \mS_2\supset ...$ the standard sequence of (closed) prefractal approximations to $\mS$, the first four of these shown in Figure \ref{fig:Sierpinksi}. It is clear that $\mS^\circ = \emptyset$, indeed that $m(\mS)=0$; further \cite[Example 9.4]{Fal} $\dimH(\mS)=\log 3/\log 2 >1$. Thus,  by Theorem \ref{thm:null}(d) and (e), $\mS$ is $1/2$-null, i.e.\ $H^{1/2}_\mS=\{0\}$, while $\mS$ is not $-1/2$-null, i.e.\ $H^{-1/2}_\mS \neq \{0\}$. Each prefractal $\mS_j$ is $C^0$ except at finitely many points, in the terminology of Theorem \ref{thm:C0}, and $\mS_j^\circ$ is a union of finitely many Lipschitz open sets.

As $V^+\subset H^{1/2}_\omS=\{0\}$, the family of formulations $\sS\sN(V^+)$ for the sound-hard scattering problem collapses to a single formulation with the trivial solution $u^s=0$. By Corollary \ref{cor:lgs}(b), this is also the solution to \SNws and the limiting geometry solution in the sense of Definition \ref{def:lgclosed}.   Let $u^s_j$ denote the solution to \SNws with $\mS$ replaced by the prefractal $\mS_j$. By Theorems \ref{thm:s62} and \ref{thm:C0}, $u^s_j$ is also the solution to $\sS\sN(V^+)$ for $V^+=H^{1/2}_{\mS_j}=\tH^{1/2}(\mS_j^\circ)$, and (assuming $\Delta u^i + k^2 u^i = 0$ in a neighbourhood of $\mS$) $u^s_j$ also satisfies \SNcls by Lemma \ref{lem:wcl}; further \SNcls is well-posed by Theorem \ref{thm:wpst}(b). Thus all the formulations we have discussed are well-posed and have the same unique solution for the prefractal $\mS_j$. By Theorem \ref{thm:limits} and Remark \ref{rem:pw}, $u^s_j\to 0$ as $j\to \infty$ uniformly on compact subsets of $D:= \R^3\setminus \mS$, and locally in $W^1$ norm.

As $\tH^{-1/2}(\mS^\circ)=\{0\}\neq H^{-1/2}_\mS$, there are, by Lemma \ref{lem:cardc}, infinitely many (with cardinality $\mathfrak{c}$) distinct formulations $\sS\sD(V^-)$ for the sound-soft scattering problem that satisfy \eqref{eq:selection}. We have shown in Theorem \ref{thm:S61} for the case of an incident plane wave \eqref{eq:pw} that, for almost all incident directions $\bd$,  there are infinitely many distinct solutions to these formulations.  By Corollary \ref{cor:lgs}(b), the solution $u^s$ to the particular formulation $\sS\sD(V^-)$ with $V^-= H^{-1/2}_\mS$ is the limiting geometry solution in the sense of Definition \ref{def:lgclosed}, and this is also the unique solution of \SDw.  This solution $u^s$ is, by Theorem \ref{strength}, non-zero if $u^i$ is $C^\infty$ in a neighbourhood of $\mS$ and $u^i(\bx)\neq 0$ for $\bx \in \mS$, for example if $u^i$ is an incident plane wave. Let $u^s_j$ denote the solution to \SDws with $\mS$ replaced by $\mS_j$. Equivalently, by Theorems \ref{thm:S61} and \ref{thm:C0}, $u^s_j$ is the solution to $\sS\sD(V^-)$ for $V^-=H^{-1/2}_{\mS_j}=\tH^{-1/2}(\mS_j^\circ)$, and (assuming $\Delta u^i + k^2 u^i = 0$ in a neighbourhood of $\mS$) $u^s_j$ also satisfies \SDcls by Lemma \ref{lem:wcl}; further \SDcls is well-posed by Theorem \ref{thm:wpst}(a). By Theorem \ref{thm:limits} and Remark \ref{rem:pw},
$u^s_j\to u^s$ as $j\to \infty$ uniformly on compact subsets of $D:= \R^3\setminus \mS$, and locally in $W^1$ norm.
\end{example}

In the case of sound-soft scattering and $n=3$, some of the results of the following example can be found in \cite[Example 4.4]{ChaHewMoi:13} for the case when the wavenumber $k$ is complex, with $\Im(k)>0$.

\begin{figure}[!t]
{\includegraphics[scale=0.31]{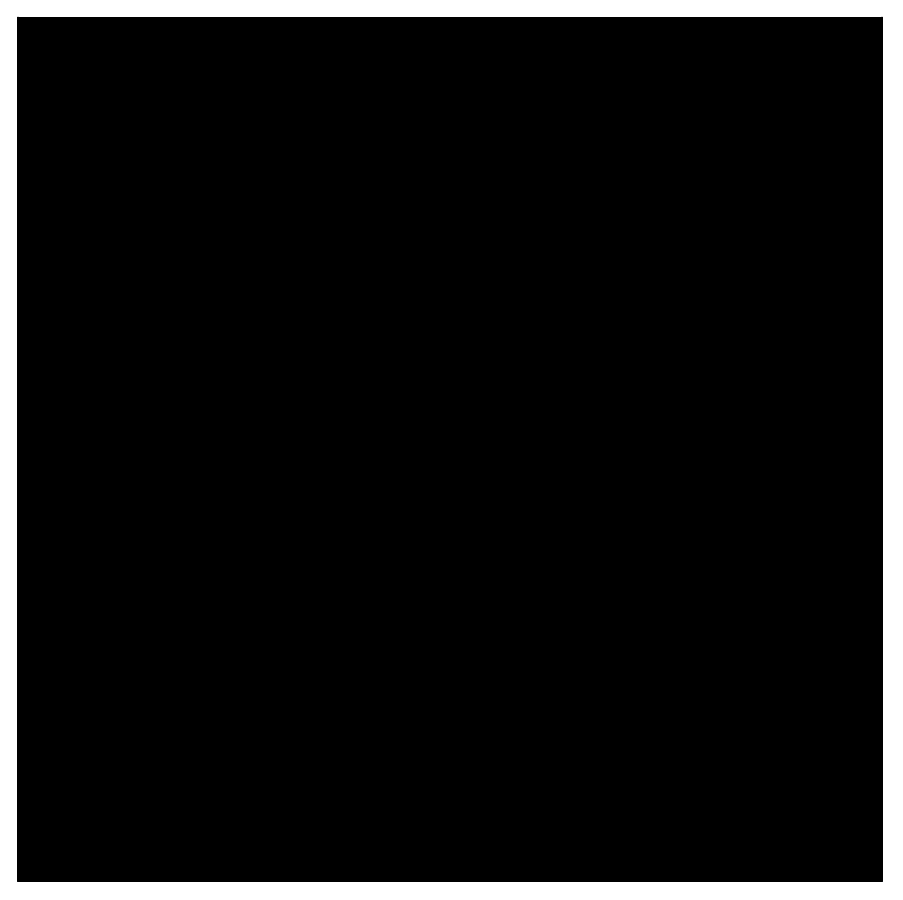}
\hs{4}\includegraphics[scale=0.31]{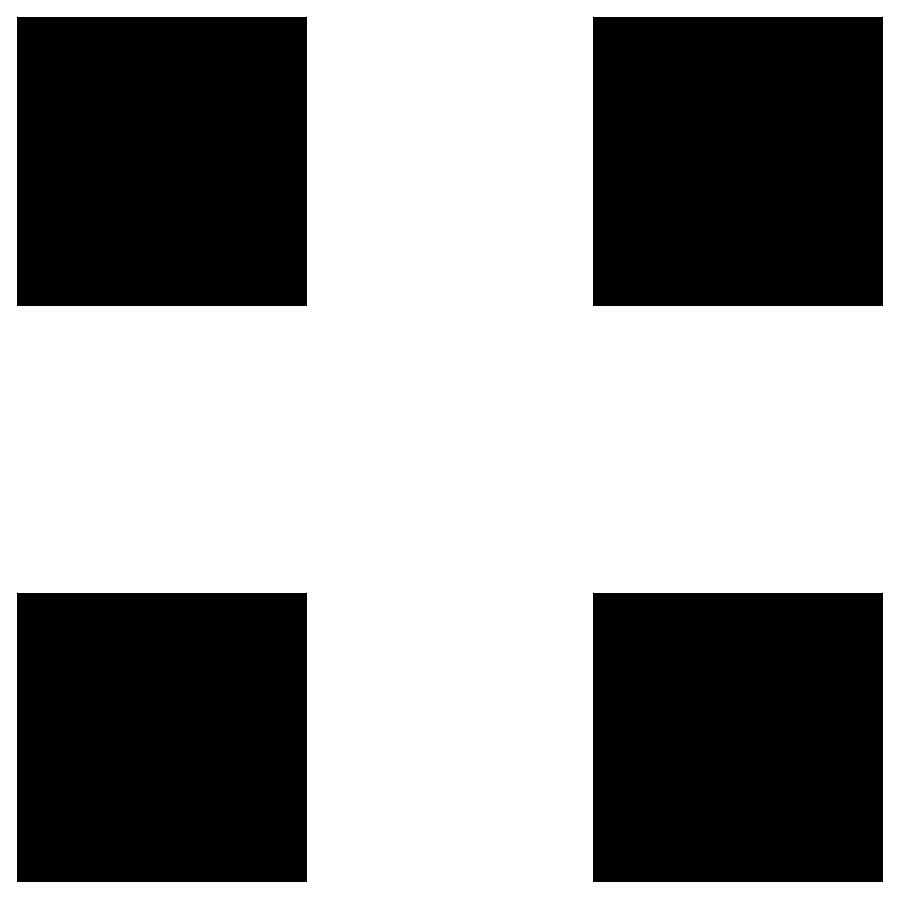}\hs{4}\includegraphics[scale=0.31]{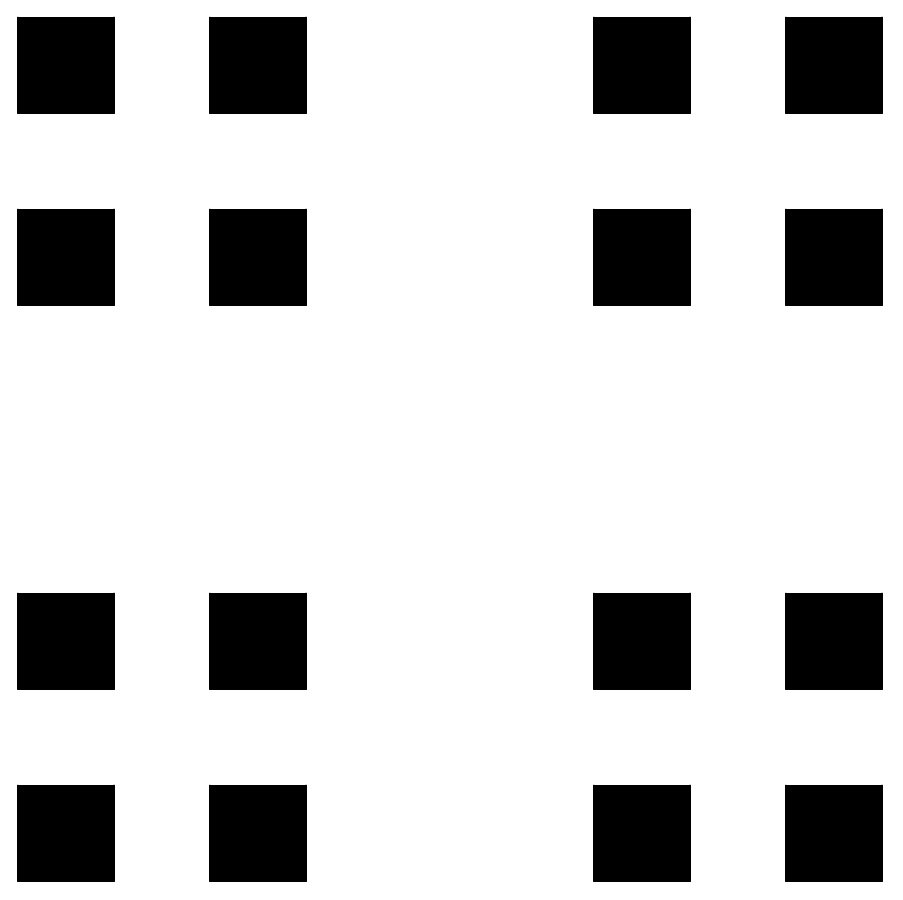}\hs{4}\includegraphics[scale=0.31]{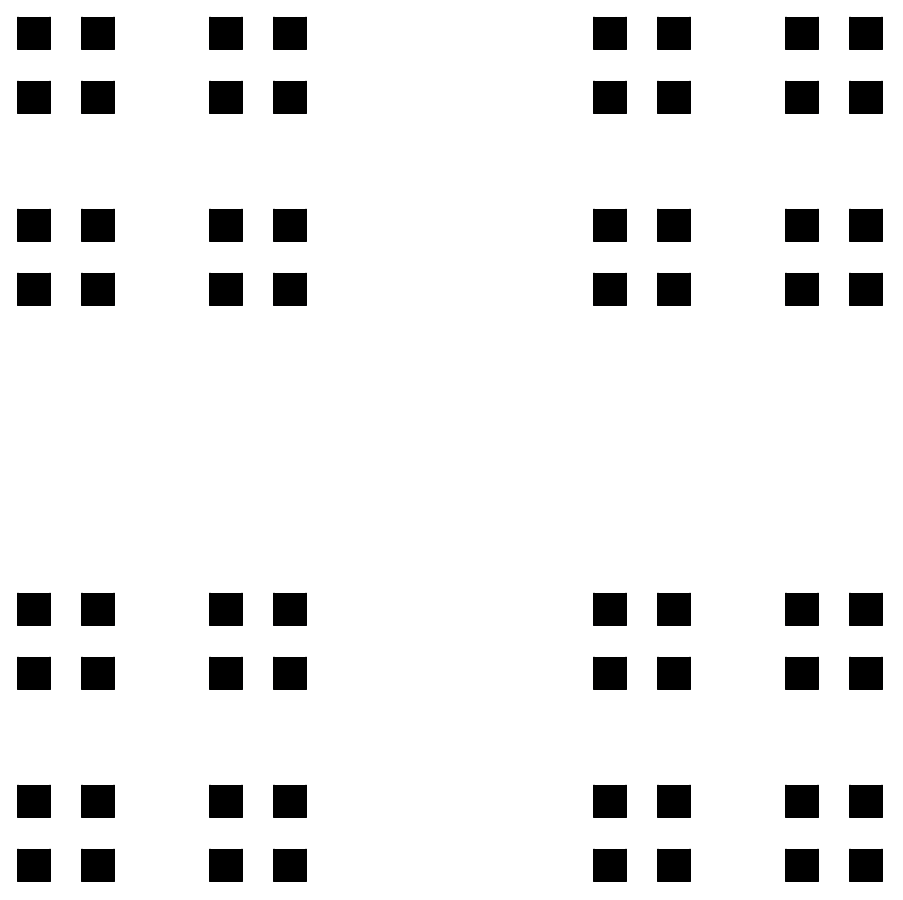}\hs{4}
}
\centering
\caption{The first four prefractal approximations to the standard two-dimensional middle-third Cantor set (or Cantor dust).}%
\label{fig:CantorDust}
\end{figure}

\begin{example}[Scattering by a Cantor set or Cantor dust] \label{ex:cantor} Suppose that $n=2$ or $3$ and, for $j\in \N$, let $\mS_j:= C_j$, where
\begin{equation*}%
C_j := \big\{(\tilde \bx, 0):\tilde \bx\in E_{j-1}^{n-1}\big\}\subset \Gamma_\infty,
\end{equation*}
with  $\R\supset E_0\supset E_1\supset\ldots$ the standard recursive sequence generating the one-dimensional ``middle-$\lambda$'' Cantor set, for some $0<\lambda<1$ \cite[Example 4.5]{Fal}. Where $\alpha=(1-\lambda)/2\in (0,1/2)$, explicitly $E_0=[0,1]$, $E_1=[0,\alpha]\cup[1-\alpha,1]$, $E_2=[0,\alpha^2]\cup [\alpha-\alpha^2,\alpha]\cup [1-\alpha,1-\alpha+\alpha^2]\cup[1-\alpha^2,1]$, $...$, so that $E_j\subset \R$ is the closure of a Lipschitz open set that is the union of $2^j$ open intervals of length $l_j=\alpha^j$, while $E_j^2\subset \R^2$ is the closure of a Lipschitz open set that is the union of $4^j$ squares of side-length $l_j$. Figure \ref{fig:CantorDust} visualises $E_0^2,\ldots, E_3^2$ in the classical ``middle third'' case $\alpha =\lambda = 1/3$.

Define the compact set $\mS\subset \Gamma_\infty$ by $\mS = \cap_{j=1}^\infty \mS_j$.  If $n=2$, $\mS$ is the (one-dimensional) middle-$\lambda$ Cantor set, with \cite[Example 4.5]{Fal} $\dimH(\mS) = \log(2)/\log(1/\alpha)$ $>0$. If $n=3$ then $\mS$ is the associated two-dimensional Cantor set (or ``Cantor dust''), with \cite[Example 4.5, Corollary 7.4]{Fal} $\dimH(\mS) = 2\log(2)/\log(1/\alpha)\in(0,2)$. Thus, by Theorem \ref{thm:null}(e), $\mS$ is not $-1/2$-null if $n=2$, or if $n=3$ and $\alpha > 1/4$, but is $-1/2$-null if $n=3$ and $\alpha<1/4$; a more detailed analysis \cite[Theorem 4.5]{HewMoi:15} establishes that $\mS$ is $-1/2$-null also for $n=3$ and $\alpha = 1/4$. For both $n=2$ and $3$, $m(\mS)=0$ and $\mS^\circ = \emptyset$, so that $\mS$ is $1/2$-null by Theorem \ref{thm:null}(e).

The second paragraph of Example \ref{ex:sier} applies verbatim also in this case: in particular, the solution to the sound-hard scattering problem $\sS\sN(V^+)$ with $V^+=H^{1/2}_\mS$ is again $u^s=0$.

The third paragraph of Example \ref{ex:sier} also applies verbatim if $n=2$, or $n=3$ and $\alpha>1/4$: in particular, $u^s$, the solution to $\sS\sD(V^-)$ with $V^-=H^{-1/2}_\mS$, or equivalently to \SDw, is non-zero as long as $u^i$ is $C^\infty$ in a neighbourhood of $\mS$ and is non-zero everywhere on $\mS$.  The statements in the third paragraph about $u^s_j$ and its convergence to $u^s$ apply also when $n=3$ and $\alpha\leq 1/4$, but, since $\mS$ is $-1/2$-null in this case, every $V^-=\{0\}$ so that the formulations $\sS\sD(V^-)$ collapse to a single formulation with the trivial solution $u^s=0$. By Corollary \ref{cor:lgs}(b) this is also the solution to \SDws and is also the limiting geometry solution of Definition \ref{def:lgclosed}.
\end{example}

Our next example is a screen with empty interior but positive surface measure which is not $1/2$-null. Our ``Swiss cheese'' construction follows that of Polking, who used it in \cite[Theorem 4]{Po:72a} to construct explicitly a compact set $F\subset \R^n$ with empty interior  that is not $n/2$-null.

\begin{example}[Scattering by a ``Swiss cheese'' screen] \label{ex:swiss} By a {\em Swiss cheese} screen we mean, for $n=2,3$, a compact subset $\mS$ of $\Gamma_\infty$ constructed as follows: take a bounded open set $\Omega\subset \Gamma_\infty$, and sequences $(\bx_j)_{j=1}^\infty \subset \Omega$ and $(r_j)_{j=1}^\infty\subset (0,\infty)$, and define $\mS_j:= F_j$ for $j\in \N$, where
$$
F_j := \overline{\Omega} \setminus \bigcup_{m=1}^j B_{r_m}(\bx_m),
$$
and set $\mS := \cap_{j=1}^\infty \mS_j$.
If the sequence $(\bx_m)_{m=1}^\infty$ is dense in $\Omega$, then $\mS^\circ$ is empty. But $\mS$ need not be empty: indeed,
if the radii $r_m$ are sufficiently small and decrease sufficiently rapidly, then $\mS$ has positive measure, since
$$
m(\mS) \geq m(\overline{\Omega}) - 2(\pi/2)^{n-1}\sum_{j=1}^\infty r_m^{n-1}, \quad \mbox{ for } n=2,3.
$$
The condition $m(\mS)>0$ is necessary (Theorem \ref{thm:null}(e)), but is not sufficient to ensure that $\mS$ is not $1/2$-null. But if the radii are small enough and decrease sufficiently rapidly then, indeed, $\mS$ is not $1/2$-null. It is shown in \cite[Theorem 4.6]{HewMoi:15} that for every open $\Omega\subset \Gamma_\infty$ there exists $\epsilon>0$ such that $\mS$ is not $1/2$-null provided $\sum_{m=1}^\infty r_m \leq \epsilon$ if $n=3$, provided each $r_m<2$ and  $\sum_{m=1}^\infty [\log(2/r_m)]^{-1} \leq \epsilon$ if $n=2$. The choice $r_m = 6\epsilon/(\pi m)^2$ works for $n=3$, the choice $r_m = 2\exp(-\pi^2m^2/(6\epsilon))$ if $n=2$.

So let us assume that $(\bx_m)$ is chosen to be dense in $\Omega$, so that $\mS^\circ$ is empty, and also that the radii $(r_m)$ are chosen so that $\mS$ is not $1/2$-null, which implies (Theorem \ref{thm:null}(e)) that $\mS$ is not $-1/2$-null, that $m(\mS)>0$ and that $\mS$ is non-empty.
Let us also assume that $\Omega$ is a Lipschitz open set. This implies that $\mS_j$ is in the algebra of subsets of $\R^n$ generated by the Lipschitz open sets so that, by Theorem \ref{thm:null}(f), $\partial \mS_j$ is $\pm1/2$-null. Further, it ensures that $\overline{\mS^\circ_j}=\mS_j$ and that  $\mS^\circ_j$ is $C^0$ except at a finite number of points. %

With the above assumptions, the third paragraph of Example \ref{ex:sier} applies verbatim to this case. The comments in the second paragraph of Example \ref{ex:sier} about the well-posedness and equivalence of all formulations for the sound-hard problem when $\mS$ is replaced by $\mS_j$ also apply here. But, as $H^{1/2}_\mS\neq \{0\}=\tH^{1/2}(\mS^\circ)$ for this Swiss cheese screen, there are, by Lemma \ref{lem:cardc}, infinitely many (with cardinality $\mathfrak{c}$) distinct formulations $\sS\sN(V^+)$ for the sound-hard problem that satisfy \eqref{eq:selection}. Further, by Theorem \ref{thm:s62}, for the case of an incident plane wave \eqref{eq:pw} and for almost all incident directions $\bd$,  there are infinitely many distinct solutions to these formulations.  By Corollary \ref{cor:lgs}(a), the solution $u^s$ to the particular formulation $\sS\sN(V^+)$ with $V^+=H^{1/2}_\mS$ is the limiting geometry solution in the sense of Definition \ref{def:lgclosed}, and this is also the unique solution of \SNw.  This solution is non-zero, by Theorem \ref{strength}, as long as $u^i$ is $C^\infty$ in a neighbourhood of $\mS$ and $\partial u^i/\partial x_n$ is non-zero everywhere on $\mS$. By Theorem \ref{thm:limits}(b),  $u^s_j\to u^s$, uniformly on compact subsets of $D$ and locally in $W^1$ norm, as $j\to\infty$.
\end{example}

Our remaining four examples consider screens $\mS$ that are non-$C^0$ open sets. In Example \ref{ex:irr} the screen is not $C^0$ but its boundary is sufficiently well-behaved so that \SDw, $\sS\sD(V^-)$  subject to \eqref{eq:selection}, and \SDcls all have the same unique solution in the sound-soft case, and \SNw, $\sS\sN(V^+)$ subject to \eqref{eq:selection}, and \SNcls all have the same unique solution in the sound-hard case. In  Examples \ref{ex:openminuscantor} and \ref{ex:openminuscheese} the classical formulations \SDcls and \SNcls both fail to be well-posed. In each case $\omS$ is an interval for $n=2$, a square for $n=3$, and the sound-hard scattered field for $\mS$ in the limiting geometry sense, equivalently the solution to our new formulation $\sS\sN(V^+)$ with $V^+=\tH^{1/2}(\mS)$, is different from the scattered field for $\omS$, which satisfies \SNw. In Example \ref{ex:openminuscheese} the same effects are seen for the sound-soft case.

 \begin{figure}
{\includegraphics[scale=0.2]{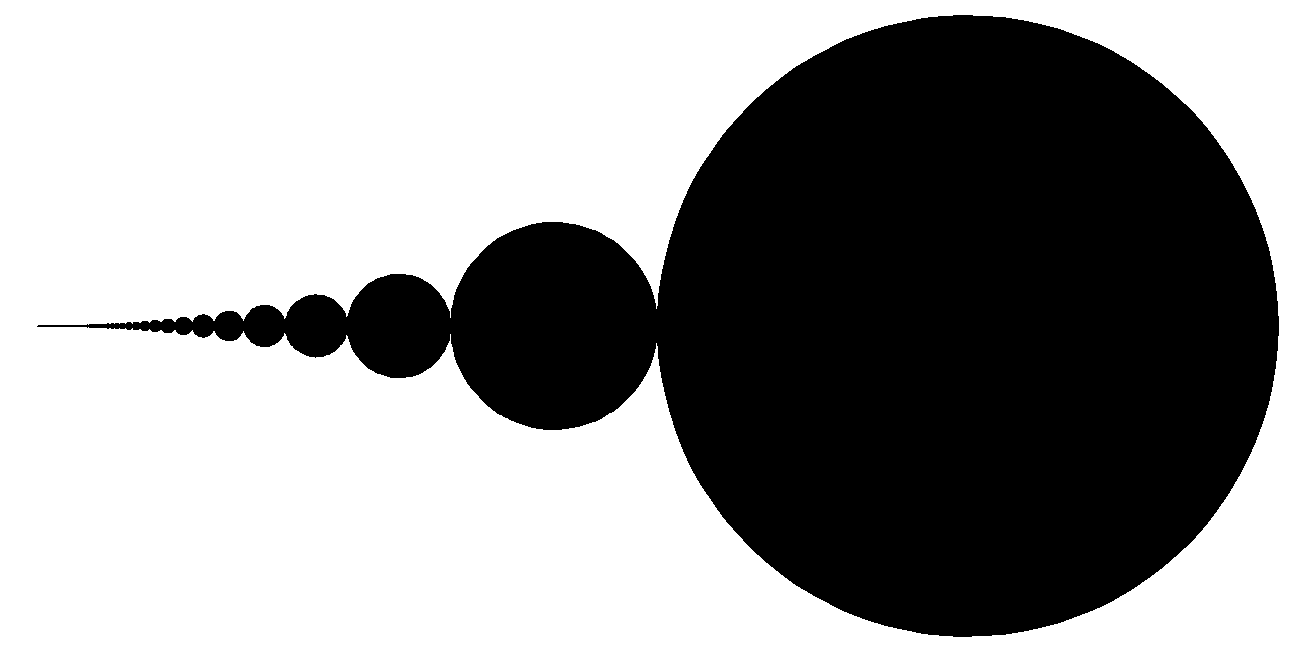}}
\centering
\caption{The irregular screen, consisting of a countable number of circles, of Example \ref{ex:irr}, this an example of a non-Lipschitz (indeed non-$C^0$) screen for which the classical formulations remain well-posed and equivalent to the weak formulations and to the standard BIEs \eqref{eq:bie}.}
\label{fig:irregular}
\end{figure}

\begin{example}[An irregular screen where all formulations are well-posed and coincide] \label{ex:irr} Suppose that $n=3$ and, for $j\in \N$, let $s_j := (2j+1)/(2j(j+1))$, $r_j:= 1/(2j(j+1))$, $\bx_j := (s_j,0,0)$, and (see Figure \ref{fig:irregular})
$$
\mS:= \Gamma_\infty \cap \bigcup_{j=1}^\infty B_{r_j}(\bx_j), \quad \mbox{ so that } \quad \partial \mS= \{\boldsymbol{0}\}\cup \left(\Gamma_\infty \cap \bigcup_{j=1}^\infty \partial B_{r_j}(\bx_j)\right).
$$
Then $\mS$ is not $C^0$, but is $C^0$ except at the countable set of points $\{\boldsymbol{0}\}\cup \{(1/(j+1),0,0):j\in \N\}$, which has the single limit point $\boldsymbol{0}$. Clearly also $\partial \mS \subset \cup_{j=1}^\infty \partial E_j$ with each $E_j$ a Lipschitz open subset of $\Gamma_\infty$. Thus $\tH^{\pm 1/2}(\mS) = H^{\pm 1/2}_\omS$ by Theorem \ref{thm:C0}, and also $\partial \mS$ is $\pm1/2$-null by Theorem \ref{thm:null}(e)-(g).

The solution $u^s$ to the sound-hard scattering problem \SNws is, by Theorem \ref{thm:s62}, also the solution to $\sS\sN(V^+)$ for $V^+=\tH^{1/2}(\mS)=H^{1/2}_{\omS}$, and (assuming $\Delta u^i + k^2 u^i = 0$ in a neighbourhood of $\omS$) $u^s$ also satisfies \SNcls by Lemma \ref{lem:wcl}; further \SNcls is well-posed by Theorem \ref{thm:wpst}(b).  The solution $u^s$ to the sound-soft scattering problem \SDws is, by Theorem \ref{thm:S61}, also the solution to $\sS\sD(V^-)$ for $V^-=\tH^{-1/2}(\mS)=H^{-1/2}_{\omS}$, and (assuming $\Delta u^i + k^2 u^i = 0$ in a neighbourhood of $\omS$) $u^s$ also satisfies \SDcls by Lemma \ref{lem:wcl}; further \SDcls is well-posed by Theorem \ref{thm:wpst}(a).
\end{example}

\begin{example}[Scattering by a solid screen with a Cantor set or dust removed] \label{ex:openminuscantor} Suppose that $n=2$ or $3$ and let $\mS_0:= \{\bx=(x_1,...,x_{n-1},0)\in \Gamma_\infty: 0<x_m<1, m = 1,...,n-1\}$. For $j\in \N$, let $\mS_j := \mS_0\setminus C_j$, where $C_j$ is as defined in Example \ref{ex:cantor}, so that $\mS_j$ is a Lipschitz open set. Let $\mS := \cup_{j=1}^\infty \mS_j = \mS_0 \setminus C$, where $C:= \cap_{j=1}^\infty C_j$ is the screen studied in Example \ref{ex:cantor}, i.e.\ the middle-$\lambda$ Cantor set for $n=2$, the corresponding Cantor dust for $n=3$.

We note that  $\mS^\circ = \mS$, and that $C\setminus \partial \mS_0 \subset \omS^\circ\setminus \mS\subset C$. Since $m(\omS^\circ\setminus \mS)\leq m(C) = 0$ and $\omS^\circ=\mS_0$ is a Lipschitz open set, $\tH^{-1/2}(\mS^\circ)= H^{-1/2}_\omS$ by Corollary \ref{prop:TildeSubscript}(iii). Further $\tH^{\pm 1/2}(\omS^\circ)= H^{\pm 1/2}_\omS$ by Theorem \ref{thm:C0}. As noted in Example \ref{ex:cantor}, $C$ is $-1/2$-null if $n=3$ and $\alpha=(1-\lambda)/2\leq 1/4$, so that  $\omS^\circ \setminus \mS^\circ \subset C$ is $-1/2$-null, and $\tH^{1/2}(\mS^\circ)= H^{1/2}_\omS$ by Corollary \ref{lem:equalityNullity}. As also noted in Example \ref{ex:cantor}, if $n=2$ then $\dimH(C)>0$ so that $\dimH(\omS^\circ\setminus \mS^\circ)> \dimH(C\setminus \partial \mS_0)>0$ since $\dimH(\partial \mS_0)=0$, while if $n=3$ and $\alpha>1/4$ then $\dimH(C)>1$ so that $\dimH(\omS^\circ\setminus \mS^\circ)> \dimH(C\setminus \partial \mS_0)>1$ since $\dimH(\partial \mS_0)=1$. Thus, by Corollary \ref{prop:TildeSubscript}(iii), $\tH^{1/2}(\mS^\circ)\neq H^{1/2}_\omS$ if $n=2$, or if $n=3$ and $\alpha>1/4$.

Since $\tH^{-1/2}(\mS^\circ)= H^{-1/2}_\omS$, the formulations $\sS\sD(V^-)$ for sound-soft scattering by $\mS$ that satisfy \eqref{eq:selection} collapse onto a single formulation with a single unique solution $u^s$, and this solution, by Corollary \ref{cor:lgs}(a), is also the unique solution of \SDw, and the limiting geometry solution in the sense of Definition \ref{def:lgopen}.  In particular, if $u^s_j$ denotes the solution, when $\mS$ is replaced by $\mS_j$, to \SDws or \SDcls (all formulations have the same solution for the screen $\mS_j$, as in the other examples above),  $u^s_j\to u^s$, uniformly on compact subsets of $D$ and locally in $W^1$ norm as $j\to\infty$, by Theorem \ref{thm:limits}(a).

Since $\mS_0$ is Lipschitz, $\mS_0=\omS^\circ$, and $\tH^{-1/2}(\omS^\circ)= H^{-1/2}_\omS$, $u^s$ is also the unique solution to (any of the formulations) for sound-soft scattering by $\mS_0$, including \SDws and \SDcl. Thus the limiting geometry solution for sound-soft scattering by the screen $\mS = \mS_0\setminus C$ is the same as the solution for the screen $\mS_0$: the fractal ``hole'' $C$ in $\mS$ does not have any effect.

Similar remarks apply in the sound-hard case if $n=3$ and $\alpha\leq 1/4$, for then  $\tH^{1/2}(\mS^\circ)= H^{1/2}_\omS$, and the limiting geometry solution of Definition \ref{def:lgopen} for $\mS$ is just the solution for scattering by the square screen $\mS_0$. But if $n=2$, or $n=3$ and $\alpha> 1/4$, then $\tH^{1/2}(\mS^\circ)\neq H^{1/2}_\omS$ so that, by Lemma \ref{lem:cardc}, there are infinitely many (with cardinality $\mathfrak{c}$) distinct formulations $\sS\sN(V^+)$ that satisfy \eqref{eq:selection}. Further, by Theorem \ref{thm:s62}, for an incident plane wave \eqref{eq:pw} and almost all incident directions $\bd$,  there are infinitely many distinct solutions to these formulations.  By Corollary \ref{cor:lgs}(a), the solution $u^s$ to the formulation $\sS\sN(V^+)$ with $V^+=\tH^{1/2}(\mS)$ is the solution that is the limiting geometry solution in the sense of Definition \ref{def:lgopen}. In particular, if $u^s_j$ denotes the solution, when $\mS$ is replaced by $\mS_j$, to \SNws or \SNcls (all formulations have the same solution for the screen $\mS_j$),  $u^s_j\to u^s$, uniformly on compact subsets of $D$ and locally in $W^1$ norm as $j\to\infty$, by Theorem \ref{thm:limits}(b).

Let $\widetilde u^s$ denote the solution to $\sS\sN(V^+)$ with $V^+= H^{1/2}_\omS$, equivalently, by Corollary \ref{NeuEquivThmm2}, the solution to \SNw. As $\mS_0$ is Lipschitz, $\mS_0=\omS^\circ$, and $\tH^{1/2}(\omS^\circ)= H^{1/2}_\omS$, $\widetilde u^s$ is also the unique solution to (any of the formulations) for sound-hard scattering by $\mS_0$. We expect, generically, that $u^s\neq \widetilde u^s$, and have shown that this is true in Theorem \ref{thm:s62} for plane wave incidence for almost all incident directions. Thus, at least for plane wave incidence and almost all incident directions, the limiting geometry scattered field for $\mS= \mS_0\setminus C$ is different from that for $\mS_0$: the fractal ``hole'' $C$ has an effect, this surprising given that $m(C)=0$ and that $\overline{\mS}=\overline{\mS_0}$, so that the domain $D$ in which \eqref{eq:he} holds is the same for both screens.
\end{example}

\begin{example}[Scattering by a solid screen with a ``Swiss cheese'' set removed] \label{ex:openminuscheese} Suppose that $n=2$ or $3$ and let $\mS_0$ be defined as in the previous example. For $j\in \N$, let $\mS_j := \mS_0\setminus F_j$, where $F_j$ is as defined in Example \ref{ex:swiss}, making the specific choice $\Omega = \mS_0$ so that $F_j\subset \overline{\mS_0}$. Let $\mS := \cup_{j=1}^\infty \mS_j = \mS_0 \setminus F$, where $F:= \cap_{j=1}^\infty F_j$ is the Swiss cheese screen studied in Example \ref{ex:cantor}. Then $\mS = \mS_0\cap \bigcup_{j=1}^\infty B_{r_j}(\bx_j)$, with $(\bx_j)\subset \mS_0$ and $(r_j)\subset (0,\infty)$, and $\mS_j = \mS_0\cap \bigcup_{m=1}^j B_{r_m}(\bx_m)$, for $j\in \N$. We choose $(\bx_j)$ and $(r_j)$ as in the second paragraph of Example \ref{ex:swiss}, in other words so that $(\bx_j)$ is dense in $\mS_0$, which implies that $\overline{\mS} = \overline{\mS_0}$, and so that $F$ is not $1/2$-null. This implies that $m(F)>0$, and also that $\omS^\circ \setminus \mS^\circ =F\setminus \partial \mS_0$ is not $1/2$-null (for otherwise  $F \subset \partial \mS_0 \cup (F\setminus \partial \mS_0)$ is $1/2$-null by Theorem \ref{thm:null}(f) and (h)), and so $\omS^\circ \setminus \mS^\circ$ is not $-1/2$-null. Thus $\tH^{\pm 1/2}(\mS^\circ)\neq H^{\pm 1/2}_\omS$ by Corollary \ref{lem:equalityNullity}.

The behaviour in the sound-hard scattering case is essentially identical to that exhibited, for $n=2$, and for $n=3$ with $\alpha>1/4$,  in the previous example, as, once again, $\tH^{1/2}(\mS^\circ)\neq H^{1/2}_\omS$.
We will not repeat the detail here. The difference between this example and the previous one is in the sound-soft case. In this example,
by Lemma \ref{lem:cardc}, there are infinitely many (with cardinality $\mathfrak{c}$) distinct formulations for the scattering problem $\sS\sD(V^-)$ that satisfy \eqref{eq:selection}. Further, by Theorem \ref{thm:S61}, for an incident plane wave \eqref{eq:pw} and almost all incident directions $\bd$,  there are infinitely many distinct solutions to these formulations.  By Corollary \ref{cor:lgs}(a), the solution $u^s$ to the formulation $\sS\sD(V^-)$ with $V^-=\tH^{-1/2}(\mS)$ is the limiting geometry solution in the sense of Definition \ref{def:lgopen}. In particular, if $u^s_j$ denotes the solution, when $\mS$ is replaced by $\mS_j$, to \SDws or \SDcls (all formulations have the same solution for $\mS_j$),  $u^s_j\to u^s$, uniformly on compact subsets of $D$ and locally in $W^1$ norm as $j\to\infty$, by Theorem \ref{thm:limits}(b).

Let $\widetilde u^s$ denote the solution to $\sS\sD(V^-)$ with $V^-= H^{-1/2}_\omS$, equivalently, by Corollary \ref{DirEquivThmm2}, the solution to \SDw. As $\mS_0$ is Lipschitz, $\mS_0=\omS^\circ$, and $\tH^{-1/2}(\omS^\circ)= H^{-1/2}_\omS$,
$\widetilde u^s$ is also the unique solution to (any of the formulations) for sound-soft scattering by $\mS_0$. For plane wave incidence for almost all incident directions, $u^s\neq \widetilde u^s$, by Theorem \ref{thm:S61}, so that the limiting geometry scattered field for $\mS= \mS_0\setminus F$ is different from that for $\mS_0$: the Swiss cheese ``hole'' $F$ has an effect.  We note that $m(F)>0$, which is suggestive that removing $F$ should have an effect. But on the other hand $\overline{\mS}=\overline{\mS_0}$, so that the domain $D$ in which \eqref{eq:he} holds is the same for the screen $\mS$ as for the screen $\mS_0$. Speaking colloquially, the ``hole'' $F$ in $\mS$ provides no ``clear passageway'' through which a wave can propagate.
\end{example}

Our final example is an open set which is the interior of a Jordan curve, but which is not $C^0$, so that Theorem \ref{thm:C0} does not apply. This is representative of the many examples for which there remain open questions.

\begin{figure}[!t]
{\includegraphics[scale=0.24]{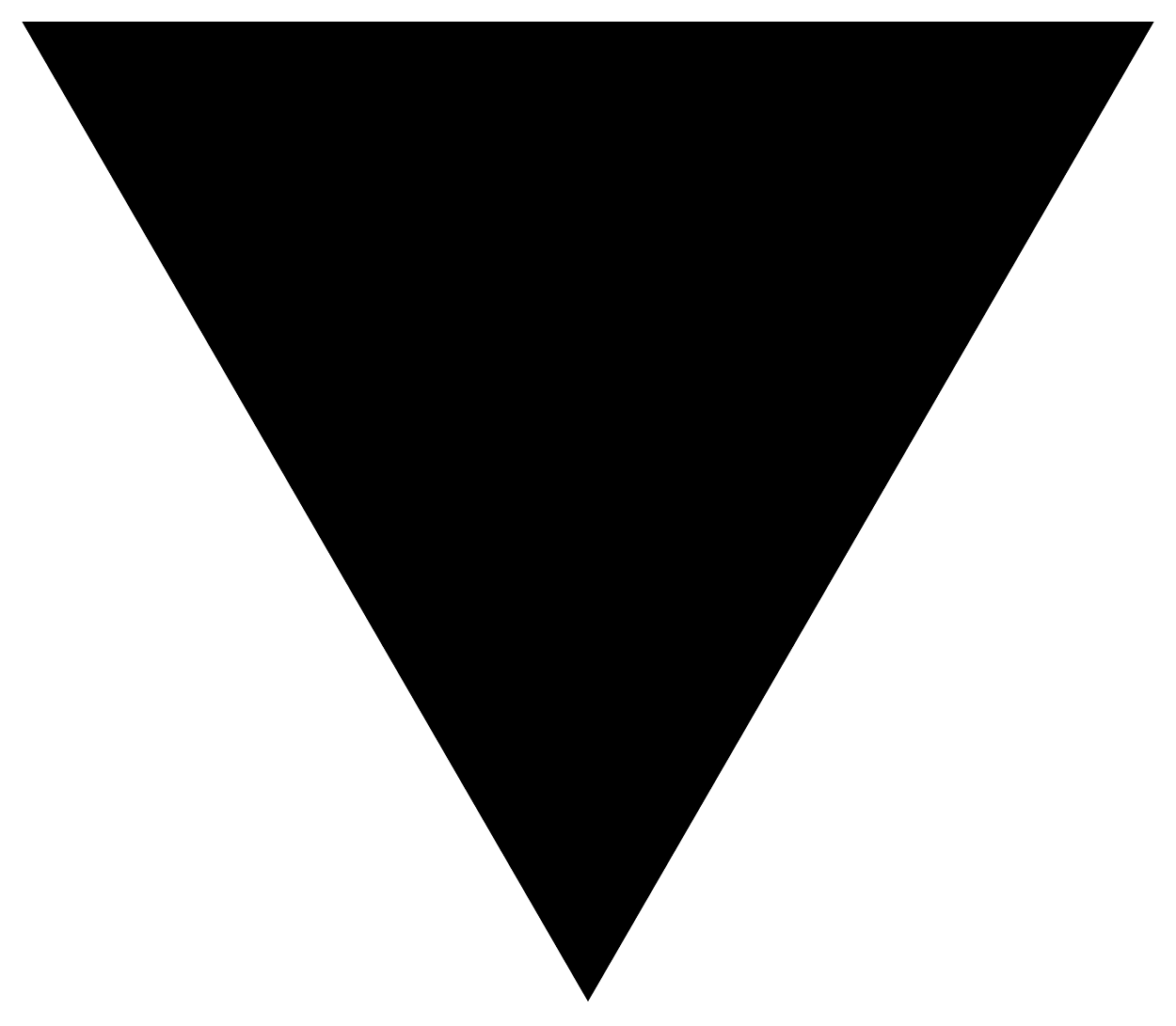}
\hs{2}\includegraphics[scale=0.24]{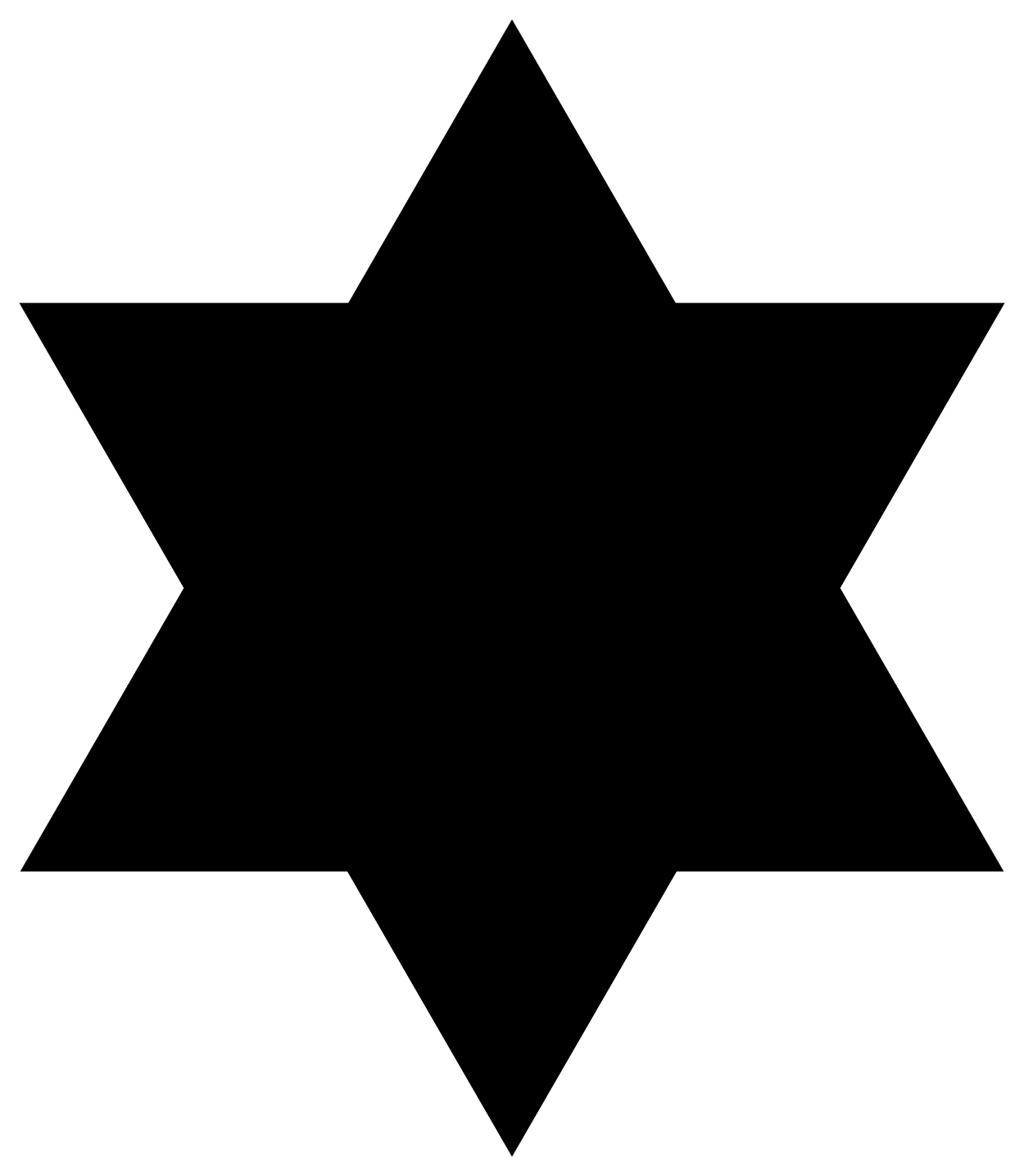}\hs{2}\includegraphics[scale=0.24]{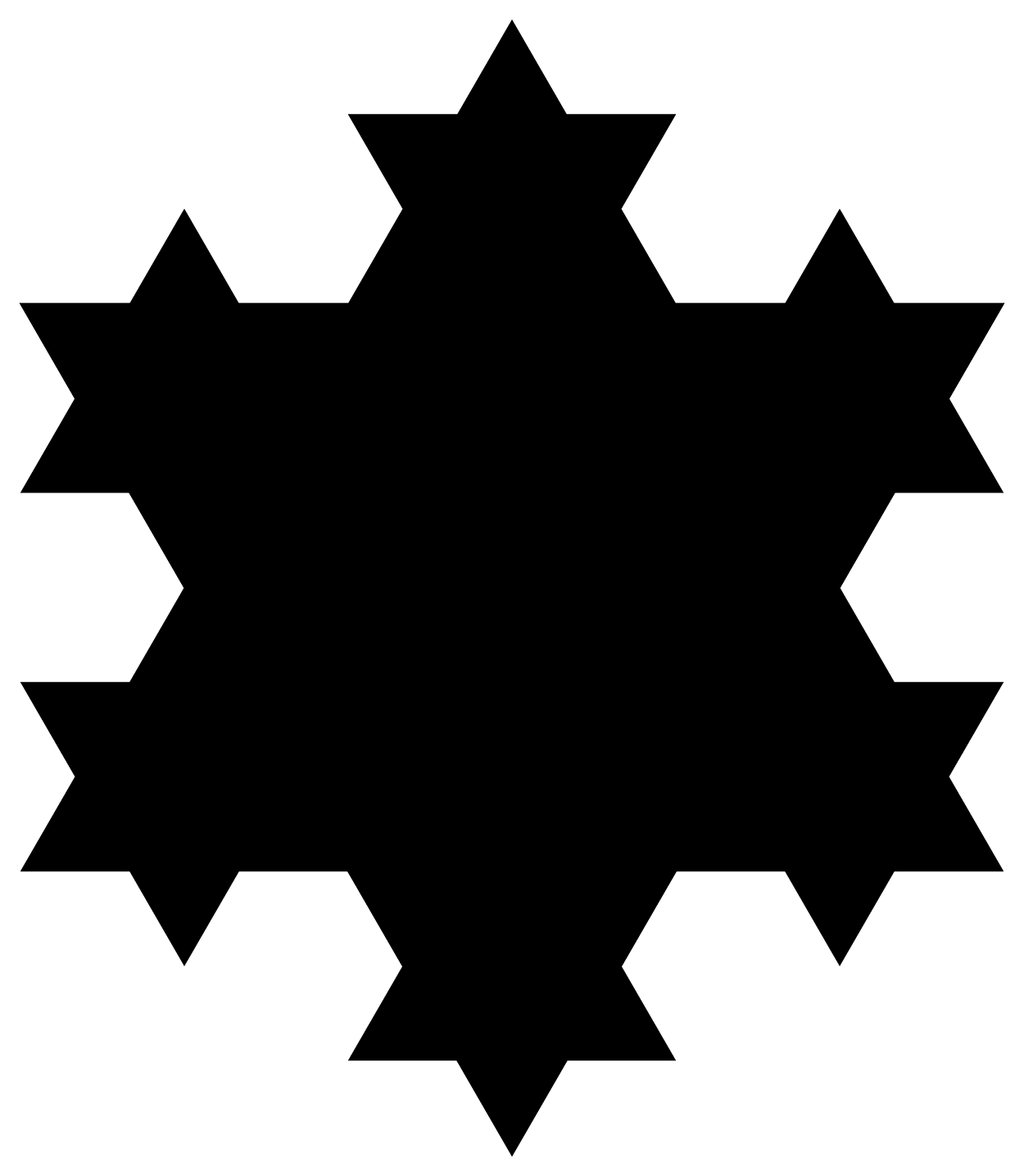}\hs{2}\includegraphics[scale=0.24]{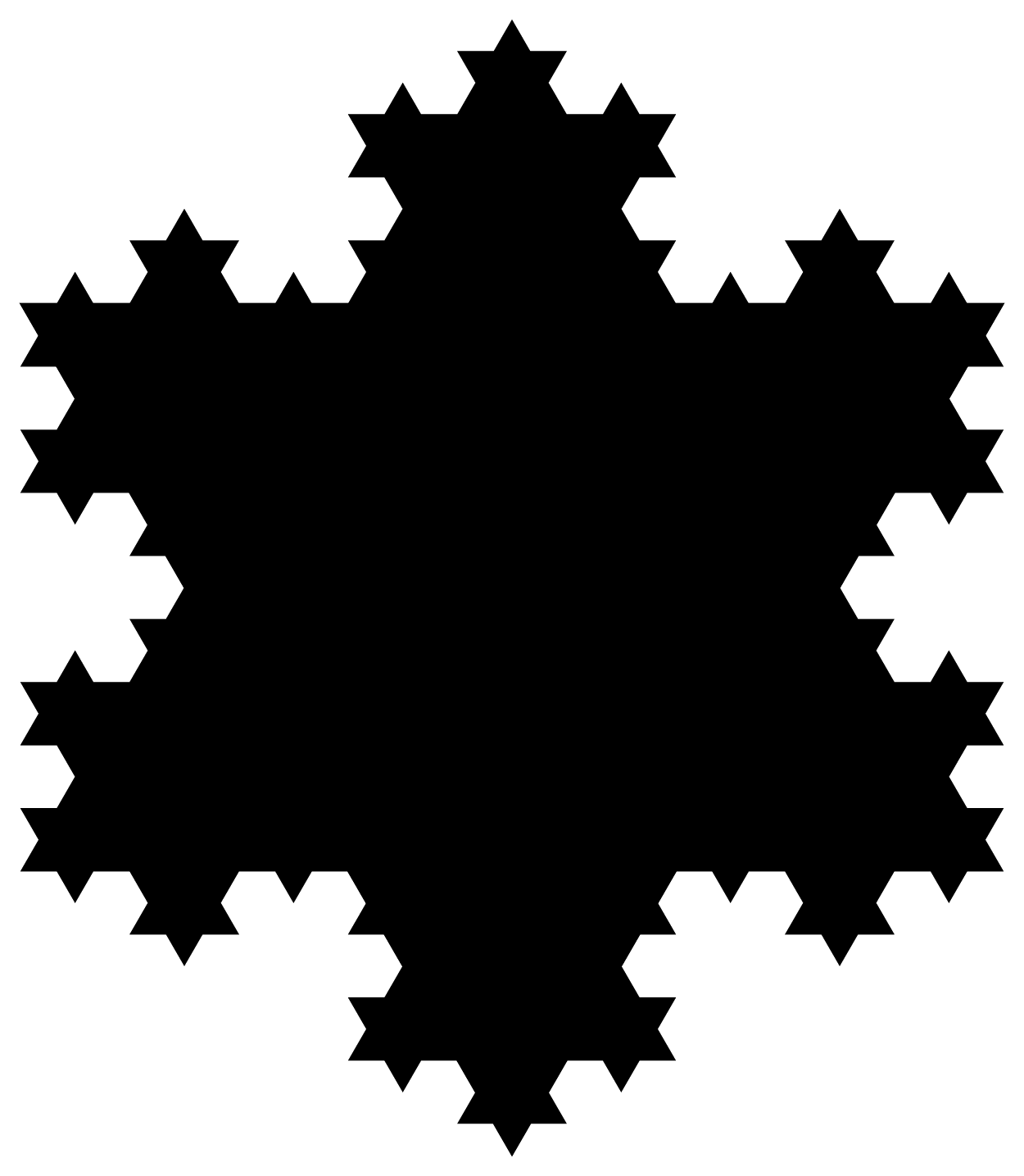}}
\centering
\caption{The first four prefractal approximations to the Koch snowflake.}
\label{fig:koch}
\end{figure}

\begin{example}[Scattering by a Koch snowflake] Let $n=3$ and let $\mS\subset \Gamma_\infty$ be the open set that is the interior of a Koch snowflake curve, defined by $\mS=\cup_{j=1}^\infty \mS_j$, where $\mS_j$ is the interior of the $j$th prefractal approximation to the Koch snowflake curve; see  \cite[Figure 0.2(b), Example 9.5]{Fal} and Figure \ref{fig:koch}. Then \cite[Example 9.5]{Fal} $\dimH(\partial \mS)=\log(4)/\log(3)>1$ and $m(\partial \mS)=0$, so that $\partial \mS$ is $1/2$-null but is not $-1/2$-null (Theorem \ref{thm:null}). Thus the formulations \SNcls and \Nsts are not well-posed in this case (Theorem \ref{thm:wpst}(b)). These formulations do have solutions, namely the solutions of the well-posed formulations $\sS\sN(V^+)$ and $\sN(V^+)$, respectively, for any $V^+$ satisfying \eqref{eq:selection}, in particular for $V^+=\tH^{1/2}(\mS)$ and $V^+=H^{1/2}_\omS$ (Theorem \ref{NeuEquivThmm} and Corollary \ref{NeuEquivThmm2}), but uniqueness does not hold.

An open problem in this case is whether or not $\tH^{1/2}(\mS)=H^{1/2}_\omS$. If this does hold then the formulations for sound-hard scattering, $\sS\sN(V^+)$ with $V^+$ satisfying \eqref{eq:selection}, collapse to a single formulation, but if this does not hold then there are infinitely many of these formulations (with cardinality $\mathfrak{c}$), with infinitely many distinct solutions (Theorem \ref{thm:s62}). In particular the solution to $\sS\sN(V^+)$ with $V^+= \tH^{1/2}(\mS)$ is the limiting geometry solution in the sense of Definition \ref{def:lgopen}, in particular is the limit as $j\to\infty$ of $u^s_j$, the solution to \SNw, or equivalently \SNcl, for the Lipschitz open set that is the $j$th prefractal $\mS_j$. The solution for $V^+=H^{1/2}_\omS$ is the limiting geometry solution for the closed von Koch snowflake $\omS$ in the sense of Definition \ref{def:lgclosed}, equivalently the solution to \SNws (Corollary \ref{cor:lgs}). It is an open question whether these solutions are the same: all we can currently say is that they are the same if $\tH^{1/2}(\mS)=H^{1/2}_\omS$, and that if $\tH^{1/2}(\mS)\neq H^{1/2}_\omS$ then these solutions are different for plane wave incidence for almost all incident wave directions (Theorem \ref{thm:s62}).

Similar remarks apply for sound-soft scattering. It is also an open problem whether or not $\tH^{-1/2}(\mS)=H^{-1/2}_\omS$. \SDcls and \Dsts are well-posed if and only if equality holds (Theorem \ref{thm:wpst}), and this also determines whether the formulations $\sS\sD(V^-)$ and $\sD(V^-)$ with $V^-$ satisfying \eqref{eq:selection} collapse to a single formulation, or whether there are infinitely many formulations with cardinality $\mathfrak{c}$ (Theorem \ref{thm:S61}). Similarly to the sound-hard case the solution to $\sS\sD(V^-)$ with $V^-= \tH^{-1/2}(\mS)$ is the limiting geometry solution in the sense of Definition \ref{def:lgopen}, in particular is the limit as $j\to\infty$ of $u^s_j$, the solution to \SDw, or equivalently \SDcl, for the Lipschitz open set that is the $j$th prefractal $\mS_j$. The solution for $V^-=H^{-1/2}_\omS$ is the limiting geometry solution for the screen $\omS$ in the sense of Definition \ref{def:lgclosed}, equivalently the solution to \SNws (Corollary \ref{cor:lgs}). We do not know whether these solutions are the same: this depends on whether or not $\tH^{-1/2}(\mS)=H^{-1/2}_\omS$ (Theorem \ref{thm:S61}).

\end{example}

\noindent
{\bf Acknowledgements.}
The authors are grateful to A.\ Moiola (Reading, UK) for many stimulating discussions in relation to this work.
\bibliography{BEMbib_short2014a}%
\bibliographystyle{siam}

\end{document}

%% file: macros.tex
\newcommand{\rf}[1]{(\ref{#1})}
\newcommand{\mmbox}[1]{\fbox{\ensuremath{\displaystyle{ #1 }}}}	
\newcommand{\hs}[1]{\hspace{#1mm}}
\newcommand{\vs}[1]{\vspace{#1mm}}
\newcommand{\ri}{{\mathrm{i}}}
\newcommand{\re}{{\mathrm{e}}}
\newcommand{\rd}{\mathrm{d}}
\newcommand{\R}{\mathbb{R}}
\newcommand{\Q}{\mathbb{Q}}
\newcommand{\N}{\mathbb{N}}
\newcommand{\Z}{\mathbb{Z}}
\newcommand{\C}{\mathbb{C}}
\newcommand{\K}{{\mathbb{K}}}
\newcommand{\cA}{\mathcal{A}}
\newcommand{\cB}{\mathcal{B}}
\newcommand{\cC}{\mathcal{C}}
\newcommand{\cS}{\mathcal{S}}
\newcommand{\cD}{\mathcal{D}}
\newcommand{\cH}{\mathcal{H}}
\newcommand{\cI}{\mathcal{I}}
\newcommand{\cItilde}{\tilde{\mathcal{I}}}
\newcommand{\cIhat}{\hat{\mathcal{I}}}
\newcommand{\cIcheck}{\check{\mathcal{I}}}
\newcommand{\cIstar}{{\mathcal{I}^*}}
\newcommand{\cJ}{\mathcal{J}}
\newcommand{\cM}{\mathcal{M}}
\newcommand{\cP}{\mathcal{P}}
\newcommand{\cV}{{\mathcal V}}
\newcommand{\cW}{{\mathcal W}}
\newcommand{\scrD}{\mathscr{D}}
\newcommand{\scrS}{\mathscr{S}}
\newcommand{\scrJ}{\mathscr{J}}
\newcommand{\sD}{\mathsf{D}}
\newcommand{\sN}{\mathsf{N}}
\newcommand{\sS}{\mathsf{S}}
 \newcommand{\sT}{\mathsf{T}}
 \newcommand{\sH}{\mathsf{H}}
 \newcommand{\sI}{\mathsf{I}}
\newcommand{\bs}[1]{\mathbf{#1}}
\newcommand{\bb}{\mathbf{b}}
\newcommand{\bd}{\mathbf{d}}
\newcommand{\bn}{\mathbf{n}}
\newcommand{\bp}{\mathbf{p}}
\newcommand{\bP}{\mathbf{P}}
\newcommand{\bv}{\mathbf{v}}
\newcommand{\bx}{\mathbf{x}}
\newcommand{\by}{\mathbf{y}}
\newcommand{\bz}{{\mathbf{z}}}
\newcommand{\bxi}{\boldsymbol{\xi}}
\newcommand{\boldeta}{\boldsymbol{\eta}}	
\newcommand{\ts}{\tilde{s}}
\newcommand{\tGamma}{{\tilde{\Gamma}}}
 \newcommand{\tbx}{\tilde{\bx}}
 \newcommand{\tbd}{\tilde{\bd}}
 \newcommand{\txi}{\xi}
\newcommand{\done}[2]{\dfrac{d {#1}}{d {#2}}}
\newcommand{\donet}[2]{\frac{d {#1}}{d {#2}}}
\newcommand{\pdone}[2]{\dfrac{\partial {#1}}{\partial {#2}}}
\newcommand{\pdonet}[2]{\frac{\partial {#1}}{\partial {#2}}}
\newcommand{\pdonetext}[2]{\partial {#1}/\partial {#2}}
\newcommand{\pdtwo}[2]{\dfrac{\partial^2 {#1}}{\partial {#2}^2}}
\newcommand{\pdtwot}[2]{\frac{\partial^2 {#1}}{\partial {#2}^2}}
\newcommand{\pdtwomix}[3]{\dfrac{\partial^2 {#1}}{\partial {#2}\partial {#3}}}
\newcommand{\pdtwomixt}[3]{\frac{\partial^2 {#1}}{\partial {#2}\partial {#3}}}
\newcommand{\bnabla}{\boldsymbol{\nabla}}
\newcommand{\dive}{\boldsymbol{\nabla}\cdot}
\newcommand{\curl}{\boldsymbol{\nabla}\times}
\newcommand{\Phixy}{\Phi(\bx,\by)}
\newcommand{\PhiOxy}{\Phi_0(\bx,\by)}
\newcommand{\dxPhixy}{\pdone{\Phi}{n(\bx)}(\bx,\by)}
\newcommand{\dyPhixy}{\pdone{\Phi}{n(\by)}(\bx,\by)}
\newcommand{\dxPhiOxy}{\pdone{\Phi_0}{n(\bx)}(\bx,\by)}
\newcommand{\dyPhiOxy}{\pdone{\Phi_0}{n(\by)}(\bx,\by)}
\newcommand{\eps}{\varepsilon}
\newcommand{\real}[1]{{\rm Re}\left[#1\right]} 
\newcommand{\im}[1]{{\rm Im}\left[#1\right]}
\newcommand{\ol}[1]{\overline{#1}}
\newcommand{\ord}[1]{\mathcal{O}\left(#1\right)}
\newcommand{\oord}[1]{o\left(#1\right)}
\newcommand{\Ord}[1]{\Theta\left(#1\right)}
\newcommand{\hsnorm}[1]{||#1||_{H^{s}(\bs{R})}}
\newcommand{\hnorm}[1]{||#1||_{\tilde{H}^{-1/2}((0,1))}}
\newcommand{\norm}[2]{\left\|#1\right\|_{#2}}
\newcommand{\normt}[2]{\|#1\|_{#2}}
\newcommand{\on}[1]{\Vert{#1} \Vert_{1}}
\newcommand{\tn}[1]{\Vert{#1} \Vert_{2}}
\newcommand{\xt}{\mathbf{x},t}
\newcommand{\PhiF}{\Phi_{\rm freq}}
\newcommand{\cone}{{c_{j}^\pm}}
\newcommand{\ctwo}{{c_{2,j}^\pm}}
\newcommand{\cthree}{{c_{3,j}^\pm}}
\newtheorem{thm}{Theorem}[section]
\newtheorem{lem}[thm]{Lemma}
\newtheorem{defn}[thm]{Definition}
\newtheorem{prop}[thm]{Proposition}
\newtheorem{cor}[thm]{Corollary}
\newtheorem{rem}[thm]{Remark}
\newtheorem{conj}[thm]{Conjecture}
\newtheorem{ass}[thm]{Assumption}
\newtheorem{example}[thm]{Example} 
\newcommand{\tH}{\widetilde{H}}
\newcommand{\Hze}{H_{\rm ze}} 	
\newcommand{\uze}{u_{\rm ze}}		
\newcommand{\dimH}{{\rm dim_H}}
\newcommand{\dimB}{{\rm dim_B}}
\newcommand{\IntClosOm}{\mathrm{int}(\overline{\Omega})}
\newcommand{\IntClosOmOne}{\mathrm{int}(\overline{\Omega_1})}
\newcommand{\IntClosOmTwo}{\mathrm{int}(\overline{\Omega_2})}
\newcommand{\Ccomp}{C^{\rm comp}}
\newcommand{\tCcomp}{\tilde{C}^{\rm comp}}
\newcommand{\uC}{\underline{C}}
\newcommand{\utC}{\underline{\tilde{C}}}
\newcommand{\oC}{\overline{C}}
\newcommand{\otC}{\overline{\tilde{C}}}
\newcommand{\capcomp}{{\rm cap}^{\rm comp}}
\newcommand{\Capcomp}{{\rm Cap}^{\rm comp}}
\newcommand{\tcapcomp}{\widetilde{{\rm cap}}^{\rm comp}}
\newcommand{\tCapcomp}{\widetilde{{\rm Cap}}^{\rm comp}}
\newcommand{\hcapcomp}{\widehat{{\rm cap}}^{\rm comp}}
\newcommand{\hCapcomp}{\widehat{{\rm Cap}}^{\rm comp}}
\newcommand{\tcap}{\widetilde{{\rm cap}}}
\newcommand{\tCap}{\widetilde{{\rm Cap}}}
\newcommand{\ccap}{{\rm cap}}
\newcommand{\ucap}{\underline{\rm cap}}
\newcommand{\uCap}{\underline{\rm Cap}}
\newcommand{\cCap}{{\rm Cap}}
\newcommand{\ocap}{\overline{\rm cap}}
\newcommand{\oCap}{\overline{\rm Cap}}
\DeclareRobustCommand
{\mathringbig}[1]{\accentset{\smash{\raisebox{-0.1ex}{$\scriptstyle\circ$}}}{#1}\rule{0pt}{2.3ex}}
\newcommand{\cirH}{\mathringbig{H}}
\newcommand{\cirHs}{\mathringbig{H}{}^s}
\newcommand{\cirHt}{\mathringbig{H}{}^t}
\newcommand{\cirHm}{\mathringbig{H}{}^m}
\newcommand{\cirHzero}{\mathringbig{H}{}^0}
\newcommand{\deO}{{\partial\Omega}}
\newcommand{\OO}{{(\Omega)}}
\newcommand{\Rn}{{(\R^n)}}
\newcommand{\Id}{{\mathrm{Id}}}
\newcommand{\gap}{\mathrm{Gap}}
\newcommand{\ggap}{\mathrm{gap}}
\newcommand{\isom}{{\xrightarrow{\sim}}}
\newcommand{\half}{{1/2}}
\newcommand{\mhalf}{{-1/2}}
\newcommand{\inter}{{\mathrm{int}}}

\newcommand{\Hsp}{H^{s,p}}
\newcommand{\Htq}{H^{t,q}}
\newcommand{\tHsp}{{{\widetilde H}^{s,p}}}
\newcommand{\SP}{\ensuremath{(s,p)}}
\newcommand{\Xsp}{X^{s,p}}

\newcommand{\dd}{{d}}\newcommand{\pp}{{p_*}}

\newcommand{\Rnn}{\R^{n_1+n_2}}
\newcommand{\Tr}{{\mathrm{Tr}}}